\documentclass{amsart}
\usepackage{amsmath}
\usepackage{amsthm}
\usepackage{amssymb}

\usepackage[utf8]{inputenc}
\usepackage[T1]{fontenc}
\usepackage{lmodern}
\usepackage{eucal}
\pdfoutput=1 %Needed for microtype and ArXiV.
\usepackage{microtype}
\usepackage{url}
\usepackage{enumitem}
\usepackage{tikz-cd}

\usepackage[letterpaper, left = 1.5in, right = 1.5in, top = 1.5in, bottom = 1.5in]{geometry}

\swapnumbers

\usepackage{hyperref}

\usepackage{xcolor}
\hypersetup{
    colorlinks,
    linkcolor={red!50!black},
    citecolor={blue!50!black},
    urlcolor={blue!80!black}
}

\usepackage[capitalise,nameinlink]{cleveref}
\crefname{subsection}{Subsection}{Subsections}
\crefname{subsubsection}{Subsubsection}{Subsubsections}

\theoremstyle{definition}
\newtheorem{theorem}{Theorem}[subsection]
\newtheorem{defn}[theorem]{Definition}
\newtheorem{ex}[theorem]{Example}
\newtheorem{cor}[theorem]{Corollary}
\newtheorem{lemma}[theorem]{Lemma}
\newtheorem{prop}[theorem]{Proposition}
\newtheorem{rmk}[theorem]{Remark}
\newtheorem{warning}[theorem]{Warning}
\newtheorem*{rmk*}{Remark}
\newtheorem*{ex*}{Example}
\newtheorem*{theorem*}{Theorem}
\newtheorem*{defn*}{Definition}

\newcommand{\bbZ}{\mathbb{Z}}

\newcommand{\bbF}{\mathbb{F}}
\newcommand{\bbN}{\mathbb{N}}

\newcommand{\bbE}{\mathbb{E}}

\newcommand{\bbA}{\mathbb{A}}
\newcommand{\bbT}{\mathbb{T}}

\newcommand{\bbL}{\mathbb{L}}

\newcommand{\calM}{\mathcal{M}}

\newcommand{\calO}{\mathcal{O}}
\newcommand{\calC}{\mathcal{C}}
\newcommand{\calD}{\mathcal{D}}

\newcommand{\calJ}{\mathcal{J}}

\newcommand{\calK}{\mathcal{K}}

\newcommand{\calA}{\mathcal{A}}
\newcommand{\calP}{\mathcal{P}}

\newcommand{\calX}{\mathcal{X}}

\newcommand{\calR}{\mathcal{R}}

\newcommand{\frakm}{\mathfrak{m}}

\newcommand{\Sp}{\mathcal{S}\mathrm{p}}
\newcommand{\Mod}{\mathcal{M}\mathrm{od}}
\newcommand{\Alg}{\mathcal{A}\mathrm{lg}}
\newcommand{\Gpd}{\mathcal{G}\mathrm{pd}}
\newcommand{\CAlg}{\mathcal{C}\mathrm{Alg}}
\newcommand{\Ab}{\mathcal{A}\mathrm{b}}
\newcommand{\Fun}{\mathrm{Fun}}
\newcommand{\s}{\mathrm{s}}
\newcommand{\Set}{\mathcal{S}\mathrm{et}}
\newcommand{\Mon}{\mathcal{M}\mathrm{on}}
\newcommand{\LMod}{\mathrm{L}\mathcal{M}\mathrm{od}}
\newcommand{\Fin}{\mathcal{F}\mathrm{in}}
\newcommand{\Hrd}{\mathcal{H}\mathrm{rd}}
\newcommand{\Model}{\mathcal{M}\mathrm{odel}}
\newcommand{\Ch}{\mathrm{Ch}}
\newcommand{\Sph}{\mathcal{S}\mathrm{ph}}

\newcommand{\Map}{\operatorname{Map}}
\newcommand{\Hom}{\operatorname{Hom}}
\newcommand{\Aut}{\operatorname{Aut}}
\newcommand{\Iso}{\operatorname{Iso}}
\newcommand{\Fib}{\operatorname{Fib}}
\newcommand{\Tor}{\operatorname{Tor}}
\newcommand{\Pro}{\operatorname{Pro}}
\newcommand{\Psh}{\operatorname{Psh}}
\newcommand{\Ext}{\operatorname{Ext}}
\newcommand{\Cof}{\operatorname{Cof}}

\newcommand{\colim}{\operatorname*{colim}}
\newcommand{\hocolim}{\operatorname*{hocolim}} 
\newcommand{\tCof}{\operatorname{tCof}}
\newcommand{\im}{\operatorname{Im}}
\renewcommand{\ker}{\operatorname{Ker}}
\newcommand{\olotimes}{\mathbin{\ol{\otimes}}}
\newcommand{\EM}{\mathcal{EM}}
\newcommand{\AB}{\mathcal{AB}}
\newcommand{\GRP}{\mathcal{GP}}
\newcommand{\Eq}{\mathrm{Eq}}
\newcommand{\EXT}{\mathcal{E}\mathrm{xt}}

\newcommand{\loc}{\mathrm{loc}}
\newcommand{\heart}{\heartsuit}
\newcommand{\h}{\mathrm{h}}
\newcommand{\op}{\mathrm{op}}
\newcommand{\cn}{\mathrm{cn}}
\newcommand{\free}{\mathrm{free}}
\newcommand{\sfg}{\mathrm{sfg}}
\newcommand{\Cpl}{\mathrm{Cpl}}

\newcommand{\bs}{{-}}
\newcommand{\ul}{\underline}
\newcommand{\ol}{\overline}
\def\hyp{{\hbox{-}}}

\newcommand\xqed[1]{%
  \leavevmode\unskip\penalty9999 \hbox{}\nobreak\hfill
  \quad\hbox{#1}}
\newcommand\tqed{\xqed{$\triangleleft$}}

\makeatletter
\DeclareRobustCommand{\tvdots}{%
  \vbox{\baselineskip4\p@\lineskiplimit\z@\kern0\p@\hbox{.}\hbox{.}\hbox{.}}}
\makeatother

\makeatletter
\@namedef{subjclassname@2020}{\textup{2020} Mathematics Subject Classification}
\makeatother

\begin{document}

\title[Deformations of homotopy theories via algebraic theories]{Deformations of homotopy theories \\ via algebraic theories}

\author{William Balderrama}


\begin{abstract}
We develop a homotopical variant of the classic notion of an algebraic theory as a tool for producing deformations of homotopy theories. From this, we extract a framework for constructing and reasoning with obstruction theories and spectral sequences that compute homotopical data starting with purely algebraic data. 
\end{abstract}

\maketitle

\section{Introduction}

Consider an example from spectral algebra. If $R$ is a $\bbE_\infty$-ring, and $M$ and $N$ are $R$-modules, then there are universal coefficient and K\"unneth spectral sequences
\begin{align*}
&\Ext^{p+q}_{R_\ast}(M_\ast,N_{\ast+p})\Rightarrow \pi_{\ast-q}\Mod_R(M,N),\\
&\Tor^{R_\ast}_{p+q}(M_\ast,N_{\ast-p})\Rightarrow \pi_{\ast+q}M\otimes_R N.
\end{align*}
We can view the existence of these spectral sequences as saying that the homotopy theory of $R$-modules is, in some sense, approximated by the homological algebra of $R_\ast$-modules.

The purpose of this paper is to describe a certain $\infty$-categorical framework that captures this idea. Central to our approach is the notion of an \textit{algebraic theory}. Building on insights of Hopkins-Lurie \cite{hopkinslurie2017brauer} and Pstrągowski \cite{pstragowski2023moduli}, in turn conceptualizing ideas of Dwyer-Kan-Stover \cite{dwyerkanstover1995bigraded} and Blanc-Dwyer-Goerss \cite{blancdwyergoerss2004realization}, we introduce a certain homotopical refinement of the classic notion of an algebraic theory, which we call \textit{loop theories}. If $\calP$ is a loop theory, then in addition to its category $\Model_\calP$ of ($\infty$-groupoid valued) models,

\begin{enumerate}
\item There is a category $\Model_\calP^\Omega$ of models of the theory $\calP$ that respect the additional homotopical structure present;
\item There is a category $\Model_{\h\calP}^\heart$ of set-valued models of the homotopy category $\h\calP$, with nonabelian derived category $\Model_{\h\calP}$ and a natural ``homotopy groups'' functor $\Model_\calP^\Omega\rightarrow\Model_{\h\calP}^\heart$.
\end{enumerate}

For example, when $\calP = \Mod_R^\free$, the category $\Model_\calP^\Omega$ recovers the category of $R$-modules, and $\Model_{\h\calP}^\heart$ is equivalent to the ordinary category of $R_\ast$-modules; when $\calP = \CAlg_{H\bbF_p}^\free$, the category $\Model_\calP^\Omega$ recovers the category of $\bbE_\infty$ algebras over $\bbF_p$, and $\Model_{\h\calP}^\heart$ is equivalent to a category of $\bbF_{p\ast}$-rings with Dyer-Lashof operations. In general, $\calM\simeq\Model_\calP^\Omega$ whenever $\calM$ is a homotopy theory with a good subcategory $\calP\subset\calM$ of free objects, and $\Model_{\h\calP}^\heart$ is then an algebraic approximation to $\calM$. However the key point is the third:

\begin{enumerate}[resume]
\item The category $\Model_\calP$ fits into a span
\begin{center}\begin{tikzcd}
\Model_{\h\calP}&\ar[l] \Model_\calP\ar[r]&\Model_\calP^\Omega
\end{tikzcd},\end{center}
and behaves as a \textit{deformation} with generic fiber $\Model_\calP^\Omega$ and special fiber $\Model_{\h\calP}$.
\end{enumerate}

Here, $\Model_{\h\calP}$ is equivalent to Quillen's homotopy theory of simplicial set-valued models of the algebraic theory $\h\calP$. The deformation theory available in this context then gives concrete computational tools, in the form of obstruction-theoretic machinery, for understanding the homotopy theory of $\Model_\calP^\Omega$ starting with the algebra of $\Model_{\h\calP}$. For example, there is an obstruction theory computing maps in $\Model_\calP^\Omega$ with obstruction groups given in terms of the Quillen cohomology of objects of $\Model_{\h\calP}^\heart$.

In addition to recovering a number of classical tools, such as the universal coefficient and K\"unneth spectral sequences already mentioned, we view as one of the primary benefits of our framework the ease in which it is adapted to new situations. Our original motivation for this was to develop some tools allowing one to compute with $K(n)$-local $\bbE_\infty$ algebras over Lubin-Tate spectra, or Morava $E$-theories, using their theories of power operations. There are a number of subtleties in this context; beyond being unstable and unpointed, the $K(n)$-local condition produces categories which are inherently \textit{infinitary}. We incorporate all of this by allowing our theories to themselves be infinitary, following a similar approach to that taken by Lurie in \cite[Section 4.2]{lurie2011quasi}, which has also been applied to Lubin-Tate spectra by Brantner \cite{brantner2017lubin}. This approach allows us to incorporate such exotic settings into one uniform theory. Particular applications appear in detail in \cite{balderrama2021algebraic}.

After laying out some general categorical conventions, we shall give an extended overview of the paper in \cref{ssec:outline} below, omitting most of the technicalities. In adddition, we give examples in \cref{ssec:examples}; the reader may prefer to browse this subsection first to get a feeling for the contexts in which our machinery may be applied.

\subsection*{Acknowledgements}

This paper was originally the first half of the author's thesis at the University of Illinois Urbana-Champaign, and has benefited from conversations with many people during this period; in particular, special thanks are due to Brian Shin and Chieu-Minh Tran for comments on an earlier draft. I would especially like to thank Charles Rezk for his support over the past several years, without which this work would not have been possible.

\subsection{Conventions}

We will freely use the theory of $\infty$-categories, which we refer to just as categories, as developed by Lurie in \cite{lurie2017highertopos} and \cite{lurie2017higheralgebra}, and by default all of our constructions should be interpreted in this sense. We write $\Gpd_\infty$ for the category of $\infty$-groupoids, also commonly known as the ($\infty$-)category of spaces, and for a small category $\calC$ we write $\Psh(\calC)$ for the category of presheaves of $\infty$-groupoids on $\calC$, writing instead $\Psh(\calC,\Set)$ when we mean presheaves of sets, and similarly for presheaves valued in other categories.

We follow the standard convention of fixing a universe of small $\infty$-groupoids, with respect to which everything in sight will be at least locally small, contained in a universe of large $\infty$-groupoids, with respect to which everything in sight is small, unless otherwise specified. For a (locally small) category $\calC$, we write $\Psh(\calC)$ for the category of presheaves on $\calC$ that arise as small colimits of representable presheaves; this is the cocompletion of $\calC$ under small colimits. We write $h\colon\calC\rightarrow\Psh(\calC)$ for the Yoneda embedding, and write the same for various restricted Yoneda embeddings.

Given a functor $f\colon\calC\rightarrow\calD$, we write $f_!\colon\Psh(\calC)\rightarrow\Psh(\calD)$ for the left adjoint to the restriction $f^\ast\colon \Psh(\calD)\rightarrow\Psh(\calC)$. This is the left Kan extension of $h\circ f\colon \calC\rightarrow\calD\rightarrow\Psh(\calD)$ along $h\colon \calC\rightarrow\Psh(\calC)$, and we shall also write $f_!\colon \Psh(\calC)\rightarrow\calD$ for the functor left Kan extended from $f$ itself when $\calD$ admits sufficiently many colimits, and for related functors.

For a category $\calC$, we write $\h\calC$ for the homotopy category of $\calC$, and $\calC^\simeq$ for the maximal sub-$\infty$-groupoid of $\calC$.

\subsection{Outline}\label{ssec:outline}

\subsubsection{Algebraic theories}(\cref{sec:malc}).

The basis of our work is a variant of the classic notion of a \textit{Lawvere theory}, which is a categorical approach to universal algebra pioneered by Lawvere in this thesis \cite{lawvere1963functorial}. Some familiarity with this story is useful for understanding the rest of the paper, so we give a very brief review here. Classically, a Lawvere theory may be defined as a category $\calC$ with object set $\bbN$ wherein $n$ is the $n$-fold coproduct of $1$ for all $n\in\bbN$. To be precise, the classical case would require $\calC$ to be a $1$-category, but we shall not require this. The category of \textit{models} of $\calC$ is then the category of presheaves $X$ on $\calC$ such that the canonical map $X(n)\rightarrow X(1)^{\times n}$ is an isomorphism for all $n\in\bbN$.

By only asking that $\calC$ has finite coproducts, and not that a specified object generates $\calC$ under coproducts, one is led to the notion of a \textit{(multisorted) algebraic theory}. Note that no particular sorts are specified in this definition; this approach emphasizes the aspects of algebraic theories which are invariant under Morita equivalence, i.e.\ emphasizes their categories of models as the primary objects of interest. Here, the category $\Model_\calC$ of models of $\calC$ is the category of presheaves $X$ on $\calC$ such that $X(\coprod_{i\in F}C_i)\simeq\prod_{i\in F}X(C_i)$ for any finite collection $\{C_i:i\in F\}$ of objects in $\calC$.

By restricting $\calC$ to be a \textit{discrete} algebraic theory, that is, a $1$-category, and considering the full category $\Model_\calC^\heart$ of discrete, or set-valued, models, one recovers from this a large number of naturally occuring algebraic categories. Taking $\calC$ to still be a discrete algebraic theory, the category $\Model_\calC$ of $\infty$-groupoid-valued models is a familiar homotopy theory: it is the underlying $\infty$-category of the category of simplicial set-valued models of $\calC$ equipped with model structure constructed by Quillen \cite[Section II.4]{quillen1967homotopical}, as can be seen starting with work of Badzioch \cite{badzioch2002algebraic}, generalized by Bergner \cite{bergner2006rigidification}, and put into the $\infty$-categorical context by Lurie \cite[Section 5.5.9]{lurie2017highertopos}.

One can view the categories arising in this manner as exactly the categories of models of multisorted finite product theories, and this is useful for understanding various examples. For example, if $\calC$ is the category of finitely generated and free abelian groups, then $\calC$ is an algebraic theory, and the category of models of $\calC$ is equivalent to the category of abelian groups; roughly, this is because an abelian group $M$ is determined by its addition map, which may be recovered as the map $\Delta^\ast\colon\Ab(\bbZ,M)\times\Ab(\bbZ,M)\cong\Ab(\bbZ\oplus\bbZ,M)\rightarrow\Ab(\bbZ,M)$ given by restriction along the diagonal $\Delta\colon\bbZ\rightarrow\bbZ\oplus\bbZ$. For our purposes, it is more useful to view these categories as those which admit a family of compact projective generators; from this perspective, the category $\Model_\calC$ is best characterized as the free cocompletion of $\calC$ under filtered colimits and geometric realizations, see \cite[Section 5.5.8]{lurie2017highertopos}, or \cite{adamekrosickyvitale2011algebraic} for a textbook account of the $1$-categorical case.

We are interested in certain categories which admit families of projective, but not necessarily compact, generators. To incorporate these, we allow our theories to be \textit{infinitary}. The classic reference for infinitary theories is Wraith \cite{wraith1969algebraic}, although certain size issues are overlooked there. To deal with these, we restrict ourselves to \textit{bounded} theories, i.e.\ those which are generated by $\kappa$-ary operations for some regular cardinal $\kappa$. This has the further benefit of making available to us all the tools from the theory of presentable categories. 

In general, infinitary theories are not as well-behaved as finitary theories. In order to obtain a story mimicking the finitary case, we must restrict ourselves to those theories which are \textit{Mal'cev}; see for instance \cite{smith1976malcev} and \cite{lambek1992ubiquity}, though we require very little of the general theory. This assumption turns out to play two roles: not only is it used to ensure that $\Model_\calP$ has good properties when $\calP$ is an infinitary theory, it is also necessary for the development of the further homotopical refinement of loop theories, which we introduce further below.

Let us proceed to the precise definition. A \textit{Mal'cev} operation on a set $H$ is a ternary operation $t\colon H\times H\times H\rightarrow H$ satisfying $t(x,x,y) = y$ and $t(x,y,y) = x$; a set equipped with a Mal'cev operation is called a \textit{herd}. This term is due to Lambek \cite{lambek1955groups}; we do not impose the associativity condition $t(t(v,w,x),y,z) = t(v,w,t(x,y,z))$. The motivating example is $H = \Iso(X,Y)$ for two objects $X$ and $Y$ of some $1$-category, with Mal'cev operation $t(f,g,h) = fg^{-1}h$, and all of our examples are ultimately derived from this. Being equationally defined objects, herds are modeled by a Lawvere theory, so it makes sense to speak of herd objects in arbitrary categories with finite products.

\begin{defn*}[\cref{def:malc}]
A \textit{Mal'cev theory} is a category $\calP$ such that
\begin{enumerate}
\item $\calP$ admits small coproducts;
\item All objects of $\calP$ admit the structure of a coherd (i.e.\ of a herd object in $\calP^\op$).
\end{enumerate}
The category of \textit{models} of $\calP$ is the category $\Model_\calP$ of presheaves $X$ on $\calP$ such that $X(\coprod_{i\in I}P_i)\simeq \prod_{i\in I}X(P_i)$ for any set $\{P_i:i\in I\}$ of objects in $\calP$, and $\Model_\calP^\heart\subset\Model_\calP$ is the full subcategory of discrete, or set-valued, models.
\tqed
\end{defn*}

This definition may be somewhat opaque. A similar notion was studied in \cite[Section 4.2]{lurie2011quasi}, with groups in place of herds; from this perspective, herds arise as an unpointed generalization of groups. When $\calP$ is a discrete theory, there is a much more elegant formulation: a discrete theory $\calP$ is Mal'cev precisely when every simplicial set-valued model of $\calP$ takes values in Kan complexes. This is exactly the condition Quillen requires in \cite[II.4]{quillen1967homotopical} to produce homotopy theories of simplicial objects in non-compactly generated settings. We expect there could be more elegant or more general formulations of the Mal'cev condition for $\infty$-categorical theories. For this reason we gather the facts which rely on the Mal'cev assumption in one place in \cref{ssec:malc}, after which it no longer appears explicitly, and everything we do holds equally well for any theory satisfying properties of the sort laid out there.

We will only be concerned with Mal'cev theories, and so will refer to them simply as \textit{theories}. If $\calC$ is a finitary theory and $\calP\subset\Model_\calC$ is generated by $\calC$ under coproducts, then $\Model_\calP\simeq\Model_\calC$ (\cref{prop:modelspresentable}); thus infinitary theories do indeed generalize finitary theories. Throughout the paper, we will make some minor size assumptions, assuming that our theories are generated in a similar way by a small, but not necessary countable, amount of data (\cref{rmk:bounded}).

\begin{rmk}
All of the discrete theories we will encounter arise as a combination of the following facts:
\begin{enumerate}
\item If $\calA$ is a cocomplete abelian category and $\calP\subset\calA$ is a full subcategory consisting of projective objects and closed under coproducts such that every $M\in\calA$ admits a projective resolution by objects of $\calP$, then $\calA\simeq\Model_\calP^\heart$ (\cref{prop:additiveabelian});
\item If $\calP$ is a discrete theory, $T$ is a monad on $\Model_\calP^\heart$ preserving reflexive coequalizers, and $T\calP\subset\Alg_T$ is the full subcategory spanned by the image of $\calP$ under $T$, then $T\calP$ is a theory and $\Alg_T\simeq\Model_{T\calP}^\heart$ (\cref{prop:monadtheory});
\item If $\calP$ is a discrete theory and $X\in\Model_\calP^\heart$, then the slice category $\calP/X$ is a theory and $\Model_{\calP/X}^\heart\simeq\Model_\calP^\heart/X$. It is for examples like this that we have used herds, rather than groups, in the definition of the theories we work with: even if the objects of $\calP$ admit the structure of a cogroup, this will generally fail for the objects of $\calP/X$, as it is common to have $f,g\in \calP/X$ with $\Map_{\calP/X}(f,g) = \emptyset$.
\tqed
\end{enumerate}
\end{rmk}

In \cref{sec:malc}, we show that (Mal'cev) theories and their categories of models indeed share all the good properties of finitary theories. Most importantly, in \cref{ssec:malc} we show that $\Model_\calP$ is the free cocompletion of $\calP$ under geometric realizations. Moreover, $\Model_\calP$ is presentable under our mild size conditions on $\calP$. In \cref{ssec:simplicial}, we verify that if $\calP$ is a discrete theory, then $\Model_\calP$ is the underlying $\infty$-category of Quillen's model category of simplicial set-valued models of $\calP$, and review the notion of left-derived functor available in this context.
\subsubsection{Loop theories}(\cref{sec:looptheories}).

Many of the categories we care most about are not of the form $\Model_\calP$. For example, no nontrivial stable category is of this form. Heuristically, theories are \textit{product} theories, and so encode operations with arities indexed by \textit{discrete} sets, whereas these categories require operations indexed over higher-dimensional objects, such as spheres. This leads to the following definition.

\begin{defn*}[\cref{def:looptheories}]
A theory $\calP$ is a \textit{loop theory} if for any finite wedge of spheres $F$ and $P\in\calP$, the tensor $F\otimes P = \colim_{x\in F} P$ exists in $\calP$. If $\calP$ is a loop theory, then the category $\Model_\calP^\Omega\subset\Model_\calP$ of \textit{loop models} is the full subcategory of models $X$ such that $X(F\otimes P)\simeq X(P)^{F}$ for all $P\in\calP$ and finite wedge of spheres $F$.
\tqed
\end{defn*}

We can now describe the general philosophy of this paper. For a great many categories $\calM$ that arise in homotopy theory, one can find (possibly multiple) full subcategories $\calP\subset\calM$ which are loop theories such that $\calM\simeq\Model_\calP^\Omega$. Upon choosing such a $\calP$, one may then embed $\calM$ into the larger category $\Model_\calP$, and this category provides a bridge between $\calM$ and the essentially algebraic category $\Model_{\h\calP}$. Conceptually, we may view $\Model_\calP$ as a deformation with generic fiber $\calM$ and special fiber $\Model_{\h\calP}$; in particular, this category gives access to new filtrations on constructions in $\calM$, with filtration quotients computed in $\Model_{\h\calP}$. By studying these filtrations, one gains access to various obstruction theories and spectral sequences by which one may approach the homotopy theory of $\calM$ starting with the algebra of $\Model_{\h\calP}$.

In the finitary and pointed case, where the objects of $\calP$ are instead asked to be homotopy cogroups, this context was first studied in general by Pstrągowski \cite{pstragowski2023moduli}, where it was used to give a conceptual approach to the realization problem for $\Pi$-algebras of Blanc-Dwyer-Goerss \cite{blancdwyergoerss2004realization}. The first instance we are aware of where a particular case of this context appears is in work of Hopkins-Lurie \cite{hopkinslurie2017brauer}. We add to this story by using a different class of theories, carrying out new constructions, and giving tools for recognizing additional examples.

In \cref{sec:looptheories}, we develop the basic properties of loop theories. Most important for the general theory is Pstrągowski's interpretation of the spiral spectral sequence, which holds equally well in our setting (\cref{thm:spiral}). This is the primary tool that allows us to identify various constructions in terms of the algebraic category $\Model_{\h\calP}$.

In \cref{ssec:constructingexamples}, we give tools for constructing and identifying examples. Heuristically, $\calM\simeq\Model_\calP^\Omega$ whenever $\calP\subset\calM$ is a reasonable collection of free objects closed under coproducts and $S^1$-tensors. Here, ``free'' does not have a precise meaning; $\calP$ must be a loop theory, but otherwise there is no particular freeness condition imposed on its objects. For example, if $\calM$ is a compactly generated stable category, then $\calM\simeq\Model_\calP^\Omega$ for any full subcategory $\calP\subset\calM$ which contains a family of compact generators and is a loop theory closed under coproducts, suspensions, and desuspensions in $\calM$ (\cref{thm:stabletheory}). In particular, there can be very different choices of $\calP\subset\calM$, producing wildly different algebraic categories $\Model_{\h\calP}$, for which one has $\calM\simeq\Model_\calP^\Omega$; we give an example below in \cref{sssec:bock}.

\begin{rmk}\label{rmk:pi0}
Before continuing on, we pause to point out the following notational subtlety.  By way of example, let $R$ be an $\bbA_\infty$-ring, and let $\calP = \LMod_R^\free$. In this case it turns out
\[
\Model_{\calP}^\Omega\simeq\LMod_R,\qquad\Model_{\calP}^\heart\simeq\LMod_{R_\ast}^\heart,
\]
and the following diagram commutes:
\begin{center}\begin{tikzcd}
\LMod_R\ar[r,"\pi_\ast"]\ar[d,"h"]&\LMod_{R_\ast}^\heart\ar[d,"h","\simeq"']\\
\Model_{\calP}\ar[r,"\pi_0"]&\Model_{\calP}^\heart
\end{tikzcd}.\end{center}
The notational subtlety is that
\[
\pi_\ast M = \pi_0 h(M).
\]
In general, if $\calP$ is a loop theory and $X\in\Model_\calP^\Omega$, then the $\h\calP$-model $\pi_0 X$ encodes all of the homotopy groups of $X$.
\tqed
\end{rmk}

\subsubsection{Stable loop theories}(\cref{sec:stable}).

\cref{sec:stable} and \cref{sec:postnikov} are independent of each other, with the exception that \cref{ssec:spiralss} connects with the material in \cref{sec:stable}. These sections both use the categories $\Model_\calP$ to produce computational tools for $\Model_\calP^\Omega$. We begin by describing the contents of \cref{sec:stable}.

\begin{defn*}
A loop theory $\calP$ is said to be \textit{stable} if 
\begin{enumerate}
\item $\calP$ is pointed and admits suspensions;
\item $\Sigma\colon\calP\rightarrow\calP$ is an equivalence.
\tqed
\end{enumerate}
\end{defn*}

Fix a stable loop theory $\calP$. Then $\Model_\calP$ is additive and thus embeds into its stabilization, which can be identified as the category of $\Sp$-valued models of $\calP$, i.e.\ of product-preserving presheaves of spectra on $\calP$. We will write $\LMod_\calP$ for this category, and moreover write $\LMod_\calP^\cn$ for the category of $\Sp_{\geq 0}$-valued models of $\calP$ and $\LMod_\calP^\heart$ for the category of abelian group-valued models of $\calP$. This notation is chosen in analogy with ordinary spectral algebra: $\LMod_\calP$ behaves very much like the category of modules over a connective ring spectrum, and $\LMod_\calP^\cn$ and $\LMod_\calP^\heart$ like the corresponding categories of connective and discrete modules. The forgetful functors allow us to identify $\LMod_\calP^\cn\simeq\Model_\calP$ and $\LMod_\calP^\heart\simeq\Model_\calP^\heart$.

There is in addition a category $\LMod_\calP^\Omega$ of $\Sp$-valued models $X$ such that $X(F\otimes P)\simeq X(P)^F$ for $P\in\calP$ and $F$ a finite wedge of spheres; it is equivalent to ask just that $X(\Sigma P)\simeq\Omega X(P)$ for $P\in\calP$. Moreover, the functor $\Omega^\infty\colon\LMod_\calP\rightarrow\Model_\calP$ restricts to an equivalence of categories $\Omega^\infty\colon\LMod_\calP^\Omega\simeq\Model_\calP^\Omega$. However, there is an important subtlety in this: $\LMod_\calP^\Omega\subset\LMod_\calP$ and $\Model_\calP^\Omega\subset\Model_\calP\simeq\LMod_\calP^\cn\subset\LMod_\calP$ are very different subcategories of $\LMod_\calP\subset\Psh(\calP,\Sp$, as the former generally takes values in nonconnective spectra. 

For $X,Y\in\LMod_\calP$, there is a mapping spectrum $\EXT_\calP(X,Y)$, so that $\Omega^\infty\EXT_\calP(X,Y) = \Map_{\calP}(X,Y)$. Despite the above subtlety, if $N\in\LMod_\calP^\Omega$, $M\in\Model_\calP^\Omega$, and $LM\in\Model_\calP$ is the preimage of $M$ under the equivalence $\LMod_\calP^\Omega\simeq\Model_\calP^\Omega$, then $\EXT_\calP(M,N)\simeq\EXT_\calP(LM,N)$. In \cref{ssec:stabletowers}, we describe the decomposition of these mapping spectra that may be obtained from Postnikov towers in $\LMod_\calP$. This takes the form of the following universal coefficient spectral sequence.

For $X\in \LMod_\calP$ and $n\in\bbZ$, write $\pi_nX = \pi_n \circ X \colon \calP^\op\rightarrow\Sp\rightarrow\Ab$. As $\pi_n$ preserves products and $\Ab$ is a $1$-category, this naturally lives in the abelian category $\Model_{\calP}^\heart\simeq\Model_{\h\calP}^\heart$. For $X\in\LMod_{\h\calP}$ and $n\in\bbZ$, write $X[n]$ for the model obtained by restricting $X$ along the functor $\h\calP\rightarrow\h\calP$ induced by $\Sigma^n\colon \calP\rightarrow\calP$, so $\pi_n X \cong (\pi_0 X)[n]$ if $X\in\LMod_\calP^\Omega$.

\begin{theorem*}[\cref{thm:stablemappingss}]
Fix $M\in \Model_\calP^{\Omega}$ and $N\in\LMod_\calP^\Omega$. Then the spectral sequence associated to the Postnikov decomposition
\[
\EXT_\calP(M,N)\simeq \lim_{n\rightarrow\infty}\EXT_\calP(M,N_{\leq n})
\]
is of signature
\begin{gather*}
E_1^{p,q} = \Ext^{p+q}_{\h\calP}(\pi_0 M;\pi_0 N[p])\Rightarrow\pi_{-q}\EXT_\calP(M,N),
\end{gather*}
with differential $d_r^{p,q}\colon E_r^{p,q}\rightarrow E_r^{p+r,q+1}$.\tqed
\end{theorem*}

For an arbitrary loop theory $\calP$, the category $\Model_\calP^\Omega$ is a localization of $\Model_\calP$; the distinguishing feature of the stable case is the identification of the localization $L\colon\LMod_\calP\rightarrow\LMod_\calP^\Omega$ (\cref{thm:stablelocalization}). Given $X\in\LMod_\calP$, write $X_{\Sigma^n}$ for the model defined by $X_{\Sigma^n}(P) = X(\Sigma^n P)$. Then $L\colon\LMod_\calP\rightarrow\LMod_\calP^\Omega$ may be identified as
\begin{equation}\label{eq:loc}
LX = \colim_{n\rightarrow\infty}\Sigma^{-n}X_{\Sigma^{-n}}.
\end{equation}
This yields a spectral sequence computing $\pi_\ast LX$, which can be considered a stable analogue of the spiral spectral sequence of \cite{dwyerkanstover1995bigraded}. We describe a particular application in \cref{ssec:lderss}; to describe this application we require some additional notation.

Fix another loop theory $\calP'$, which need not be stable, together with some functor $F\colon\Model_{\calP'}^\Omega\rightarrow \LMod_\calP^\Omega\simeq\Model_\calP^\Omega$ which preserves geometric realizations. Then $F$ is determined by its restriction $f$ to $\calP'$; explicitly, where $f_!\colon\Model_{\calP'}\rightarrow\Model_\calP$ is obtained from $f$ by left Kan extension, there is an equivalence $Lf_! X\simeq F X$ for $X\in\Model_{\calP'}^\Omega$. The composite $\pi_0\circ f \colon \calP'\rightarrow\LMod_\calP^\heart$ uniquely factors through the homotopy category to give $\ol{f}\colon\h\calP'\rightarrow\LMod_\calP^\heart$, and by left Kan extension this gives  functor $\ol{F}\colon\Model_{\calP'}^\heart\rightarrow\LMod_\calP^\heart$ preserving reflexive coequalizers, and this admits a total left-derived functor $\bbL\ol{F} = \ol{f}_! \colon\Model_{\h\calP'}\rightarrow\Model_{\h\calP}$ (cf.\ \cref{prop:leftderivedfunctors}).

\begin{theorem*}[\cref{thm:quillenss}]
In the situation of the previous paragraph, given $R\in\Model_{\calP'}^{\Omega}$, the spectral sequence in $\Model_\calP^\heart$ associated to the tower
\[
F R \simeq\colim_{n\rightarrow\infty} \Sigma^{-n} (f_! R)_{\Sigma^{-n}}
\] is of signature
\[
E^1_{p,q} = (\bbL_{p+q}\ol{F}\,\pi_0 R)[-p]\Rightarrow(\pi_0 F R)[q],\qquad d^r_{p,q}\colon E^r_{p,q}\rightarrow E^r_{p-r,q-1}.
\]
This spectral sequence converges, for instance, if $\pi_0$ preserves filtered colimits or if $\bbL\ol{F}\,\pi_0 R$ is truncated.
\tqed
\end{theorem*}

As described, the localization $L\colon \LMod_\calP\rightarrow\LMod_\calP^\Omega$ naturally factors through a functor $\LMod_\calP\rightarrow\Fun(\bbZ,\LMod_\calP)$. There is natural notion of a \textit{monoidal loop theory} (\cref{defn:monoidalres}), and this functor turns out to be lax monoidal whenever $\calP$ is monoidal. This leads to the introduction of pairings on spectral sequences constructed in the above manner (\cref{thm:pairingquillen}). For example, this gives rise to pairings on K\"unneth-type spectral sequences; such pairings have a history of being difficult to construct by hand \cite{tilson2016power}. In some cases, there is still more structure, coming from additional unstable data. Accessing this requires unstable methods, and we will come to it in \cref{ssec:spiralss}.

\subsubsection{Postnikov decompositions}\label{sssec:intropostnikov}(\cref{sec:postnikov}).

In \cref{sec:postnikov}, we consider Postnikov decompositions in $\Model_\calP$, where $\calP$ is a loop theory which need not be stable. The content of this section is more closely related to the earlier simplicial work of Blanc-Dwyer-Goerss \cite{blancdwyergoerss2004realization} and Goerss-Hopkins \cite{goersshopkinsxxxxmoduli}, and directly builds on Pstrągowski \cite{pstragowski2023moduli}. One consequence of the theory is the following.

\begin{theorem*}[\cref{thm:mappingspaces}]
Fix $A,C\in\Model_\calP^\Omega$, together with a map $\phi\colon \pi_0 A\rightarrow \pi_0 C$ in $\Model_\calP^\heart$. Let $\Map^\phi_{\calP}(A,C)\subset\Map_{\calP}(A,C)$ be the space of maps $f$ such that $\pi_0 f = \phi$. Then the Postnikov tower of $C$ in $\Model_{\calP}$ gives a decomposition 
\[
\Map^\phi_{\calP}(A,C)\simeq\lim_{n\rightarrow\infty}\Map^{\phi}_{\calP}(A,C_{\leq n}),
\]
where $\Map^{\phi}_{\calP}(A,C_{\leq 0}) = \{\phi\}$ and for each $n\geq 1$ there is a canonical fiber sequence
\[
\Map^\phi_{\calP}(A,C_{\leq n})\rightarrow\Map^\phi_{\calP}(A,C_{\leq n-1})\rightarrow \Map_{\h\calP/\pi_0 C}(\pi_0 A;B^{n+1}_{\pi_0 C}\Pi_n C).
\]
Here we have written $\Pi_n C = \pi_0 C^{S^n}$, which is an abelian group object in the slice category $\Model_{\h\calP}^\heart/\pi_0 C$, and $B^{n+1}_{\pi_0 C}\Pi_n C$ for its delooping therein. In particular, there are successively defined obstructions in the algebraic Quillen cohomology groups $H^{n+1}_{\h\calP/\pi_0 C}(\pi_0 A;\Pi_n C) = \pi_0 \Map_{\h\calP/\pi_0 C}(\pi_0 A;B^{n+1}_{\pi_0 C}\Pi_n C)$ to realizing $\phi$ as arising from a map $A\rightarrow C$.
\tqed
\end{theorem*}

In fact we prove something stronger, giving the analogous decomposition for mapping spaces in slice categories of $\Model_\calP$.

This theorem makes use of Postnikov towers in $\Model_\calP$, which may be defined levelwise. In \cref{ssec:spiralss}, we explain how one may derive these levelwise Postnikov towers (\cref{thm:derivedpostnikov}), leading to the spiral spectral sequence (\cref{thm:spiralss}), which plays the role of the spectral sequence of the same name constructed by Dwyer-Kan-Stover in \cite{dwyerkanstover1995bigraded} in the context of simplicial spaces. This is an unstable analogue of the spectral sequence associated to \cref{eq:loc}, and we describe explicitly how they relate (\cref{prop:synth}).

In \cref{ssec:realizations}, we verify that the obstruction theory of \cite{pstragowski2023moduli} for realizing an object $\Lambda\in\Model_\calP^\heart$ as $\Lambda = \pi_0 R$ with $R\in \Model_\calP^\Omega$ holds in our setting (\cref{thm:obstructions}). The short form of this is that there are successively defined obstructions in certain quotients of $H^{n+2}_{\h\calP/\Lambda}(\Lambda;\Lambda\langle n \rangle)$, where $\Lambda\langle n \rangle(P) = \Lambda(S^n\otimes P)$, to producing such an $R$ (\cref{prop:coarseobstruction}).

\subsubsection{Miscellaneous}(\cref{sec:completions}, \cref{sec:app}).

\cref{sec:completions} discusses localizations and completions in the context of theories. In \cref{ssec:localizingtheories}, we describe some general facts about localizations of theories. In \cref{ssec:completionsr}, we introduce $R$-linear theories for a connective $\bbE_2$-ring $R$, and study the corresponding notion of $I$-completions for a finitely generated ideal $I\subset R_0$. In particular, we give conditions under which the algebraic categories arising in this context may be described more explicitly. The examples which may be obtained in this context were our original motivation for working with infinitary theories. The use of infinitary theories as a tool for working with completeness conditions has also been employed by Brantner in \cite{brantner2017lubin} in the Noetherian case.

\cref{sec:app} can be considered an appendix to \cref{sec:stable}. In the latter, we required some properties of the spectral sequence associated to a tower in a stable category with $t$-structure; in particular, we required the relation between pairings of towers and pairings of their spectral sequences. We review the construction and convergence of these spectral sequences in \cref{ssec:constructss}, and give an overview of their multiplicative properties in \cref{ssec:monoidaltower} and \cref{ssec:pairings}.

\subsection{Examples}\label{ssec:examples}

The bulk of our work is, by necessity, somewhat abstract. The following are some examples that fit into the framework developed. For the most part, that these are examples may be verified using the results of \cref{ssec:constructingexamples}.

\subsubsection{Generalized $\Pi$-algebras}

Let $\calC$ be a theory, and let $\calP\subset\Model_{\calC}$ be the full subcategory generated by objects of the form $F\otimes h(C)$ where $F$ is a finite product of finite wedges of spheres and $C\in\calC$. If $\calC$ is pointed, then one may also wish to close this under suspensions. The category $\calP$ is then a loop theory with $\Model_{\calP}^\Omega\simeq\Model_{\calC}$ (\cref{thm:pialg}). There is no reason to expect the category $\Model_\calP^\heart$ to be easily described in general, but as a simple example we note that if $R$ is a connective ring spectrum and $\calC\subset\LMod_R$ is the full subcategory on objects of the form $R^{\oplus I}$ for a set $I$, then $\Model_\calC$ is the category of connective left $R$-modules and $\Model_\calP^\heart$ is the category of left $R_\ast$-modules concentrated in nonnegative degrees.

Many ideas in this paper, particularly in \cref{sec:postnikov}, may be traced back to work of Dwyer-Kan-Stover \cite{dwyerkanstover1995bigraded} and Blanc-Dwyer-Goerss \cite{blancdwyergoerss2004realization} on simplicial spaces and $\Pi$-algebras. These ideas were first explicitly interpreted in terms of $\infty$-categorical algebraic theories by Pstrągowski in \cite{pstragowski2023moduli}, and this example is a slight generalization of \cite[Proposition 4.3]{pstragowski2023moduli}, which may be regarded as the case where $\calC$ is the theory of groups.

\subsubsection{Modules over ring spectra}\label{sssec:rmodfilt}

Let $R$ be an $\bbA_\infty$ ring spectrum, and let $\calP = \LMod_R^\free$ be the full subcategory of $\LMod_R$ generated by $R$ under small coproducts, suspensions, and desuspensions. Then $\calP $ is a loop theory, and $\LMod_\calP ^\Omega\simeq\LMod_R$; moreover $\h\calP \simeq\LMod_{R_\ast}^\free$, and therefore $\LMod_{\h\calP }\simeq\LMod_{R_\ast}$ is the unbounded derived category of left $R_\ast$-modules.

The category $\LMod_\calP $ of all models is more exotic. This category fits into a span
\begin{center}\begin{tikzcd}
\LMod_{R_\ast}&\ar[l,"\tau_!"']\LMod_\calP \ar[r,"L"]&\LMod_R
\end{tikzcd},\end{center}
and we have proposed to consider it a deformation with generic fiber $\LMod_R$ and special fiber $\LMod_{R_\ast}$. There is another method for building deformations of stable homotopy theories that is perhaps more concrete, which proceeds via \textit{filtered objects}; see for instance \cite{gheorgheisaksenkrausericka2018cmotivic} for an application of this philosophy. The category $\LMod_\calP $ turns out to admit a description in terms of filtered spectra: there is an equivalence $\LMod_\calP \simeq\LMod_{W(R)}$, where $W(R)$ is the Whitehead tower of $R$ viewed as a filtered ring spectrum. We expect that various other categories of the form $\LMod_{\calP}$ admit similar descriptions using filtered objects.

\begin{rmk}\label{rmk:synth}
Let $X$ be a filtered spectrum, and set $X(\infty) = \colim_p X(p)$. Then there is filtration spectral sequence of signature
\[
E^1_{p,q} = \pi_q\Cof(X(p-1)\rightarrow X(p))\Rightarrow \pi_q X(\infty),\qquad d^r_{p,q}\colon E^r_{p,q}\rightarrow E^r_{p-r,q-1}.
\]
Examples from the filtered approach to deformations suggest that to understand $\pi_\ast X(\infty)$ as computed via this filtration spectral sequence, one should go further and consider all of the intermediate objects $X(p)$. At a basic level, this amounts to studying the bigraded homotopy groups $\pi_{s,c} X = \pi_s X(c)$. Each $\pi_{s,\ast} X$ is a module over a graded polynomial ring $\bbZ[\sigma]$, where $\sigma\colon \pi_{s,c-1}X\rightarrow\pi_{s,c}X$ is induced by $X(c-1)\rightarrow X(c)$; this satisfies $\pi_{s,\ast} X[\sigma^{-1}] = \pi_s X(\infty)\otimes\bbZ[\sigma^{\pm 1}]$, and in general the $\bbZ[\sigma]$-module structure of $\pi_{\ast,\ast}X$ tracks information about the computation of $\pi_\ast X(\infty)$ via the filtration spectral sequence, such as differentials and hidden extensions. All of the obstruction theories and spectral sequences we construct, both stably and unstably, proceed by constructing some filtered or cofiltered object with identifiable filtration quotients, and so we may directly import these insights into our setting. For example, in the context of \cref{thm:quillenss}, one might instead compute $\pi_\ast(f_! R)$. We encounter this idea again in \cref{ssec:spiralss}.
\tqed
\end{rmk}

\subsubsection{Modules over equivariant ring spectra}\label{sssec:equivariantex}

Let $G$ be a finite group,  $\Sp^G$ be the category of genuine $G$-equivariant spectra, and $R$ be an $\bbA_\infty$ ring in $\Sp^G$. Let $\calP \subset\LMod_R$ be the full subcategory generated under coproducts by objects of the form $\Sigma^\alpha R \otimes S_+$ for $\alpha\in RO(G)$ and $S$ a finite $G$-set. Then $\calP $ is a loop theory, $\LMod_\calP ^\Omega\simeq\LMod_R$, and $\LMod_{\h\calP } \simeq\LMod_{R_\star}$ is the unbounded derived category of left Mackey modules over the $RO(G)$-graded Green functor $R_\star$. In fact we would still have $\LMod_\calP ^\Omega\simeq\LMod_R$ had we restricted $\alpha$ to $\bbZ\subset RO(G)$, only in this case $\LMod_{\calP }^\heart$ would be the category of left Mackey modules over the $\bbZ$-graded Green functor $R_\ast$.

Now suppose for simplicity that $R$ is an $\bbE_2$ ring, so that $\LMod_R$ is a monoidal category. The monoidal product on $\LMod_R$ restricts to $\calP $, which then extends to a monoidal product on $\LMod_\calP $. In addition the monoidal product on $\calP $ induces one on $\h\calP $, which then extends to a monoidal product on $\LMod_{\h\calP }\simeq\LMod_{R_\star}$, which is simply the (derived) box product over $R_\star$. Now fix $M,N\in\LMod_R$, giving models $h(M),h(N)\in\LMod_\calP^\cn$. The methods of \cref{thm:quillenss} applied to $h(M)\otimes h(N)$ then give a Mackey functor K\"unneth spectral sequence
\[
E^1_{p,q} = \pi_{p+q} (\pi_\star M \square^\bbL_{\pi_\star R} \pi_\star N)_{\star-p}\Rightarrow \pi_{\star+q} (M\otimes_R N),\qquad d^r_{p,q}\colon E^r_{p,q}\rightarrow E^r_{p-r,q-1},
\]
such as those considered in \cite{lewsismandell2006equivariant}. Moreover, one has access to all the pairings on these spectral sequences that one could hope for.

As stated, \cref{thm:stablemappingss} gives a universal coefficient spectral sequence
\[
E_1^{p,q} = \Ext^{p+q}_{R_\star}(\pi_\star M;\pi_{\star+p}N)\Rightarrow\pi_{-q}\EXT_R(M,N),\qquad d_r^{p,q}\colon E_r^{p,q}\rightarrow E_r^{p+r,q+1}
\]
of abelian groups. By naturality, as discussed at the end of \cref{ssec:stabletowers}, this may be enhanced to a spectral sequence
\[
E_1^{p,q} = \Ext^{p+q}_{R_\star}(\pi_\star M;\pi_{\star+p}N)_\star\Rightarrow\pi_{\star-q}\EXT_R(M,N),\qquad d_r^{p,q}\colon E_r^{p,q}\rightarrow E_r^{p+r,q+1}
\]
internal to the category of $R_\star$-modules.

\subsubsection{Functor categories}

Let $\calP$ be a loop theory and $\calJ$ be a $1$-category. Then
\[
\Fun(\calJ,\Model_\calP^\Omega)\simeq\Model_{\calP^\calJ}^\Omega, \qquad \Fun(\calJ,\Model_\calP^\heart)\simeq\Model_{\calP^\calJ}^\heart,
\]
where $\calP^\calJ$ is the loop theory obtained as the image of the representable functors under the composite
\[
\Fun(\calJ\times\calP^\op,\Gpd_\infty)\simeq\Fun(\calJ,\Psh(\calP))\rightarrow\Fun(\calJ,\Model_\calP^\Omega),
\]
where the second functor is obtained from localization map $\Psh(\calP)\rightarrow\Model_\calP^\Omega$. Explicitly, for $P\in\calP$ and $i\in\calJ$, define $H_{P,i}\colon\calJ\rightarrow\Model_\calP^\Omega$ by
\[
H_{P,i}(j) = \calJ(i,j)\otimes h(P) = h\left(\coprod_{x\in \calJ(i,j)}P\right);
\]
then $\calP^\calJ$ is generated under coproducts in the category $\Fun(\calJ,\Model_\calP^\Omega)$ by functors of the form $H_{P,i}$. The assumption that $\calJ$ is a $1$-category is necessary to identify $\h(\calP^\calJ)\simeq(\h\calP)^\calJ$.

Write $q^\ast\colon \Model_\calP^\Omega\rightarrow\Fun(\calJ,\Model_\calP^\Omega)$ for the diagonal. Given $F\colon \calJ\rightarrow\Model_\calP^\Omega$ and $P\in\calP$, there is an equivalence
\[
\left(\lim_{j\in\calJ}F(j)\right)(P)\simeq\Map_\calP(q^\ast h(P),F).
\]
Applying \cref{thm:mappingspaces}, or \cref{thm:stablemappingss} if $\calP$ is stable, in this context then returns a homotopy limit spectral sequence for computing $\pi_\ast\left(\lim_{j\in\calJ}F(j)\right)$. If $\calP$ is stable, then applying \cref{thm:quillenss} to the colimit functor $\Fun(\calJ,\Model_\calP^\Omega)\rightarrow\Model_{\calP}^\Omega$ returns a homotopy colimit spectral sequence.

For example, when $\calJ = BG$ for a discrete group $G$, these return homotopy fixed point and homotopy orbit spectral sequences respectively. When $\calJ = \Delta^\op$, the homotopy colimit spectral sequence recovers the standard spectral sequence of a simplicial object. The fact that geometric realizations respect monoidal structures then combines with \cref{thm:pairingquillen} to produce the standard pairings on these spectral sequences.

\subsubsection{Completed modules}

The following example is developed in further detail in \cref{ssec:completionsr}. Let $R$ be an $\bbE_2$ ring and $I\subset\pi_0 R$ be a finitely generated ideal. Then there is a category $\LMod_R^{\Cpl(I)}$ of $I$-complete $R$-modules, and a category of $\LMod_{R_\ast}^{\Cpl(I)}$ of ``derived $I$-complete'' $R_\ast$-modules. These are developed in detail in \cite[Section 7]{lurie2018spectral}. For example, taking $R = E$ to be a Lubin-Tate spectrum of height $h$ and $I = \frakm\subset E_0$ to be the maximal ideal, the category $\Mod_E^{\Cpl(\frakm)}$ recovers the category $\Mod_E^\loc$ of $K(h)$-local $E$-modules, and $\Mod_{E_\ast}^{\Cpl(\frakm)}$ is the derived category of $L$-complete $E_\ast$-modules in the sense of \cite[Appendix A]{hoveystrickland1999morava}.

Let $\calP = \LMod_R^{\Cpl(I),\free}$ be the category of $I$-completions of free $R$-modules. Then $\calP$ is a loop theory, and $\LMod_\calP^\Omega\simeq\LMod_R^{\Cpl(I)}$. By contrast, one subtlety of completions is that it is not always the case that $\LMod_{\h\calP}\simeq\Mod_{R_\ast}^{\Cpl(I)}$. This does however hold under a minor algebraic tameness condition on $I\subset R_\ast$ which is satisfied in practice, at least in situations where one would wish to compute in $\Mod_{R_\ast}^{\Cpl(I)}$ (cf.\ \cref{sec:completions}).

\subsubsection{Bocksteins}\label{sssec:bock}

If $\calP$ is a theory, then $\calP\subset\Model_\calP$ is essentially a distinguished subcategory: the idempotent completion of $\calP$ is equivalent to the full subcategory of $\Model_\calP$ consisting of those models $X$ which are projective in the sense that $\Map_\calP(X,\bs)$ preserves geometric realizations (\cref{cor:modelsresolved}). By contrast, if $\calP$ is a loop theory, then $\calP\subset\Model_\calP^\Omega$ carries no particular universal property, and wildly different loop theories can model the same category. Here is an example (cf.\ \cite{hopkinslurie2017brauer}).

For simplicity, let $R$ be an $\bbE_\infty$ ring, and $u\in \pi_0 R$ be a non-zero-divisor in $\pi_\ast R$. By the preceding example, there is a loop theory $\calP = \Mod_R^{\Cpl(u),\free}$ with $\LMod_\calP^\Omega\simeq\Mod_R^{\Cpl(u)}$. The relevant algebraic tameness conditions are satisfied to ensure $\LMod_{\h\calP}\simeq\Mod_{R_\ast}^{\Cpl(I)}$. On the other hand, $C(u) = \Cof(u\colon R\rightarrow R)$ is a compact generator of $\Mod_R^{\Cpl(u)}$; thus where $E(u) = \EXT_R(C(u),C(u))$, there is an equivalence
\[
\Mod_R^{\Cpl(u)}\simeq\LMod_{E(u)},\qquad M\mapsto \EXT_R(C(u),M).
\]
To be precise, we should take the product on $E(u)$ which is opposite to the standard composition product, or else consider right $E(u)$-modules instead. This equivalence is an instance of Schwede-Shipley's Morita theory \cite{schwedeshipler2003stable} \cite[Section 7.1.2]{lurie2017higheralgebra}, and in the language of loop theories may be understood as follows. Let $\calP'\subset\Mod_R^{\Cpl(u)}$ be the full subcategory generated by $C(u)$ under small coproducts, suspensions, and desuspensions; all of these may also be computed in $\Mod_R$ itself. Because $C(u)$ is a compact generator of $\Mod_R^{\Cpl(u)}$, there is an equivalence $\Mod_R^{\Cpl(u)}\simeq\LMod_{\calP'}^\Omega$. On the other hand, $\calP'$ may be identified as the theory associated to $E(u)$ via the construction of \cref{sssec:rmodfilt}, and therefore $\LMod_{\calP'}^\Omega\simeq\LMod_{E(u)}$.

The algebraic category we extract from this choice of loop theory modeling $\Mod_R^{\Cpl(u)}$ is very different from $\Mod_{R_\ast}^{\Cpl(u)}$. Indeed, regularity of $u$ implies that
\[
\pi_\ast E(u) \cong (R_\ast/(u))[\ol{u}]/(\ol{u}^2)
\]
where $\ol{u}$ is a \textit{Bockstein} element associated to $u$. Thus, whereas $\calP$ leads one to approximate $\Mod_R^{\Cpl(u)}$ by complete $R_\ast$-modules, $\calP'$ leads one to approximate $\Mod_R^{\Cpl(u)}$ by $R_\ast/(u)$-modules together with $u$-Bockstein information. To illustrate this, first note that as $\EXT_R(C(u),M)\simeq\Sigma^{-1}M\otimes C(u)$, we may realize the equivalence $\Mod_R\simeq\LMod_{E(u)}$ instead by $M\mapsto M\otimes C(u)$. Now if $M\in\LMod_R^{\Cpl(u)}$, then the universal coefficient spectral sequence of \cref{thm:stablemappingss} applied to $M\simeq\EXT_R(R,M)$ using the theory $\calP'$ takes the form
\[
E_1^{p,q} = \Ext_{\pi_\ast E(u)}^{p+q}(R_\ast/(u),\pi_{\ast+p}(M\otimes C(u)))\Rightarrow \pi_{\ast-q} M.
\]
Taking the standard $\pi_\ast E(u)$-resolution of $R_\ast/(u)$, we may step back a page and view this as being of the form
\[
\pi_\ast (M\otimes C(u))[u]\Rightarrow \pi_\ast M,
\]
where $u$ is in stem $0$ and filtration $1$. This is a $u$-Bockstein spectral sequence for $\pi_\ast M$.

\subsubsection{Algebras over monads}

Let $\calP$ be a loop theory, and let $T$ be a monad on $\Model_\calP^\Omega$ which preserves geometric realizations. Define $T\calP\subset\Alg_T$ to be the full subcategory spanned by the essential image of $\calP$ under $T$. Then $T\calP$ is a loop theory, and there is an equivalence $\Alg_T\simeq\Model_{T\calP}^\Omega$. The condition that $T$ preserves geometric realizations may be relaxed at the expense of a more technical statement; this is all treated in \cref{ssec:constructingexamples}. The theory $\h T\calP$ should be thought of as the theory of \textit{homotopy operations} for $T$-algebras; compare the treatments of theories of power operations in \cite{rezk2006lectures} and \cite{lawsonxxxxen}.

This gives rise to more examples than we could hope to enumerate. Two examples in particularly good standing are $\calP = \Mod_{H\bbF_p}^\free$ and $T$ the free $\bbE_\infty$ algebra functor, and $\calP = \smash{\Mod_E^{\loc,\free}}$ for $E$ a Lubin-Tate spectrum of height $h$ and $T$ the free $K(h)$-local $\bbE_\infty$ algebra functor. These examples are developed in greater detail in \cite{balderrama2021algebraic}.

\subsubsection{Associative algebras}

Let $R$ be an $\bbE_2$ ring, so that $\LMod_R$ is monoidal and there is a category $\Alg_R = \Mon(\LMod_R)$ of $\bbA_\infty$ algebras over $R$. Let $\calP = \Alg_R^\free$; then $\Model_\calP^\Omega\simeq\Alg_R$ and $\Model_{\h\calP}^\heart\simeq\Alg_{R_\ast}^\heart$. In this context, \cref{thm:mappingspaces} is an obstruction theory for mapping spaces in $\Alg_R$ built from the ordinary Hochschild cohomology of $R_\ast$-algebras. By taking $R$ to instead be a $G$-equivariant $\bbE_2$ ring, we would obtain an obstruction theory built from the cohomology of associative Green algebras for the commutative Green functor $R_\star$.

\subsubsection{$\bbF_p$-synthetic $p$-profinite spaces}\label{sssec:syntheticprofinite}

We end with an example of a different flavor.

Fix a prime $p$, and let $\calP_0^\op\subset\Gpd_\infty$ be the full subcategory spanned by finite products of the Eilenberg-MacLane spaces $K(\bbF_p,n)$. Then $\calP_0$ is a finitary loop theory; we may complete it to a full loop theory $\calP$ by declaring $\calP\subset\Model_{\calP_0}$ to be the full subcategory generated by $\calP_0$ under coproducts. Now consider $\Model_\calP^\op$; we propose to consider this a category of $\bbF_p$-synthetic $p$-profinite spaces. This is in analogy with the category of $\bbF_p$-synthetic spectra in the sense of Pstrągowski \cite{pstragowski2023synthetic}, which is a deformation of the category of spectra refining the classic mod $p$ Adam spectral sequence; see in particular \cite[Appendix A]{burklundhahnsenger2019boundaries}.

Let $\Fin_p\subset\Gpd_\infty$ be the full subcategory of $p$-finite spaces, i.e.\ those $X$ such that $\pi_0 X$ is finite, $X$ is truncated, and each $\pi_n(X,x)$ is a finite $p$-group. There is then an equivalence of categories $\Pro(\Fin_p)\simeq(\Model_\calP^\Omega)^\op$, and $\Model_{\h\calP}^\heart$ is equivalent to the category of unstable algebras over the Steenrod algebra. In this context \cref{thm:mappingspaces} may be interpreted as an unstable Adams spectral sequence.

Now let $\calP'$ be defined in the same manner as $\calP$, only using pointed objects throughout. To any pointed space $X$, we may define a model $h(X)$ of $\calP'$ by $h(X)(P) = \Map(X,P)$, and $X_p^\wedge = \Map_{\calP'}(h(X),h(S^0))$ recovers the $p$-profinite completion of $X$. Define
\[
\pi_{s,c}^{\bbF_p} X = \pi_{s}\Fib(X_p^\wedge\rightarrow\Map_{\calP'}(h(X),h(S^0)_{\leq c-1})).
\]
Although we are now working unstably, the ideas mentioned in \cref{rmk:synth} still indicate that $\pi_{\ast,\ast}^{\bbF_p} X$ may be viewed as a deformation of $\pi_\ast X$ associated to the unstable Adams spectral sequence. Thus we may consider $\pi_{\ast,\ast}^{\bbF_p}X$ as a candidate definition for the $\bbF_p$-synthetic \textit{unstable} homotopy groups of $X$.

\section{Mal'cev theories}\label{sec:malc}

This section covers the general properties of Mal'cev theories. In particular, in \cref{ssec:malc}, we show that if $\calP$ is a Mal'cev theory, then the category $\Model_{\calP}$ of models of $\calP$ freely adjoins geometric realizations to $\calP$. Moreover, we verify that $\Model_\calP$ is presentable under some minor smallness conditions on $\calP$. In \cref{ssec:simplicial}, we restrict to the case where $\calP$ is a discrete Mal'cev theory, verify that $\Model_{\calP}$ is the underlying $\infty$-category of Quillen's model structure on simplicial objects in $\Model_\calP^\heart$, and review the resulting notion of left-derived functor.

\subsection{Definitions and universal properties}\label{ssec:malc}

The structure of a \textit{herd} on a set $X$ is a ternary operation $t$ satisfying $t(x,x,y) = y$ and $t(x,y,y) = x$; write $\Hrd$ for the category of herds. Herds are the models of a finitary and discrete algebraic theory: the defining relations $t(x,x,y) = y$ and $t(x,y,y) = x$ determine a presentation of a Lawvere theory in the sense of \cite[Section II.2]{lawvere1963functorial}. In particular, herd objects can be defined in an arbitrary category with finite products, allowing for the following definition.

\begin{defn}\label{def:malc}
A \textit{Mal'cev theory} is a category $\calP$ such that
\begin{enumerate}
\item $\calP$ admits all small coproducts;
\item All objects of $\calP$ admit the structure of a coherd (i.e.\ of a herd object in $\calP^\op$).
\end{enumerate}
For a regular cardinal $\kappa$, a Mal'cev theory $\calP$ is said to be \textit{$\kappa$-bounded} if there exists a small full subcategory $\calP_0\subset\calP$ such that
\begin{enumerate}[resume]
\item $\calP_0$ is closed under $\kappa$-small coproducts;
\item Every object of $\calP$ is a retract of a small coproduct of objects of $\calP_0$;
\item For every $P\in\calP_0$ and every set of objects $\{P_i':i\in I\}$ in $\calP$, the canonical map $\colim_{F\subset I,\, |F|<\kappa}\Map_\calP(P,\coprod_{i\in F}P_i')\rightarrow\Map_\calP(P,\coprod_{i\in I}P_i')$ is an equivalence.
\end{enumerate}
A Mal'cev theory $\calP$ is \textit{bounded} if it is $\kappa$-bounded for some $\kappa$, and is \textit{discrete} if it is a $1$-category.
\tqed
\end{defn}

We will refer to Mal'cev theories as just \textit{theories}. After this subsection, we will moreover assume that all of our theories are bounded; see \cref{rmk:bounded}. Throughout this subsection, and everywhere else, $\calP$ will always refer to a theory, at times satisfying additional assumptions. 

If $X$ is a herd object in some category $\calC$, then $\Map_\calC(A,X)$ will be a herd object in $\Gpd_\infty$ for any $A\in \calC$. The category of herd objects in $\Gpd_\infty$ is modeled by the Quillen's model structure on simplicial herds \cite[Section II.4]{quillen1967homotopical}, where the weak equivalences and fibrations of simplicial herds are detected on their underlying simplicial sets, see \cite[Section 5.5.9]{lurie2017highertopos}. This allows us to use simplicial herds to prove things about herd objects in $\Gpd_\infty$, and we will freely do so.

\begin{samepage}
\begin{lemma}\label{lem:herdfacts}
\hphantom{blank}
\begin{enumerate}
\item Every surjection of simplicial herds is a Kan fibration;
\item If $X$ is a herd object in $\Gpd_\infty$ and $A$ and $B$ are pointed spaces, then the restriction $r\colon \Map(A\times B,X)\rightarrow\Map(A\vee B,X)$ admits a section.
\end{enumerate}
Now let $p\colon B\rightarrow X$ be a map of herd objects in $\Gpd_\infty$ equipped with a section $s$, and let $q\colon A\rightarrow X$ be any map of $\infty$-groupoids.
\begin{enumerate}[resume]
\item The canonical map $j\colon A\times_X B\rightarrow A\times B$ admits a retraction;
\item The natural map $\pi_0(A\times_X B)\rightarrow\pi_0A\times_{\pi_0 X}\pi_0 B$ is a bijection.
\end{enumerate}
In particular, if $p$ is a covering map, then $p$ has trivial monodromy.
\end{lemma}
\end{samepage}
\begin{proof}
(1)~~ This is well known should we replace herds with groups, and the same proof applies. In brief, suppose given a surjection $p\colon E\rightarrow B$ of simplicial herds, elements $x_0,\ldots,x_{k-1},x_{k+1},\ldots,x_{n+1}\in E_n$ such that $d_i(x_j)=d_{j-1}(x_i)$ for $i<j$ and $i\neq k$, and $y\in  B_{n+1}$ such that $d_i(y) = p(x_i)$ for $i\neq k$. Inductively define $w_r\in E_{n+1}$ such that $p(w_r) = y$ and $d_i w_r = x_i$ for $i\leq r$ and $i\neq k$ by choosing $w_{-1}$ to be any element in the preimage of $y$, and setting $w_r = t(w_{r-1},s_r d_r w_{r-1},s_r x_r)$, except when $r=k$, in which case $w_r = w_{r-1}$. Then $w_{n+1}\in E_{n+1}$ witnesses the Kan condition.

(2)~~ If $G$ is a group object in $\Hrd$, then as the herd operation $t\colon G\times G\times G \rightarrow G$ is a group homomorphism we have $t(u,v,w)t(x,y,z) = t(ux,vy,wz)$ for any $u,v,w,x,y,z\in G$. In particular, if $e\in G$ is the unit and $g,h\in G$ are arbitrary, then $gh = t(g,e,e)t(e,e,h) = t(g,e,h) = t(e,e,h)t(g,e,e) = h g$. 

(3)~~ Fix a herd object $X$ and pointed spaces $A$ and $B$, whose basepoint we denote by $e$. We may model these objects by simplicial sets, in which case a section $s\colon \Map(A\vee B,X)\rightarrow\Map(A\times B,X)$ to the restriction is given by $s(f)(a,b) = t(f(a),f(e),f(b))$.

(4)~~ We may model the objects and maps in question by simplicial sets, choosing $p$ in such a way that the strict pullback $A\times_X B$ models the homotopy pullback. Now a retraction $r\colon A\times B\rightarrow A\times_X B$ is given by $r(a,b) = (a,t(s(q(a)),s(p(b)),b))$.

(5)~~ The map $\pi_0(A\times_X B)\rightarrow\pi_0 A\times_{\pi_0 X}\pi_0 B$ is always surjective, so we must only verify that it is injective. This follows from (4), which implies that the composite $\pi_0(A\times_X B)\rightarrow\pi_0 A\times_{\pi_0 X}\pi_0 B\subset \pi_0 A\times \pi_0 B$ admits a retraction.
\end{proof}

\begin{defn}
The category of \textit{models} of a theory $\calP$ is the full subcategory $\Model_\calP\subset\Psh(\calP)$ of small presheaves $X$ on $\calP$ such that for any set $\{P_i:i\in I\}$ of objects in $\calP$, the canonical map
\[
X\left(\coprod_{i\in I}P_i\right)\rightarrow \prod_{i\in I}X(P_i)
\]
is an equivalence. The category of \textit{discrete models} of $\calP$ is the full subcategory $\Model_\calP^\heart\subset\Model_\calP$ of $\Set$-valued models of $\calP$.
\tqed
\end{defn}

In general we will write, for instance, $\Map_\calP$ rather than $\Map_{\Model_\calP}$. As $\calP$ consists of coherds, if $X\in\Model_\calP$ and $P\in\calP$, then $X(P) \simeq \Map_\calP(h(P),X)$ is itself a herd. This need not be natural in $P$, but it is natural in $X$, i.e.\ maps $X\rightarrow Y$ of models induce maps $X(P)\rightarrow Y(P)$ of herds. This is sufficient for the following.

\begin{prop}\label{prop:geometricrealizationsgroups}
The subcategory $\Model_{\calP}\subset\Psh(\calP)$ is closed under small limits and geometric realizations. Moreover, geometric realizations preserve small products in $\Model_\calP$.
\end{prop}
\begin{proof}
Let $X\colon \calJ\rightarrow\Model_\calP$ be a small diagram and $\lim_{j\in\calJ}X_j\in\Psh(\calP)$ be its pointwise colimit, which we are claiming lives in $\Model_\calP\subset\Psh(\calP)$, and let $\{P_i:i\in I\}$ be a set of objects in $\calP$. Then
\[
\lim_{j\in\calJ}X_j(\coprod_{i\in I}P_i) \simeq \lim_{j\in\calJ}(\prod_{i\in I}X_j(P_i))\simeq \prod_{i\in I}(\lim_{j\in\calJ}X_j(P_i)),
\]
so that $\lim_{j\in\calJ}X_j\in\Model_\calP$ as claimed.

Next let $X\colon \Delta^\op\rightarrow\Model_\calP$ be a simplicial object, and write $X_j = X([j])$. Then the pointwise colimit $\colim_{j\in\Delta^\op}X_j$ will live in $\Model_\calP$ by the same reasoning, with $\lim_{j\in\calJ}$ replaced by $\colim_{j\in\Delta^\op}$, provided we can verify that
\[
\colim_{j\in\Delta^\op}\prod_{j\in\calJ}X_j(P_i)\simeq \prod_{j\in\calJ} \colim_{j\in\Delta^\op}X_j(P_i).
\]
This will also prove that geometric realizations preserve small products in $\Model_\calP$. Abbreviate $X_{i,j} = X_j(P_i)$. Then each $X_{i,j}$ is a herd object in $\Gpd_\infty$, and this is natural in $j$. As the category of herd objects in $\Gpd_\infty$ is modeled by Quillen's model category of simplicial herds, we may model each simplicial herd object $\{X_{i,j}:j\in\Delta^\op\}$ by a bisimplicial herd $\{X_{i,j,k}:j,k\in\Delta^\op\}$, in which case the geometric realization $\colim_{j}X_{j,i}$ is modeled by the diagonal $\{X_{i,j,j}:j\in\Delta^\op\}$. As simplicial herds are fibrant by \cref{lem:herdfacts}(1), small products of simplicial herds are homotopy products. Thus $\colim_j\prod_i X_{i,j}\simeq \prod_i \colim_j X_{i,j}$ as products and diagonals of bisimplicial sets commute.
%The assertion regarding small limits is clear, so we must verify that the pointwise geometric realization of a simplicial object in $\Model_{\calP}$ again lives in $\Model_{\calP}$. As $\calP$ consists of coherds and the forgetful functor $\Hrd(\Gpd_\infty)\rightarrow\Gpd_\infty$ preserves geometric realizations, it is sufficient to verify that geometric realizations preserve small products in the category $\Hrd(\Gpd_\infty)$ of herd objects in $\Gpd_\infty$. As $\Hrd(\Gpd_\infty)$ is modeled by the model category of simplicial herds, we can model a simplicial object in $\Hrd(\Gpd_\infty)$ by a bisimplicial herd and its geometric realization by the diagonal of this bisimplicial herd. As all simplicial herds are fibrant by \cref{lem:herdfacts}(1), small products of simplicial herds are homotopy products, so the result follows as products and diagonals of bisimplicial sets commute.
\end{proof}

We also record the following here. 

\begin{lemma}\label{lem:cartrel}
Fix a levelwise Cartesian square
\begin{center}\begin{tikzcd}
W\ar[r]\ar[d]&X\ar[d,"\pi"]\\
Y\ar[r]&Z
\end{tikzcd}\end{center}
of simplicial objects in $\Model_{\calP}$, and suppose that $\pi$ is levelwise a $\pi_0$-surjection. Then the square remains Cartesian after geometric realization.
\end{lemma}
\begin{proof}
By \cref{prop:geometricrealizationsgroups}, pullbacks and geometric realizations are computed pointwise, so by evaluating an arbitrary element $P\in\calP$ we may reduce to proving the corresponding statement with $\Model_{\calP}$ replaced by $\Hrd(\Gpd_\infty)$. The square can now be modeled as a Cartesian square of bisimplicial herds in which the map $\pi$ is levelwise a surjection. This square remains Cartesian upon taking diagonals, and remains homotopy Cartesian as $\pi$ remains a Kan fibration by \cref{lem:herdfacts}(1). This proves the claim.
\end{proof}

\begin{rmk}\label{rmk:malcdone}
After this point, herds will no longer appear explicitly.
\tqed
\end{rmk}

Observe that $\Model_{\calP}$ consists of those small presheaves on $\calP$ which are local with respect to the class of maps of the form $\coprod_{i\in I}h(P_i)\rightarrow h(\coprod_{i\in I}P_i)$ for $\{P_i:i\in I\}$ a set of objects in $\calP$. 

\begin{lemma}\label{prop:modelsreflective}
The inclusion $R\colon\Model_{\calP}\rightarrow\Psh(\calP)$ admits a left adjoint.
\end{lemma}
\begin{proof}
By the Yoneda lemma, it is sufficient to verify the pointwise assertion that for all $X\in\Psh(\calP)$, the functor $\Map_{\Psh(\calP)}(X,R(\bs))\colon \Model_{\calP}\rightarrow\Gpd_\infty$ is representable; see for instance \cite[Proposition 6.1.11]{cisinski2019higher}. By definition of $\Psh(\calP)$, the presheaf $X$ is small, and thus admits a presentation of the form $X \simeq \colim_{n\in\Delta^\op}\coprod_{i\in I_n}h(P_{n,i})$ for some sets $I_n$ and $P_{n,i}\in\calP$, see for example \cite[Section 5]{mazelgee2019grothendieck}. As a consequence, we have
\[
\Map_{\Psh(\calP)}(X,R(\bs))\simeq\Map_{\Psh(\calP)}(\colim_{n\in\Delta^\op}h(\coprod_{i\in I_n} P_{n,i}),R(\bs)).
\]
\cref{prop:geometricrealizationsgroups} shows that $\colim_{n\in\Delta^\op}h(\coprod_{i\in I_n}P_{n,i})$ lives in $\Model_{\calP}$, and we conclude as
\[
\Map_{\Psh(\calP)}(\colim_{n\in\Delta^\op}h(\coprod_{i\in I_n} P_{n,i}),R(\bs))\simeq \Map_{\calP}(\colim_{n\in\Delta^\op}h(\coprod_{i\in I_n} P_{n,i}),\bs).\qedhere
\]
\end{proof}

\begin{prop}\label{cor:modelsresolved}
For a theory $\calP$,
\begin{enumerate}
\item The category $\Model_{\calP}$ admits all small colimits;
\item The subcategory $\Model_{\calP}\subset\Psh(\calP)$ is the smallest full subcategory containing all representables and closed under geometric realizations;
\item For $X\in\Model_{\calP}$, the functor $\Map_{\calP}(X,\bs)$ preserves geometric realizations if and only if $X$ is a retract of a representable.
\end{enumerate}
\end{prop}
\begin{proof}
(1)~~ The category $\Model_\calP$ admits all small colimits by \cref{prop:modelsreflective}, see for instance \cite[Remark 5.2.7.5]{lurie2017highertopos}.

(2)~~ This follows from the proof of \cref{prop:modelsreflective}, which gives a way of writing any $X\in\Model_{\calP}$ as a geometric realization of representables. 

(3)~~ If $X$ is representable, then $\Map_{\Psh(\calP)}(X,\bs)$ preserves all small colimits. By \cite[Proposition 5.1.4.9]{lurie2017highertopos}, the same is then true if $X$ is a retract of a representable. Thus $\Map_{\calP}(X,\bs)$ preserves all geometric realizations, as these are computed in $\Psh(\calP)$. Conversely, if $\Map_{\calP}(X,\bs)$ preserves geometric realizations, then upon using (2) to write $X\simeq\colim_{n\in\Delta^\op} h(P_n)$, we find $\Map_{\calP}(X,X)\simeq\colim_{n\in\Delta^\op}\Map_{\calP}(X,h(P_n))$, so that the identity of $X$ factors through some representable.
\end{proof}

\cref{cor:modelsresolved} may be used to identify $\Model_\calP$ as a certain cocompletion of $\calP$, giving access to two related universal properties.

\begin{prop}\label{prop:lkan}
Let $\calD$ be a category admitting geometric realizations. Then restriction $\Fun(\Model_\calP,\calD)\rightarrow\Fun(\calP,\calD)$ along the Yoneda embedding induces an equivalence
\[
\Fun^{\Delta^\op}(\Model_\calP,\calD)\simeq\Fun(\calP,\calD)
\]
between the category of functors $\Model_\calP\rightarrow\calD$ preserving geometric realizations and the category of arbitrary functors $\calP\rightarrow\calD$. If $\calD$ also admits coproducts and thus all small colimits, then this further restricts to an equivalence
\[
\Fun^L(\Model_\calP,\calD)\simeq\Fun^{\amalg}(\calP,\calD)
\]
between the category of functors $\Model_\calP\rightarrow\calD$ preserving colimits and the category of functors $\calP\rightarrow\calD$ preserving coproducts. In either case, the inverse is given by left Kan extension along $h\colon \calP\rightarrow\Model_\calP$.
\end{prop}
\begin{proof}
\cref{cor:modelsresolved}(2) implies that $\Model_\calP$ is equivalent to the category $\Psh^{\{\Delta^\op\}}_\emptyset(\calP)$ constructed in \cite[Proposition 5.3.6.2]{lurie2017highertopos} (where we write $\Psh^{\calK}_{\calR}(\calC)$ for what would there be written $\calP^{\calK}_{\calR}(\calC)$). As $\calP\rightarrow\Model_\calP$ preserves coproducts, the category $\Model_\calP$ is further equivalent to $\Psh^{\{\Delta^\op\}\cup S}_{S}(\calP)$, where $S$ is the class of small sets. Thus the claims follow from \cite[Proposition 5.3.6.2]{lurie2017highertopos}.
\end{proof}

Now is a good time to introduce the following notation.

\begin{defn}
Given a functor $f\colon \calP\rightarrow\calD$ from a theory to a category admitting geometric realizations, we write
\[
f_!\colon \Model_\calP\rightarrow\calD
\]
for the left Kan extension of $f$ along the Yoneda embedding $h\colon \calP\rightarrow\Model_\calP$.
\tqed
\end{defn}

\begin{rmk}\label{rmk:func1}
Given $f\colon \calP\rightarrow\calD$, \cref{prop:lkan} implies that $f_!\colon \Model_\calP\rightarrow\calD$ is the unique functor which preserves geometric realizations and satisfies $f_!(h(P)) = f(P)$ for $P\in\calP$. A good special case is obtained when $\calD = \Model_{\calP'}$ for another theory $\calP'$ and $f$ restricts to a coproduct-preserving functor $f\colon \calP\rightarrow\calP'$. In this case $f_!$ preserves colimits, and restriction $f^\ast\colon \Psh(\calP')\rightarrow\Psh(\calP)$ itself restricts to a functor $f^*\colon \Model_{\calP'}\rightarrow\Model_{\calP}$ which is right adjoint to $f_!$.
\tqed
\end{rmk}

We now give the following recognition theorem for categories of the form $\Model_\calP$.

\begin{theorem}\label{thm:recognitiontheorem}
Let $\calD$ be a category admitting geometric realizations, and let $F\colon \Model_{\calP}\rightarrow\calD$ be a functor. Write $f = F \circ h \colon \calP\rightarrow\Model_{\calP}\rightarrow\calD$. Then
\begin{enumerate}
\item $F$ preserves geometric realizations if and only if it arises as the left Kan extension of $f$ along $h\colon\calP\rightarrow\Model_\calP$.
\item Suppose that $\calD$ admits coproducts. Then $F$ preserves colimits if and only if $F$ preserves geometric realizations and $f$ preserves coproducts.
\item If the following hold, then $F$ is fully faithful:
\begin{enumerate}
\item[(a)]$F$ preserves geometric realizations,
\item[(b)]$f$ is fully faithful,
\item[(c)]For all $P\in\calP$, the functor $\Map_\calD(f(P),\bs)$ preserves geometric realizations;
\end{enumerate}
\item If the following hold, then $F$ is an equivalence:
\begin{enumerate}
\item[(d)]$F$ preserves colimits,
\item[(e)]$F$ is fully faithful,
\item[(f)]The right adjoint to $F$, given by $G(D) = \Map_\calD(f(\bs),D)$, is conservative.
\end{enumerate}
\end{enumerate}
\end{theorem}
\begin{proof}
(1,2)~~ These follow from \cref{prop:lkan}.

(3)~~ Suppose given $F\colon\Model_{\calP}\rightarrow\calD$ satisfying conditions (a)-(c). We must show that
\[
\Map_{\calP}(X,Y)\simeq\Map_{\calD}(F(X),F(Y)).
\]
As $X$ may be written as a geometric realization of representable functors, by (a) we reduce to showing that
\[
\Map_{\calP}(h(P),Y)\simeq\Map_{\calD}(f(P),F(Y))
\]
for $P\in\calP$. As $Y$ may also be written as a geometric realization of representable functors, by (a) and (c) we reduce to showing that
\[
\Map_{\calP}(h(P),h(P'))\simeq\Map_\calD(f(P),f(P'))
\]
for $P,P'\in\calP$, which is a consequence of (b).

(4)~~ Condition (d) ensures that the functor $G$ described in (f) is right adjoint to $F$, and the assertion then follows from the general fact that an adjunction $F\dashv G$ with $F$ fully faithful and $G$ conservative is an equivalence.
\end{proof}

Suppose now that $\calP$ is $\kappa$-bounded, choose a subcategory $\calP_0\subset\calP$ realizing this, and let $\Psh^{\Pi_\kappa}(\calP_0)\subset\Psh(\calP_0)$ be the full subcategory consisting of presheaves which preserve $\kappa$-small products.

\begin{lemma}\label{lem:boundedmodels}
\hphantom{blank}
\begin{enumerate}
\item The subcategory $\Psh^{\Pi_\kappa}(\calP_0)\subset\Psh(\calP_0)$ is closed under geometric realizations and $\kappa$-filtered colimits;
\item The category $\Psh^{\Pi_\kappa}(\calP_0)$ consists of those objects of $\Psh(\calP_0)$ local with respect to the set of maps of the form $\coprod_{i\in F}h(P_i)\rightarrow h(\coprod_{i\in I} P_i)$ where $\{P_i:i\in F\}$ is a set of objects of $\calP_0$ with $|F|<\kappa$.
\end{enumerate}
In particular, $\Psh^{\Pi_\kappa}(\calP_0)$ is a $\kappa$-compactly generated presentable category.
\end{lemma}
\begin{proof}
Claim (1) follows from the same argument as \cref{prop:geometricrealizationsgroups}, using the fact that $\kappa$-filtered colimits preserve $\kappa$-small products. Claim (2) is just a reformulation of the definition of $\Psh^{\Pi_\kappa}(\calP_0)$. Presentability then follows from \cite[Proposition 5.5.4.15]{lurie2017highertopos}.
\end{proof}

\begin{prop}\label{prop:modelspresentable}
Restriction $R\colon \Model_{\calP}\rightarrow\Psh^{\Pi_\kappa}(\calP_0)$ is an equivalence.
\end{prop}
\begin{proof}
We verify the conditions of \cref{thm:recognitiontheorem}. First, as geometric realizations are computed pointwise in either category, they are preserved by $R$. Next, by our smallness assumption on the objects of $\calP_0$, we find that for any collection of objects $\{P_i : i \in I\}$ in $\calP_0$ there is an equivalence
\[
R(h(\coprod_{i\in I}P_i)) \simeq \colim_{\substack{F\subset I\\ |F|<\kappa}}h(\coprod_{i\in F}P_i) \simeq \colim_{\substack{F\subset I \\ |F|<\kappa}}\coprod_{i\in F}h(P_i) \simeq \coprod_{i\in F}h(P_i)
\]
in $Psh^{\Pi_\kappa}(\calP_0)$. As all objects of $\calP$ may be written as a small coproduct of objects of $\calP_0\subset\calP$, it follows that $\calP\rightarrow\Psh^{\Pi_\kappa}(\calP_0)$ preserves coproducts, and so $R$ preserves colimits. In particular, if $P\in \calP$ is written as a small coproduct $P = \coprod_{i\in I}P_i$ with $P_i\in \calP_0$, then
\[
\Map_{\Psh^{\Pi_\kappa}(\calP_0)}(R(h(P)),\bs) \simeq \prod_{i\in I} \Map_{\Psh^{\Pi_\kappa}(\calP_0)}(h(P_i),\bs)
\]
preserves geometric realizations, cf.\ \cref{prop:geometricrealizationsgroups}. The right adjoint $G$ to $R$ satisfies $G(X)(P) = X(P)$ when $P\in \calP_0\subset \calP$, and thus is conservative. So the conditions of \cref{thm:recognitiontheorem} are satisfied and $R$ is an equivalence as claimed.
\end{proof}

\begin{cor}
If $\calP$ is bounded, then $\Model_\calP$ is presentable. In particular, if $\calP$ is bounded then $\Model_\calP$ admits all small limits and colimits.
\end{cor}
\begin{proof}
Combine \cref{prop:modelspresentable} and \cref{lem:boundedmodels}.
\end{proof}

\begin{rmk}
Everything in this subsection has a fully $1$-categorical analogue, where all categories are taken to be $1$-categories, $\Model_{\calP}$ is replaced by $\Model_\calP^\heart$, and geometric realizations reduce to reflexive coequalizers. 
\tqed
\end{rmk}

\begin{rmk}\label{rmk:bounded}
We assume for the rest of this paper that all theories are bounded. In order to avoid cumbersome notation, we will adopt the following convention: if $\calP$ is a $\kappa$-bounded theory, choose $\calP_0\subset\calP$ realizing this; now, $\Psh(\calP)$ refers to $\Psh(\calP_0)$, $\Model_{\calP}$ refers to $\Psh^{\Pi_\kappa}(\calP_0)$, and so forth. In case one should meet a theory which is not bounded, we point out that an arbitrary theory $\calP$ is of the form $\calP = \calP_0'$ where $\calP'$ is a bounded theory with respect to a larger universe.
\tqed
\end{rmk}

\subsection{Rigidification and left-derived functors}\label{ssec:simplicial}

Throughout this subsection, all of our theories are assumed to be discrete theories. In this subsection, we briefly review the notion of left-derived functors available for categories of the form $\Model_\calP^\heart$. This story is classical, and goes back to \cite{quillen1967homotopical} and \cite{doldpuppe1961homologie}; see also \cite{tierneyvogel1969simplicial}. To facilitate comparisons with the classical theory, we begin with an identification of a model for $\Model_{\calP}$.

For a category $\calC$, write $\s\calC = \Fun(\Delta^\op,\calC)$ for the category of simplicial objects in $\calC$.

\begin{lemma}[{\cite[Section II.4]{quillen1967homotopical}}]
There is a simplicial model structure on $\s\Model_\calP^\heart$ in which a map $f\colon X\rightarrow Y$ is a weak equivalence, resp., fibration, if and only if for all $P\in\calP$ the map $f(P)\colon X(P)\rightarrow Y(P)$ is a weak equivalence, resp., fibration.\qed
\end{lemma}

\begin{prop}\label{thm:rigidification}
The ($\infty$-categorical) colimit functor 
\[
C\colon\s\Model_\calP^\heart \subset\Fun(\Delta^\op,\Model_{\calP})\rightarrow\Model_{\calP}
\]
realizes $\Model_{\calP}$ as the underlying $\infty$-category of $\s\Model_\calP^\heart$.
\end{prop}
\begin{proof}
Given a simplicial set $X$, write $|X| = \colim_{n\in\Delta^\op}X_n \in \Gpd_\infty$ for its underlying $\infty$-groupoid. Let $W$ denote the class of weak equivalences in $\s\Model_\calP^\heart$. Then to show that $C$ induces an equivalence $\s\Model_\calP^\heart[W^{-1}]\simeq \Model_\calP$, we must verify the following:
\begin{enumerate}
\item $C$ inverts $W$;
\item $C$ is essentially surjective;
\item For $X,Y\in\s\Model_\calP^\heart$ with $X$ cofibrant, $C$ gives an equivalence $|\ul{\Map}_{\s\Model_\calP^\heart}(X,Y)|\simeq\Map_{\calP}(C X,C Y)$, where $\ul{\Map}_{\s\Model_\calP^\heart}$ denotes the simplicial enrichment of $\s\Model_\calP^\heart$.
\end{enumerate}
These are verified as follows.

(1)~~ Suppose that $f\colon X\rightarrow Y$ is a weak equivalence in $\s\Model_\calP^\heart$. This means that $f$ induces an equivalence $|X(P)|\rightarrow |Y(P)|$ for each $P\in \calP$. As $|X(P)| = C(X)(P)$, it follows that $C(X)\rightarrow C(Y)$ is an equivalence.

(2)~~ The restriction of $C$ to $\s\calP\subset\s\Model_\calP^\heart$ is already essentially surjective by \cref{cor:modelsresolved} .

(3)~~ First note that if $P\in \calP$ and $Y\in \s\Model_\calP^\heart$, then
\[
|\ul{\Map}_{\s\Model_\calP^\heart}(h(P),Y)| \simeq |Y(P)| \simeq C(Y)(P) \simeq \Map_{\calP}(h(P),C(Y)).
\]
In general, if $X \in \s\Model_\calP^\heart$ is cofibrant, then $X_n \in \calP$ for each $n$. As homotopy geometric realizations of simplicial objects in $\s\Model_\calP^\heart$ are modeled by diagonals, we have $X\simeq\hocolim_{n\in\Delta^\op}X_n$, and thus
\begin{align*}
|\ul{\Map}_{\s\Model_\calP^\heart}(X,Y)|&\simeq \lim_{n\in\Delta}|\ul{\Map}_{\s\Model_\calP^\heart}(X_n,Y)|\\
&\simeq \lim_{n\in\Delta}\Map_\calP(X_n,C(Y))\simeq \Map_\calP(C(X),C(Y))
\end{align*}
as claimed.
\end{proof}

We now review the relevant notion of left-derived functor. Let $\calP$ and $\calP'$ be discrete theories, and fix an arbitrary functor $\ol{f}\colon \calP'\rightarrow\Model_\calP^\heart$. By left Kan extension, we obtain a functor $\ol{F}\colon\Model_{\calP'}^\heart\rightarrow\Model_\calP^\heart$ preserving reflexive coequalizers. By left Kan extension of the composite $f\colon\calP'\rightarrow\Model_\calP^\heart\subset\Model_{\calP}$, we obtain a functor $f_!\colon \Model_{\calP'}\rightarrow\Model_\calP$ preserving geometric realizations such that $\pi_0 f_! X = \ol{F} X$ for any $X\in\Model_{\calP'}^\heart$.

\begin{prop}\label{prop:leftderivedfunctors}
Fix notation as above. Fix $X'\in\Model_{\calP'}$, and choose some $X_\bullet'\in\s\Model_{\calP'}^\heart$ modeling $X$. Choose a simplicial object $P'_\bullet$ of $\calP'$ together with a weak equivalence $h(P_\bullet')\rightarrow X_\bullet'$. Then $f_! X'$ is modeled by $f P'_\bullet$.
\end{prop}
\begin{proof}
By \cref{thm:rigidification}, to say that $X_\bullet'$ models $X$ is to say we have chosen an identification $\colim_{n\in\Delta^\op}X_n'= X'$ in $\Model_{\calP'}$, and to say $h(P'_\bullet)\rightarrow X_\bullet'$ is a weak equivalence is to say it induces a weak equivalence $\colim_{n\in\Delta^\op}h(P_n')\simeq \colim_{n\in\Delta^\op}X_n'= X'$ in $\Model_{\calP'}$. By definition of $f_!$, we learn
\[
f_! X' \simeq f_! \colim_{n\in\Delta^\op}h(P_n')\simeq\colim_{n\in\Delta^\op}f P_n',
\]
and the result follows as $\colim_{n\in\Delta^\op}f P_n'$ is modeled by $f P_\bullet'$.
\end{proof}

This justifies writing $\bbL \ol{F} = f_! \colon \Model_{\calP'}\rightarrow\Model_{\calP}$ and calling it the total left-derived functor of $\ol{F}$, for by \cref{prop:leftderivedfunctors} this is equivalent to any other correct definition of $\bbL \ol{F}$.

\section{Loop theories}\label{sec:looptheories}

This section covers some generalities of loop theories. The basic example is the category $\calP = \Mod_R^\free$ of free $R$-modules for some $\bbA_\infty$-ring $R$; here we allow ``free $R$-module'' to include suspensions and desuspensions of $R$. If $M$ is an $R$-module, then $h(M)\in\Model_\calP$ lives in the full subcategory $\Model_\calP^\Omega$ consisting of those $X$ with the additional property that $X(\Sigma F)\simeq \Omega X(F)$, and this turns out be a full characterization, i.e.\ there is an equivalence $\Mod_R\simeq\Model_\calP^\Omega$; a particular case of this appears in \cite[Proposition 4.2.5]{hopkinslurie2017brauer}. Loop theories axiomatize the general situation.

After giving some definitions and fixing some notation in \cref{ssec:looptheories}, in \cref{ssec:spiral} we verify that the spiral sequence, as interpreted in \cite{pstragowski2023moduli}, holds equally well in our setting; this is the main tool for relating $\Model_\calP^\Omega$ to $\Model_{\h\calP}$. In \cref{ssec:constructingexamples}, we record some tools for writing categories $\calM$ as $\Model_\calP^\Omega$ for some $\calP\subset\calM$, and for describing the categories $\Model_{\h\calP}^\heart$.

\subsection{Definitions and notation}\label{ssec:looptheories}
Fix a theory $\calP$.

\begin{defn}\label{def:looptheories}
The theory $\calP$ is:
\begin{enumerate}
\item A \textit{loop theory} if for all finite wedges of spheres $F$ and all $P\in\calP$, the tensor $F\otimes P = \colim_{x\in F}P$ exists in $\calP$;
\item A \textit{pointed loop theory} if moreover $\calP$ is pointed and admits suspensions;
\item An \textit{additive loop theory} if it is pointed and additive;
\item A \textit{stable loop theory} if it is pointed and $\Sigma\colon\calP\rightarrow\calP$ is an equivalence.
\tqed
\end{enumerate}
\end{defn}

\begin{rmk}
Each case of \cref{def:looptheories} refines the previous; the final implication is not obvious, but follows from \cref{lem:stableisstable}.
\tqed
\end{rmk}

Suppose now that $\calP$ is a loop theory, and define the subcategory $\Model_\calP^\Omega\subset\Model_{\calP}$ of \textit{loop models} to consist of those $X$ satisfying the additional condition that $X(F\otimes P)\simeq X(P)^{F}$ for all $P\in\calP$ and finite wedge of spheres $F$. In other words, $\Model_\calP^\Omega$ is the full subcategory of $\Model_{\calP}$ local with respect to $F\otimes h(P)\rightarrow h(F\otimes P)$ for all $P\in\calP$ and finite wedge of spheres $F$. Because we are assuming that $\calP$ is bounded, as laid out in \cref{rmk:bounded}, we obtain the following.

\begin{lemma}\label{lem:looploc}
The category $\Model_\calP^\Omega$ is an accessible localization of $\Model_{\calP}$. In particular, it is presentable.
\qed
\end{lemma}

We will generally write $L\colon \Model_\calP\rightarrow\Model_\calP^\Omega$ for the localization.

\begin{rmk}\label{rmk:func2}
Given loop theories $\calP$ and $\calP'$, any functor $F\colon \Model_\calP^\Omega\rightarrow\Model_{\calP'}^\Omega$ preserving geometric realizations, and even some which do not, may be recovered as the left Kan extension of its restriction to $\calP$ along the inclusion $\calP\subset\Model_\calP^\Omega$. Moreover, if we write $f\colon \calP\rightarrow\Model_{\calP'}$ for the restriction of $F$ to $\calP$, then the diagram
\begin{center}\begin{tikzcd}
\Model_\calP^\Omega\ar[r,"F"]&\Model_{\calP'}^\Omega\\
\Model_\calP\ar[r,"f_!"]\ar[u,"L"]&\Model_{\calP'}\ar[u,"L"]
\end{tikzcd}\end{center}
commutes; indeed, all the functors involved preserve geometric realizations and both composites agree on $\calP\subset\Model_\calP$. In particular, if $X\in \Model_{\calP}^\Omega\subset\Model_\calP$ then $F(X) = Lf_!(X)$, and conversely any functor $f\colon \calP\rightarrow\Model_{\calP'}^\Omega\subset\Model_\calP$ extends to a diagram as above. A good special case is obtained when $f$ restricts to a functor $\calP\rightarrow\calP'$ which preserves coproducts and tensors by finite wedges of spheres. In this case, restriction $f^\ast\colon \Psh(\calP')\rightarrow\Psh(\calP)$ itself restricts to a functor $f^*\colon \Model_{\calP'}^\Omega\rightarrow\Model_\calP^\Omega$ which is right adjoint to $F$.
\tqed
\end{rmk}

\begin{rmk}
Although we do not know a good description of the localization $L\colon \Model_\calP\rightarrow\Model_\calP^\Omega$ in general, more can be said in particular cases; see in particular \cref{thm:pialg} and \cref{thm:stablelocalization}, as well as \cref{lem:loccommute}.
\tqed
\end{rmk}

Observe that $\calP$ is tensored not just over finite wedges of spheres, but also finite products thereof. Indeed, if $F$ and $F'$ are finite wedges of spheres and $P\in\calP$, then we may identify $(F\times F')\otimes P\simeq F\otimes (F'\otimes P)$, which lives in $\calP$. In general, if $F$ is a space for which constant colimits over $F$ exist in $\calP$, then for $X\in\Model_\calP$ we write $X_F$ for the model of $\calP$ defined by $X_F(P) = X(F\otimes P)$. There are canonical maps $X_F\rightarrow X^F$, and the condition that $X\in\Model_\calP^\Omega$ is equivalent to the condition that these maps be equivalences for all finite wedges of spheres $F$. It turns out to be sufficient to verify this when $F = S^1$.

\begin{lemma}\label{lem:pbwedge}
Fix a coCartesian diagram
\begin{center}\begin{tikzcd}
F_1&F_2\ar[l]\\
F_3\ar[u]&F_4\ar[u]\ar[l]
\end{tikzcd}\end{center}
of finite products of finite wedges of spheres, and suppose that the map $h(P)_{F_2}\rightarrow h(P)_{F_4}$ is a surjection on path components for $P\in\calP$. Then for any $X\in\Model_\calP$, the square
\begin{center}\begin{tikzcd}
X_{F_1}\ar[d]\ar[r]&X_{F_2}\ar[d]\\
X_{F_3}\ar[r]&X_{F_4}
\end{tikzcd}\end{center}
is Cartesian. In particular, the cogroup structure on $S^n$ gives maps
\[
X_{S^n}\rightarrow X_{S^n\vee S^n}\simeq X_{S^n}\times_X X_{S^n}
\]
making $X_{S^n}$ into a group object in $\Model_{\calP}/X$ for $n\geq 1$.
\end{lemma}
\begin{proof}
If $X$ is representable, or more generally if $X\in\Model_\calP^\Omega$, then this is clear. In general, the claim follows by writing $X$ as a geometric realization of representables and appealing to \cref{lem:cartrel}.
\end{proof}

\begin{prop}\label{lem:loopequiv}
Fix $X\in\Model_{\calP}$. Then the following are equivalent:
\begin{enumerate}
\item $X\in\Model_\calP^\Omega$;
\item The map $X_{S^1}\rightarrow X^{S^1}$ is an equivalence.
\end{enumerate}
\end{prop}
\begin{proof}
The implication $(1)\Rightarrow (2)$ is clear, so suppose conversely that $X_{S^1}\rightarrow X^{S^1}$ is an equivalence. We must show that $X_F\rightarrow X^F$ is an equivalence when $F$ is any finite wedge of spheres. By \cref{lem:pbwedge}, there are natural equivalences $X_{F'\vee F''}\simeq X_{F'}\times_X X_{F''}$, so we can reduce to $F = S^n$. For $n=0$, this is a consequence of the fact that $X\in\Model_{\calP}$, and for $n=1$ this is what we have assumed. For $n\geq 2$, that $X_{S^n}\rightarrow X^{S^n}$ is an equivalence follows from an inductive argument using Cartesian squares
\begin{center}\begin{tikzcd}
X_{S^{n+1}}\ar[r]\ar[d]&(X_{S^n})_{S^1}\ar[d]\\
X\ar[r]&X_{S^n}\times_X X_{S^1}
\end{tikzcd}\end{center}
for $n\geq 1$. These may be obtained by applying \cref{lem:pbwedge} to the cofibering $S^n\vee S^1\rightarrow S^n\times S^1\rightarrow S^{n+1}$ and identifying $X_{S^n\times S^1}\simeq(X_{S^n})_{S^1}$ and $X_{S^n\vee S^1}\simeq X_{S^n}\times_X X_{S^1}$; the relevant surjectivity hypothesis is satisfied by \cref{lem:herdfacts}(3).
\end{proof}

We end this subsection by introducing some additional notation. When $\calP$ is pointed, write $X_{\Sigma^n}$ for the presheaf $X_{\Sigma^n}(P) = X(\Sigma^n P)$. If, for instance, $\calP$ is additive, then one may split $X_{S^n}\simeq X \times X_{\Sigma^n}$, so in particular $\Model_\calP^\Omega$ consists of those $X\in\Model_{\calP}$ such that $X_{\Sigma}\simeq\Omega X$.

The functor $P\mapsto S^n\otimes P$ descends to a functor on $\h\calP$; for $X\in\Model_{\h\calP}$, write $X\langle n \rangle$ for the restriction of $X$ along this functor. Thus $\pi_0 (X_{S^n}) = (\pi_0 X)\langle n \rangle$ for $X\in\Model_{\calP}$. Similarly, when $\calP$ is pointed, write $X[n]$ for the restriction of $X$ along the functor on $\h\calP$ obtained from $\Sigma^n\colon\calP\rightarrow\calP$. We point out that these constructions are not intrinsic to the theory $\h\calP$, but rely on extra structure coming from $\calP$.

\subsection{The spiral}\label{ssec:spiral}

Let $\calP$ be a loop theory, and write $\tau\colon\calP\rightarrow\h\calP $ for the canonical map to its homotopy category. To connect $\Model_\calP^\Omega$ with $\Model_{\h\calP}$, we will need to understand $\tau_!\colon\Model_{\calP}\rightarrow\Model_{\h\calP}$. This understanding is achieved via the following.

\begin{theorem}[{\cite[Section 3.5]{pstragowski2023moduli}}]\label{thm:spiral}
For $X\in\Model_{\calP}$,
\begin{enumerate}
\item The map $X\rightarrow\tau^\ast\tau_! X$ is a $\pi_0$-equivalence;
\item There is a natural Cartesian square
\begin{center}\begin{tikzcd}
B_X X_{S^1}\ar[d]\ar[r]&X\ar[d]\\
X\ar[r]&\tau^\ast\tau_! X
\end{tikzcd},\end{center}
where $B_X X_{S^1}$ is the delooping of $X_{S^1}$ in the slice category $\Model_{\calP}/X$.
\end{enumerate}
\end{theorem}
\begin{proof}
When $X = h(P)$ with $P\in\calP$, as $\tau_! h(P) = h(\tau P)$ it follows that $\tau^\ast\tau_! X = \pi_0 X$. In this case, $X\rightarrow \tau^\ast\tau_! X$ is certainly a $\pi_0$-equivalence, and the above square becomes
\begin{center}\begin{tikzcd}
B_X X^{S^1}\ar[r]\ar[d]&X\ar[d]\\
X\ar[r]&\pi_0 X
\end{tikzcd},\end{center}
which is Cartesian. Next, observe that the terms in the original square, as well as the property of $X\rightarrow \tau^\ast\tau_! X$ being a $\pi_0$-equivalence, are compatible with the formation of geometric realizations. By writing $X$ as a geometric realization of representables, we may conclude with an application of \cref{lem:cartrel}.
\end{proof}

The following is sufficient for many applications.

\begin{cor}\label{cor:spiraldiscrete}
For $X\in\Model_{\calP}$, the map $\tau_! X\rightarrow \pi_0 X$ is an equivalence if and only if $X\in\Model_\calP^\Omega$.
\end{cor}
\begin{proof}
Fix $X\in\Model_{\calP}$, and consider the cube
\begin{center}\begin{tikzcd}
B_XX_{S^1}\ar[rr]\ar[dd]\ar[dr]&&X\ar[dr]\ar[dd]\\
&B_XX^{S^1}\ar[rr]\ar[dd]&&X\ar[dd]\\
X\ar[dr]\ar[rr]&&\tau^\ast\tau_! X\ar[dr]\\
&X\ar[rr]&&\pi_0 X
\end{tikzcd}\end{center}
in which the front and back faces are Cartesian. If $\tau^\ast \tau_! X \rightarrow \pi_0 X$ is an equivalence, then $B_X X_{S^1}\rightarrow B_X X^{S^1}$ must be an equivalence. Conversely, if $B_X X_{S^1}\rightarrow B_X X^{S^1}$ is an equivalence, then as $\tau^\ast\tau_! X\rightarrow \pi_0 X$ is a $\pi_0$-equivalence, the right square is Cartesian, and this implies that $\tau^\ast\tau_! X \simeq \pi_0 X$.
\end{proof}

\subsection{Producing examples}\label{ssec:constructingexamples}

This subsection is concerned with producing and identifying categories of the form $\Model_\calP^\Omega$, as well as their associated algebraic categories $\Model_\calP^\heart\simeq\Model_{\h\calP}^\heart$. See \cref{ssec:examples} for some explicit examples.

A simple class of examples is given by the following observation: if $\calP$ is a discrete theory, then $\calP$ is a loop theory with $F\otimes P\simeq (\pi_0 F) \otimes P$ for $F$ a finite wedge of spheres and $P\in\calP$. In this case, $\pi_0\colon \Model_\calP^\Omega\rightarrow\Model_\calP^\heart$ is an equivalence. 

More interesting examples come from loop theories with more homotopical structure. In \cite[Proposition 4.3]{pstragowski2023moduli}, the following example is given: if $\calP\subset\Gpd_\infty^\ast$ is the full subcategory of wedges of positive-dimensional spheres, then $\Model_\calP^\Omega$ is the category of pointed connected $\infty$-groupoids, and $\Model_\calP^\heart$ is the category of $\Pi$-algebras. This admits the following generalization.

Let $\Sph^+\subset\Gpd_\infty$ denote the full subcategory generated by the positive-dimensional spheres under finite products.

\begin{theorem}\label{thm:pialg}
Let $\calC$ be a theory, and let $\calP\subset\Model_{\calC}$ be the full subcategory generated under coproducts by objects of the form $F\otimes h(P)$ with $F\in\Sph^+$ and $P\in\calC$. Then $\calP$ is a loop theory, and the following hold.
\begin{enumerate}
\item Restriction along the inclusion $i\colon \calC\subset\calP$ is left adjoint to the restricted Yoneda embedding $h\colon\Model_\calC\rightarrow\Model_\calP$;
\item $h$ is fully faithful, and restricts to an equivalence $\Model_\calC\simeq\Model_\calP^\Omega$.
\end{enumerate}
\end{theorem}
\begin{proof}
Given $C\in\calC$ and $F\in\Sph^+$, we shall write $F\boxtimes h(C)$ for $F\otimes h(C)$ considered as an object of $\Model_\calP$; this notation will only appear in this proof.

(1)~~As $\Model_\calC$ is cocomplete, $h$ admits a left adjoint which may be identified as the functor obtained from the inclusion $\calP\subset\Model_\calC$ by left Kan extension. As $i^\ast$ preserves geometric realizations, it is thus sufficient to verify that $i^\ast(F\boxtimes h(C)) = F \otimes h(C)$ for any $C\in\calC$ and $F\in\Sph^+$. This is clear, for if $C'\in\calC$, then $\Map_\calC(h(C'),i^\ast(F\boxtimes h(C)))\simeq\Map_\calP(h(C'),F\boxtimes h(C))\simeq \Map_\calC(h(C),F\otimes h(C))$ by definition.

(2)~~Consider the adjunction $h:\Model_\calC\rightarrow\Model_\calP^\Omega:i^\ast$ obtained by restriction from that given by (1). Clearly the counit $i^\ast(h(X))\simeq X$ is an equivalence for $X\in\Model_\calC$, so it suffices to verify that the counit $Y\rightarrow h(i^\ast(Y))$ is an equivalence for $Y\in\Model_\calP^\Omega$. Indeed, if $F\in\Sph^+$ and $C\in\calC$, then $Y(F\boxtimes C)\simeq Y(C)^F\simeq (h(i^\ast(Y))(C)^F\simeq (h(i^\ast(Y)))(F\boxtimes C)$.
\end{proof}

To recover \cite[Proposition 4.3]{pstragowski2023moduli}, one takes $\calC$ to be the theory of groups, or equivalently the full subcategory of pointed spaces consisting of wedges of $S^1$, so that $\Model_\calC$ is equivalent to the category of pointed connected spaces. This restriction to pointed connected spaces is necessary in order for $\calC$ to satisfy the Mal'cev condition.

We are particularly interested in examples arising from spectral algebra, which cannot be of this form. To get at these, we begin by considering the stable case.

\begin{lemma}\label{lem:stableisstable}
If $\calP$ is a stable loop theory, then $\Model_\calP^\Omega$ is a stable category.
\end{lemma}
\begin{proof}
This is a consequence of \cite[Corollary 1.4.2.27]{lurie2017higheralgebra}, as precomposition with the equivalence $\Sigma\colon\calP\rightarrow\calP$ agrees with $\Omega$ on $\Model_\calP^\Omega$.
\end{proof}

The following can be regarded as a generalization of \cite[Proposition 4.2.5]{hopkinslurie2017brauer}.

\begin{theorem}\label{thm:stabletheory}
Let $\calM$ be a stable category admitting small colimits, and let $\calP\subset\calM$ be a full subcategory which is a stable loop theory closed under coproducts and suspensions in $\calM$. Then
\begin{enumerate}
\item The restricted Yoneda embedding $h\colon\calM\rightarrow\Psh(\calP)$ is fully faithful upon restriction to the thick subcategory generated by $\calP$;
\item The restricted Yoneda embedding yields an equivalence $\calM\simeq\Model_\calP^\Omega$ provided either of the following is satisfied:
\begin{enumerate}
\item[(a)]The restricted Yoneda embedding is conservative and $\calP$ is generated under coproducts by objects which are compact in $\calM$;
\item[(b)]There is a fixed finite diagram $\calJ$ such that every object of $\calM$ may be written as a $\calJ$-shaped colimit of objects of $\calP$.
\end{enumerate}
\end{enumerate}
\end{theorem}
\begin{proof}
(1)~~ Write $k\colon\calM\rightarrow\Model_\calP^\Omega$. As $k$ is a limit-preserving functor between stable categories, $k$ preserves finite colimits. Fix $Y\in\calM$. Then the collection of $X\in\calM$ such that
\[
\Map_\calM(X,Y)\rightarrow\Map_{\calP}(k(X),k(Y))
\]
is an equivalence is a thick subcategory of $\calM$ containing $\calP$, proving (1).

(2)~~ First we claim that in either case $k$ preserves all colimits. In case (a), $k$ preserves filtered colimits, so this follows from preservation of finite colimits. In case (b), it is sufficient to verify that $k$ preserves coproducts. Given a collection $\{M_i:i\in I\}$ of objects of $\calM$, we may write $M_i\simeq\colim_{j\in\calJ}P_{i,j}$ with $P_{i,j}\in\calP$, and so compute
\[
k\left(\bigoplus_{i\in I}M_i\right)\simeq k\left(\colim_{j\in\calJ}\bigoplus_{i\in I}P_{i,j}\right)\simeq\colim_{j\in\calJ}\bigoplus_{i\in I}k\left(P_{i,j}\right)\simeq\bigoplus_{i\in I}k\left(M_i\right),
\]
as $k$ preserves all finite colimits and all small coproducts of objects of $\calP$.

Now as $\calM$ admits small colimits, $k$ admits a left adjoint $L$, and the fact that $k$ preserves small colimits implies that $X\simeq kLX$ for $X\in\Model_\calP^\Omega$. It is then sufficient to verify that $L k M \simeq M$ for $M\in\calM$. This is immediate in case (b), and in case (a) follows as $k$ is conservative and $k M\simeq k L k M$.
\end{proof}

If $\calP$ is stable, then as $\Model_\calP^\Omega$ is stable, $\calP$ is additive. The structure of $\Model_{\h\calP}^\heart$ in this case can be understood through the following.

\begin{prop}\label{prop:additiveabelian}
\hphantom{blank}
\begin{enumerate}
\item If $\calP$ is a discrete additive theory, then $\Model_\calP^\heart$ is a complete and cocomplete abelian category with enough projectives;
\item If $\calA$ is a cocomplete abelian category and $\calP\subset\calA$ is a full subcategory consisting of projective objects and closed under coproducts such that every $M\in\calA$ admits a projective resolution by objects of $\calP$, then $\calA\simeq\Model_\calP^\heart$.
\end{enumerate}
\end{prop}

\begin{proof}
(1)~~ Observe that as $\calP$ is additive, every model $X\colon \calP^\op\rightarrow\Set$ admits an essentially unique lift through $\Ab\rightarrow\Set$. Thus $\Model_\calP^\heart$ is equivalent to the category of $\Ab$-valued models of $\calP$, which is a full subcategory of $\Psh(\calP,\Ab)$ closed under finite limits and colimits. This implies that $\Model_\calP^\heart$ is abelian, and that it is complete and cocomplete with enough projectives follows from \cref{cor:modelsresolved}.

(2)~~ Fix such $\calA$ and $\calP\subset\calA$. As $\calA$ admits small colimits, the restricted Yoneda embedding $h\colon \calA\rightarrow\Model_\calP^\heart$ admits a left adjoint $L\colon\Model_\calP^\heart\rightarrow\calA$. As $\Model_\calP^\heart$ is the free  $1$-categorical cocompletion of $\calP$ under reflexive coequalizers, we must only verify that $h$ is conservative, which follows from the assumption that every $M\in\calA$ is resolved by objects of $\calP$.
\end{proof}

We now move on to methods that allow for the production of unstable examples.

\begin{prop}\label{prop:monadicforget}
Let $\calP$ and $\calP'$ be theories, and let $f\colon\calP'\rightarrow\calP$ be an essentially surjective coproduct-preserving functor.
\begin{enumerate}
\item Restriction $f^\ast\colon\Model_{\calP}\rightarrow\Model_{\calP'}$ is the forgetful functor of a monadic adjunction;
\item If $\calP$ and $\calP'$ are loop theories and $f$ preserves tensors by finite wedges of spheres, then restriction $f^\ast\colon\Model_\calP^\Omega\rightarrow\Model_{\calP'}^\Omega$ is the forgetful functor of a monadic adjunction.
\end{enumerate}
\end{prop}
\begin{proof}
Observe that both instances of $f^\ast$ are right adjoints, with left adjoints left Kan extending $f$ (see \cref{rmk:func1} and \cref{rmk:func2}). Moreover, the assumption that $f$ is essentially surjective implies that each $f^\ast$ is conservative. By Beck's monadicity theorem \cite[Theorem 4.7.3.5]{lurie2017higheralgebra}, it is sufficient to verify that $f^\ast$ creates $f^\ast$-split geometric realizations. Indeed, split geometric realizations are in particular pointwise geometric realizations, so this follows from the fact that $f$ is essentially surjective.
\end{proof}

\begin{prop}\label{prop:monadtheory}
Let $\calP$ be a theory, and let $t_!\colon\Model_{\calP}\rightarrow\Model_{\calP}$ be a monad which preserves geometric realizations, so that $t_!$ is the left Kan extension of its restriction $t$ to $\calP$. Let $T\calP\subset\Alg_{t_!}$ be the full subcategory spanned by objects of the form $t(P)$ for $P\in\calP$. Then $\Alg_{t_!}\simeq\Model_{T\calP}$.
\end{prop}
\begin{proof}
This follows by an application of \cref{thm:recognitiontheorem} to the functor $\Model_{T\calP}\rightarrow\Alg_{t_!}$ obtained by left Kan extension from the inclusion $T\calP\subset\Alg_{t_!}$.
\end{proof}

\begin{theorem}\label{thm:recmon}
Let $\calP$ be a loop theory, and let $T$ be an accessible monad on $\Model_\calP^\Omega$. Let $t=Th$ denote the restriction of $T$ to $\calP$, and let $T\calP\subset\Alg_T$ be the full subcategory spanned by objects of the form $t(P)$ for $P\in\calP$. Write $L\colon\Model_{\calP}\rightarrow\Model_\calP^\Omega$ for the localization. Then the restricted Yoneda embedding yields an equivalence $\Alg_T\simeq\Model_{T\calP}^\Omega$ if and only if the canonical map $L t_! X\rightarrow T X$ is an equivalence for $X\in\Model_\calP^\Omega$. In particular, $\Alg_T\simeq\Model_{T\calP}^\Omega$ if $T$ preserves geometric realizations.
\end{theorem}
\begin{proof}
Observe that there is a factorization of forgetful functors 
\begin{center}\begin{tikzcd}
\Alg_T\ar[r,"h"]&\Model_{T\calP}^\Omega\ar[r,"t^\ast"]& \Model_\calP^\Omega
\end{tikzcd}.\end{center}
By \cref{prop:monadicforget}, both $\Alg_T\rightarrow\Model_\calP^\Omega$ and $\Model_{T\calP}^\Omega\rightarrow\Model_\calP^\Omega$ are the forgetful functors of monadic adjunctions, with associated monads $T$ and $L t_!$. The above factorization gives rise to a map $L t_!\rightarrow T$ of monads which is an equivalence if and only if $\Alg_T\simeq\Model_{T\calP}^\Omega$.
\end{proof}

In the situation of \cref{thm:recmon}, we would like to identify the algebraic category $\Model_{T\calP}^\heart$. To that end, we have the following; compare \cite[Section 4]{rezk2009congruence}.

\begin{prop}\label{prop:algebraicapproximations}
Let $\calP$ be a theory, and fix a monad on $\Model_{\calP}$ which preserves geometric realizations, and so has underlying functor of the form $t_!$ for some $t\colon\calP\rightarrow\Model_{\calP}$. Let $T\calP\subset\Alg_{t_!}$ be the full subcategory spanned by objects of the form $t(P)$ for $P\in\calP$, so that $\Alg_{t_!}\simeq\Model_{T\calP}$. Then
\begin{enumerate}
\item $t^\ast \colon \Model_{T\calP}^\heart\rightarrow\Model_\calP^\heart$ is the forgetful functor of a monadic adjunction; write $\bbT$ for the associated monad on $\Model_\calP^\heart$.
\item The monad $\bbT$ is determined by natural isomorphisms $\bbT(\pi_0 X)\simeq\pi_0 t_! X$ of $t_!$-algebras for $X\in\Model_{\calP}$.
\item If $\calP$ is a loop theory and $t_!$ is obtained from a monad $T$ on $\Model_\calP^\Omega$, then $\bbT$ can instead be described in terms of $T$ in the following manner:
\begin{enumerate}
\item[(a)]$\bbT$ preserves reflexive coequalizers;
\item[(b)]There are natural maps $\bbT(\pi_0 X)\rightarrow\pi_0 T X$ for $X\in\Model_\calP^\Omega$ which are isomorphisms when $X = h(P)$ for some $P\in\calP$;
\item[(c)]The diagrams
\begin{center}\begin{tikzcd}
\bbT\bbT(\pi_0 X)\ar[r]\ar[d,"\mu"]&\bbT(\pi_0 TX)\ar[r]&\pi_0 TTX\ar[d,"\mu"]\\
\bbT(\pi_0 X)\ar[rr]&&\pi_0 TX
\end{tikzcd}
\begin{tikzcd}
\pi_0 X\ar[dr,"\eta"']\ar[r,"\eta"]&\bbT(\pi_0 X)\ar[d]\\
&\pi_0 T X
\end{tikzcd}\end{center}
commute.
\end{enumerate}
\end{enumerate}
\end{prop}
\begin{proof}
(1)~~ This follows immediately from the crude monadicity theorem.

(2)~~ First observe that $\Model_{T\calP}^\heart$ can be identified as the category of discrete $t_!$-algebras. In particular, if $X$ is a $t_!$-algebra, then $\pi_0 X$ is a $t_!$-algebra. As a consequence, for $X\in\Model_{\calP}$ the map $X\rightarrow t_! X\rightarrow \pi_0 t_! X$ extends uniquely to a map $\bbT(\pi_0 X)\rightarrow\pi_0 t_! X$ of $t_!$-algebras, which is evidently an isomorphism when $X=h(P)$ with $P\in\calP$. As this is a natural transformation of functors which preserve geometric realizations computed in $\Model_\calP^\heart$, it is a natural isomorphism, verifying (2).

(3)~~ As there are maps $t_! X\rightarrow T X$ for $X\in\Model_\calP^\Omega$, we can take as our natural transformation the map $\bbT(\pi_0 X)\simeq\pi_0 t_! X\rightarrow \pi_0 T X$; this has the indicated properties as $t(P) = Th(P)$ by assumption, and these evidently determine $\bbT$.
\end{proof}

In part (3) of \cref{prop:algebraicapproximations}, the assumption that $t_!$ is obtained from a monad $T$ on $\Model_\calP^\Omega$ is, by \cref{prop:monadicforget}, equivalent to the assumption that $t(P)\in\Model_\calP^\Omega$ for $P\in\calP$. In \cref{sec:completions}, we will see situations where this can fail; examples where $\Alg_T\simeq\Model_{T\calP}^\Omega$ even when $T$ does not preserve geometric realizations; and examples where the hypotheses of \cref{thm:stabletheory} are not satisfied yet nonetheless $\calM\simeq\Model_\calP^\Omega$.

\section{Stable loop theories}\label{sec:stable}

This section concerns those loop theories $\calP$ that give rise to stable categories. In the stable setting, it is natural to consider the category $\LMod_\calP$ of spectrum-valued models, and corresponding full subcategory $\LMod_\calP^\Omega\subset\LMod_\calP$ of spectrum-valued models that preserve loops. These categories behave somewhat differently from their unstable versions; the most important aspect is that the inclusion $\LMod_\calP^\Omega\subset\LMod_\calP$ has an explicitly describable left adjoint, which we give in \cref{thm:stablelocalization}. 

To illustrate what can be done in this context, we construct in \cref{ssec:lderss} a spectral sequence for computing with suitable functors into $\Model_\calP^\Omega$, and verify in \cref{ssec:monoidal} that it is multiplicative when one would expect it to be; this relies on the material of \cref{sec:app}. In \cref{ssec:stabletowers}, we give a spectral sequence for computing mapping spectra; this can be seen in part as a warmup for the more involved obstruction theory available in unstable contexts given in \cref{ssec:mappingspaceobstructions}, although it is not quite a special case of the latter.

\subsection{Additive and stable theories}\label{ssec:stabletheories}

Fix a theory $\calP$, and suppose moreover that $\calP$ is additive. In this case, it turns out that $\Model_{\calP}$ is close to being stable; specifically, it is prestable in the sense of \cite[Appendix C]{lurie2018spectral}. The properties we need are summarized below. 

Write $\Sp$ for the category of spectra, and write $\LMod_\calP$ for the category of $\Sp$-valued models of $\calP$ and $\LMod_\calP^\cn$ for the category of $\Sp_{\geq 0}$-valued models of $\calP$, where $\Sp_{\geq 0}$ is the category of connective spectra. These categories are all presentable.

\begin{lemma}\label{prop:stabilizeadditive}
Let $\calP$ be an additive theory. Then
\begin{enumerate}
\item Postcomposition with $\Omega^\infty$ yields an equivalence $\LMod_\calP^\cn\simeq\Model_\calP$;
\item The embedding $\Model_{\calP}\simeq\LMod_\calP^\cn\subset\LMod_\calP$ realizes $\LMod_\calP$ as the category of spectrum objects of $\Model_{\calP}$.
\end{enumerate}
\end{lemma}
\begin{proof}
If $\calP$ is finitary, then one may appeal directly to \cite[Remark C.1.5.9]{lurie2018spectral}; in general, one may appeal to the same by use of the embedding $\Model_\calP\subset\Psh^{\Pi_\omega}(\calP)$. More directly, using the description of colimits in $\Model_{\calP}$ given by \cref{prop:modelsreflective}, it is seen that $\Model_{\calP}$ is additive, from which it follows, as in the proof of \cite[Proposition C.1.5.7]{lurie2018spectral}, that $\Model_{\calP}\simeq\LMod_\calP^\cn$; the second claim follows in turn as $\Omega\colon\Model_{\calP}\rightarrow\Model_{\calP}$ is computed pointwise.
\end{proof}

We may at times abuse notation by identifying $\Model_\calP$ with $\LMod_\calP^\cn$ as a subcategory of $\LMod_\calP$.

Fix now a stable loop theory $\calP$; in particular $\calP$ is additive. Write $\LMod_\calP^\Omega\subset\LMod_\calP$ for the full subcategory of objects $X$ such that $X_\Sigma\simeq\Omega X$. This is distinct from the image of $\Model_\calP^\Omega$ under $\Model_{\calP}\simeq\LMod_\calP^\cn\subset\LMod_\calP$, as objects of $\LMod_\calP^\Omega\subset\Psh(\calP,\Sp)$ will generally take values in nonconnective spectra.

\begin{theorem}\label{thm:stablelocalization}
The inclusion $\LMod_\calP^\Omega\subset\LMod_\calP$ is the inclusion of a reflective subcategory, and the associated localization $L$ on $\LMod_\calP$ is given by
\[
L X = \colim_{n\rightarrow\infty}\Sigma^{-n} X_{\Sigma^{-n}}.
\]
\end{theorem}
\begin{proof}
As both $X\mapsto \Sigma^{-1} X$ and $X\mapsto X_{\Sigma^{-1}}$ are automorphisms of $\LMod_\calP$, for $X\in\LMod_\calP$ and $Y\in\LMod_\calP^\Omega$ there are equivalences
\[
\Map_\calP(X,Y)\simeq\Map_\calP(\Sigma^{-1} X_{\Sigma^{-1}},\Sigma^{-1}Y_{\Sigma^{-1}})\simeq\Map_\calP(\Sigma^{-1}X_{\Sigma^{-1}},Y),
\]
with composite given by restriction along $\Sigma^{-1}X_{\Sigma^{-1}}\rightarrow X$. Thus if we define $L'X = \colim_{n\rightarrow\infty} \Sigma^{-n}X_{\Sigma^{-n}}$, then $\Map_\calP(X,Y)\simeq\Map_\calP(L' X,Y)$, and to show $LX\simeq L'X$ we must only verify that $L' X \in \LMod_\calP^\Omega$. As $\LMod_\calP$ is stable, finite limits commute past arbitrary colimits. Thus we may compute
\begin{align*}
\Omega L' X \simeq \Omega\colim_{n\rightarrow\infty} \Sigma^{-n}X_{\Sigma^{-n}} \simeq \colim_{n\rightarrow\infty}\Sigma^{-n-1}X_{\Sigma^{-n}} 
\simeq (\colim_{n\rightarrow\infty}\Sigma^{-n}X_{\Sigma^{-n}})_{\Sigma}\simeq (L'X)_{\Sigma},
\end{align*}
showing $L'X\in\LMod_\calP^\Omega$.
\end{proof}

\begin{rmk}
The inclusion $\LMod_\calP^\Omega\subset\LMod_\calP$ is also the inclusion of a coreflective category, with colocalization $R$ on $\LMod_\calP$ given by
\[
R X = \lim_{n\rightarrow\infty}\Sigma^{n}X_{\Sigma^{n}}.
\]
We will not make use of this.
\tqed
\end{rmk}

\begin{cor}\label{cor:stablelocalization}
When $\calP$ is stable, 
\begin{enumerate}
\item If $X\in\Model_\calP^\Omega\subset\LMod_\calP^\cn\subset\LMod_\calP$, then the tower of \cref{thm:stablelocalization} producing $LX$ is exactly the Whitehead tower of $LX$;
\item Postcomposition with $\Omega^\infty$ yields an equivalence $\LMod_\calP^\Omega\simeq\Model_\calP^\Omega$;
\item The full subcategory $\LMod_\calP^\Omega\subset\LMod_\calP$ is closed under all small limits and colimits;
\item The composite $\Model_{\calP}\simeq\LMod_\calP\rightarrow\LMod_\calP^\Omega\simeq\Model_\calP^\Omega$ is left adjoint to the inclusion $\Model_\calP^\Omega\subset\Model_{\calP}$.
\end{enumerate}
\end{cor}
\begin{proof}
(1)~~ This is clear, as the map $X\rightarrow\Sigma^{-1}X_{\Sigma^{-1}}$ is an equivalence on $(-1)$-connected covers.

(2)~~ Under the identification $\Model_\calP\simeq\LMod_\calP^\cn$, postcomposition with $\Omega^\infty$ is identified with $\tau_{\geq 0}$. Observe that if $X\in\Model_\calP^\Omega$ then $(LX)_{\geq 0}\simeq X$ by (1), and thus $\tau_{\geq 0}$ is essentially surjective. To see that it is fully faithful, fix $X,Y\in\LMod_\calP^\Omega$ and compute
\[
\Map_\calP(X,Y)\simeq\Map_\calP(L(X_{\geq 0}),Y)\simeq\Map_\calP(X_{\geq 0},Y)\simeq\Map_\calP(X_{\geq 0},Y_{\geq 0}).
\]

(3)~~ This is clear.

(4)~~ Observe that if $X\in\Model_{\calP}$ and $Y\in\Model_\calP^\Omega$, then
\[
\Map_\calP(X,Y)\simeq\Map_\calP(X,(LY)_{\geq 0})\simeq\Map_\calP(X,LY)\simeq\Map_\calP(LX,LY),
\]
so the claim follows from (2).
\end{proof}

\begin{warning}
Part (4) of \cref{cor:stablelocalization} does not combine with \cref{thm:stablelocalization} to give an explicit description of the localization $L\colon\Model_\calP\rightarrow\Model_\calP^\Omega$ in general, but it does when $\calP$ is finitary. Here the issue is that in general $\Omega^\infty\colon \LMod_\calP\rightarrow\Model_\calP$ need not preserve filtered colimits.
\tqed
\end{warning}

Any fiber sequence $X\rightarrow Y \rightarrow Z$ in $\Model_{\calP}$ with second map a $\pi_0$-surjection remains a fiber sequence in $\LMod_\calP$. In particular, \cref{thm:spiral} gives such a fiber sequence $B X_\Sigma\rightarrow X\rightarrow\tau^\ast\tau_! X$ in $\Model_{\calP}$, yielding the following.

\begin{lemma}\label{lem:stablespiral}
For $X\in\Model_\calP$, there is a fiber sequence
\[
\Sigma X_\Sigma \rightarrow X\rightarrow\tau^\ast \tau_! X
\]
in $\LMod_\calP$.
\qed
\end{lemma}

\subsection{Left-derived functor spectral sequences}\label{ssec:lderss}

Throughout this subsection, we fix a stable loop theory $\calP$ and arbitrary loop theory $\calP'$.

Let $f\colon\calP'\rightarrow\Model_\calP^\Omega$ be a functor, and extend this to $F\colon\Model_{\calP'}\rightarrow\Model_\calP^\Omega$ by left Kan extension. The composite $\pi_0\circ f\colon \calP'\rightarrow\Model_{\h\calP}^\heart$ factors through $\h\calP$, and we denote the resulting functor as $\ol{f}\colon \h\calP\rightarrow\Model_{\h\calP}^\heart$. Again by left Kan extension we obtain $\ol{F}\colon\Model_{\h\calP'}^\heart\rightarrow\Model_{\h\calP}^\heart$. Recall from \cref{ssec:simplicial} our discussion of the total left-derived functor $\bbL \ol{F}$ defined as $\bbL\ol{F}=\ol{f}_!\colon \Model_{\h\calP'}\rightarrow\Model_{\h\calP}$.

\begin{prop}\label{prop:lder}
The diagram
\begin{center}\begin{tikzcd}
\Model_{\calP'}\ar[d,"\tau_!"]\ar[r,"f_!"]&\Model_{\calP}\ar[d,"\tau_!"]\\
\Model_{\h\calP'}\ar[r,"\bbL \ol{F}"]&\Model_{\h\calP}
\end{tikzcd}\end{center}
commutes.
\end{prop}
\begin{proof}
As all functors involved preserve geometric realizations, it suffices to check that the diagram commutes upon restriction to $\calP'$. This itself follows from \cref{cor:spiraldiscrete} together with the assumption that $f(P')\in\Model_\calP^\Omega$ for $P'\in\calP'$.
\end{proof}

\begin{theorem}\label{thm:quillenss}
Fix notation as above, and fix $R\in\Model_{\calP'}^\Omega$. Then the spectral sequence in $\Model_\calP^\heart$ associated to the tower
\[
F R \simeq\colim_{n\rightarrow\infty} \Sigma^{-n} (f_! R)_{\Sigma^{-n}}
\] 
guaranteed by \cref{thm:stablelocalization} is of signature
\[
E^1_{p,q} = (\bbL_{p+q}\ol{F}\,\pi_0 R)[-p]\Rightarrow(\pi_0 F R)[q],\qquad d^r_{p,q}\colon E^r_{p,q}\rightarrow E^r_{p-r,q-1}.
\]
This spectral sequence converges, for instance, if $\pi_0$ preserves filtered colimits or if $\bbL\ol{F}\pi_0 R$ is truncated.
\end{theorem}
\begin{proof}
By \cref{lem:stablespiral} and \cref{prop:lder}, the tower $F R \simeq\colim_{n\rightarrow\infty} \Sigma^{-n} (f_! R)_{\Sigma^{-n}}$ has layers described by cofiber sequences
\[
\Sigma^{-(n-1)}(f_! R)_{\Sigma^{-(n-1)}}\rightarrow\Sigma^{-n}(f_! R)_{\Sigma^{-n}}\rightarrow \tau^\ast \Sigma^{-n}(\bbL\ol{F}\pi_0 R)[-n].
\]
This gives rise to the indicated spectral sequence in the usual way; we will review the construction and convergence in \cref{ssec:constructss}.
\end{proof}

The method of constructing spectral sequences by analyzing the tower obtained from \cref{thm:stablelocalization} is more general than just that given in \cref{thm:quillenss}. Roughly, given $M\in\Model_\calP^\Omega$, to obtain a tool for computing $\pi_\ast M$ one wants to find some $M'\in\Model_{\calP}$ with $L M ' = M$ such that $\tau_! M'$ is something computable. In the preceding theorem, we had $M = F X$, and took $M' = f_! X$; another simple case is the following.

\begin{ex}
Each of $\Model_{\calP}$, $\Model_\calP^\Omega$, and $\Model_{\h\calP}$ are tensored over $\Sp_{\geq 0}$ by additivity. Denote the resulting tensors by $\otimes_!$, $\otimes$, and $\mathop{\olotimes{}^\bbL}$. These are all compatible, in that if $X\in\Sp_{\geq 0}$ and $M\in\Model_{\calP}$, then
\[
\tau_!(X\otimes_! M) = X \mathop{\olotimes{}^\bbL} \tau_! M,\qquad L(X\otimes_! M) = X\otimes L M.
\]
For $X\in\Sp_{\geq 0}$ and $\Lambda\in\Model_\calP^\heart$, write
\[
H_\ast(X;\Lambda) = \pi_\ast (X\mathop{\olotimes{}^\bbL} \Lambda);
\]
this is a form of ordinary homology. For $M\in\Model_\calP^\Omega$, we obtain an Atiyah-Hirzebruch-type spectral sequence
\[
E^1_{p,q} = H_{p+q}(X;\pi_0 M)[-p]\Rightarrow \pi_q (X\otimes M),\qquad d^r\colon E^r_{p,q}\rightarrow E^r_{p-r,q-1},
\]
by analyzing the tower $X\otimes M = \colim_{n\rightarrow\infty}\Sigma^{-n}(X\otimes_! M)_{\Sigma^{-n}}$.
\tqed
\end{ex}

\subsection{Monoidal matters}\label{ssec:monoidal}

This subsection relies on the material of \cref{sec:app}.

We would like to introduce monoidal properties of the constructions discussed in \cref{ssec:stabletheories} and \cref{ssec:lderss}. Our primary reason for doing so is to introduce pairings into the spectral sequence of \cref{thm:quillenss}. For the sake of completeness, we will work briefly in a more general setting than is necessary for just the production of these pairings, and for this generality we require the theory of $\infty$-operads as developed in \cite{lurie2017higheralgebra}. However, the cases of the general theory necessary for our primary application, \cref{thm:pairingquillen}, are just as easily performed by hand, completely bypassing the theory of $\infty$-operads.

Fix a single-colored $\infty$-operad $\calO$ in the sense of \cite{lurie2017higheralgebra}. We will implicitly use the fact that every symmetric monoidal category canonically inherits the structure of an $\calO$-monoidal category. Say that an $\calO$-monoidal structure on a category $\calD$ \textit{respects} some class of colimits in $\calD$ if for every $n\geq 0$ and $f\in\calO(n)$, the tensor product $\otimes_f$ preserves such colimits in each variable. An $\calO$-monoidal category $\calD$ is said to be \textit{$\calO$-monoidally cocomplete} if it admits small colimits and its $\calO$-monoidal structure respects these. In \cite[Section 2.2.6]{lurie2017higheralgebra}, generalizing \cite{glasman2016day}, it is shown that if $\calC$ is a small $\calO$-monoidal category and $\calD$ is an $\calO$-monoidally cocomplete category, then $\Fun(\calC,\calD)$ admits the structure of an $\calO$-monoidally cocomplete category under Day convolution, informally described as follows: for $n\geq 0$, $f\in\calO(n)$, and $F_1,\ldots,F_n\colon\calC\rightarrow\calD$, the tensor product $\otimes_f(F_1,\ldots,F_n)$ is the left Kan extension of $\otimes_f\circ(F_1\times\cdots\times F_n)\colon\calC^{\times n}\rightarrow\calD^{\times n}\rightarrow\calD$ along $\otimes_f\colon\calC^{\times n}\rightarrow\calC$. Of interest is the case where $\calC$ is the poset $(\bbZ,<)$ with symmetric monoidal structure given by addition, where for towers $X_1,\ldots,X_n$ in $\calD$ we identify
\[
\otimes_f(X_1,\ldots,X_n)(p) = \colim_{p_1+\cdots+p_n\leq p}\otimes_f(X_1(p_1),\ldots,X_n(p_n)).
\]

\begin{defn}\label{defn:monoidalres}
A loop theory $\calP$ is an \textit{$\calO$-monoidal loop theory} if it is equipped with an $\calO$-monoidal structure compatible with coproducts and tensors by finite wedges of spheres.
\tqed
\end{defn}

Fix an $\calO$-monoidal loop theory $\calP$. We obtain by Day convolution, following \cite[Proposition 4.8.1.10]{lurie2017higheralgebra}, $\calO$-monoidal categories, all compatible with colimits, and strong $\calO$-monoidal functors, fitting into the diagram
\begin{center}\begin{tikzcd}
\Model_\calP^\heart&\ar[l]\Model_{\calP}\ar[r]&\Model_\calP^\Omega
\end{tikzcd}.\end{center}
When $\calP$ is in addition stable, we similarly obtain
\begin{center}\begin{tikzcd}
\Model_{\calP}\ar[r]\ar[d]&\Model_\calP^\Omega\ar[d,"\simeq"]\\
\LMod_\calP\ar[r]&\LMod_\calP^\Omega
\end{tikzcd}.\end{center}
There is an evident notion of a general $\calO$-monoidal theory, and if $\calP$ is such then we obtain the same diagrams, only with $\Model_\calP^\Omega$ and $\LMod_\calP^\Omega$ omitted. The only thing we have to say in this level of generality is the following.

\begin{prop}\label{prop:monoidallocalization}
Suppose that $\calP$ is an $\calO$-monoidal stable loop theory. Then the functor $\LMod_{\calP}\rightarrow\Fun(\bbZ,\LMod_\calP)$ sending $X$ to the tower
\[
\cdots\rightarrow \Sigma X_{\Sigma}\rightarrow X\rightarrow \Sigma^{-1}X_{\Sigma ^{-1}}\rightarrow\cdots
\]
is canonically lax $\calO$-monoidal.
\end{prop}
\begin{proof}
This functor preserves colimits, so by the universal property of Day convolution it is sufficient to verify that it is canonically lax $\calO$-monoidal upon restriction to $\calP$. By \cref{cor:stablelocalization}, its restriction to $\calP$ is equivalent to the composite
\[
\calP\rightarrow\LMod_\calP^\Omega\subset\LMod_\calP\xrightarrow{W}\Fun(\bbZ,\LMod_\calP),
\]
where $W$ is the functor sending an object to its Whitehead tower. We conclude by applying \cref{prop:whiteheadmonoidal}.
\end{proof}

We restrict now to the case where $\calO$ is the nonunital $\bbA_2$-operad, i.e.\ where our monoidal structures consist merely of a single pairing subject to no further coherence conditions. This is both the most general and the simplest situation: in the context of pairings of spectral sequences, additional properties such as associativity and commutativity can be verified at the level of homotopy groups, so we need not be concern ourselves with the coherence problems they present. Fix nonunital $\bbA_2$-monoidal loop theories $\calP$ and $\calP'$, and suppose that $\calP$ is stable. Write the associated pairings on $\LMod_\calP^\Omega$ and $\Model_{\calP'}^\Omega$ as $\otimes$, and the associated pairings on $\Model_{\h\calP}^\heart$ and $\Model_{\h\calP'}^\heart$ as $\olotimes$.

Fix a functor $F\colon\Model_{\calP'}^\Omega\rightarrow\Model_\calP^\Omega$ which preserves geometric realizations; from here we will use notation as in \cref{ssec:lderss}. Suppose that $F$ is lax monoidal; equivalently, that we have chosen a natural transformation $f(P')\otimes f(Q')\rightarrow f(P'\otimes Q')$. This gives rise to lax monoidal structures on $f_!$, $\ol{F}$, and $\bbL\ol{F}$.

By \cref{thm:rigidification} and the classic Dold-Kan correspondence \cite[Section 3]{doldpuppe1961homologie}, $\Model_{\h\calP}$ can be modeled as $\s\Model_{\h\calP}^\heart\simeq\Ch^+(\Model_{\h\calP}^\heart)$. Following \cref{prop:leftderivedfunctors}, the pairing on $\Model_{\h\calP}$ induced by that on $\h\calP $ can be modeled by the levelwise pairing on $\s(\h\calP)\subset\s\Model_{\h\calP}^\heart$, and by the Eilenberg-Zilber theorem this is modeled by the standard pairing on $\Ch^+(\h\calP)\subset\Ch^+(\Model_{\h\calP}^\heart)$. To be precise, we choose the pairing on $\Ch^+(\Model_{\h\calP}^\heart)$ given by $(C'\olotimes C'')_p = \bigoplus_{p'+p''}C'_{p'}\olotimes C''_{p''}$, with differential $d(x'\otimes x'') = d(x')\otimes x''+(-1)^{|x'|}x'\otimes d(x'')$. Having made a choice, a pairing $C'\olotimes C''\rightarrow C$ of chain complexes gives K\"unneth maps $H_{q'}C'\olotimes H_{q''}C''\rightarrow H_{q'+q''}C$. From this, for $R,S\in\Model_{\h\calP'}^\heart$ we obtain pairings $\bbL_p\ol{F}(R)\olotimes \bbL_q\ol{F}(S)\rightarrow\bbL_{p+q}\ol{F}(R\olotimes S)$.

\begin{theorem}\label{thm:pairingquillen}
Fix notation as above, and given $X\in\Model_{\calP'}^\Omega$, write $E(X)$ for the spectral sequence of \cref{thm:quillenss} computing $\pi_\ast F(X)$. Then a pairing $X'\otimes X''\rightarrow X$ in $\Model_{\calP'}^\Omega$ gives rise to a pairing $E(X')\olotimes E(X'')\rightarrow E(X)$ of spectral sequences, i.e.\ maps
\[
\smile\colon E^r_{p',q'}(X')\olotimes E^r_{p'',q''}(X'')\rightarrow E^r_{p'+p'',q'+q''}(X)
\]
satisfying
\[
d^r(x'\smile x'') = d^r(x')\smile x'' + (-1)^{q'}x\smile d^r(x''),
\]
with the pairing on $E^{r+1}$ induced by that on $E^r$. When $r=1$, this is the algebraic pairing on $\bbL_\ast\ol{F}$ twisted by $(-1)^{q'p''}$.
\end{theorem}
\begin{proof}
For the construction of the pairings, combine \cref{prop:monoidallocalization} and \cref{thm:pairingss}. By this construction and the identification of the $E^1$ pages of these spectral sequences, the diagram
\begin{center}\begin{tikzcd}
E^1_{p',q'}(X')\olotimes E^1_{p'',q''}(X'')\ar[dd,"\smile"]\ar[r,"="]&(\bbL_{p'+q'}\ol{F}\pi_0 X')[-p']\olotimes (\bbL_{p''+q''}\ol{F}\pi_0 X'')[-p'']\ar[d]\\
&(\bbL_{p'+q'+p''+q''}\ol{F}\pi_0 X)[-p'-p'']\ar[d,"\simeq"]\\
E^1_{p'+q',p''+q''}(X'')\ar[r,"="]&(\bbL_{p'+p''+q'+q''}\ol{F}\pi_0 X)[-(p'+p'')]
\end{tikzcd}\end{center}
commutes, where the top right vertical map is the algebraic pairing, and the bottom right vertical map induces a sign of $(-1)^{q'p''}$.
\end{proof}

The pairings produced by \cref{thm:pairingquillen} behave well with respect to the pairing $F(X')\otimes F(X'')\rightarrow F(X)$, see our comments at the end of \cref{ssec:pairings}.

\subsection{Universal coefficient spectral sequences}\label{ssec:stabletowers}

Fix a stable loop theory $\calP$. For $X,Y\in\LMod_\calP$, there is a mapping spectrum $\EXT_\calP(X,Y)$ with $\Omega^{\infty-n}\EXT_\calP(X,Y) = \Map_{\calP}(X,\Sigma^nY)$. 

\begin{theorem}\label{thm:stablemappingss}
Fix $X,Y\in\LMod_\calP^\Omega$. Then the spectral sequence associated to the Postnikov decomposition
\[
\EXT_\calP(X,Y)\simeq \lim_{n\rightarrow\infty}\EXT_\calP(X_{\geq 0},Y_{\leq n})
\]
is of signature
\begin{gather*}
E_1^{p,q} = \Ext^{p+q}_{\h\calP}(\pi_0 X;\pi_0 Y[p])\Rightarrow\pi_{-q}\EXT_\calP(X,Y),\qquad d_r^{p,q}\colon E_r^{p,q}\rightarrow E_r^{p+r,q+1}.
\end{gather*}
\end{theorem}

\begin{proof}
As $LX_{\geq 0}\simeq X$ and $Y\in\LMod_\calP^\Omega$, we have $\EXT_\calP(X,Y)\simeq \EXT_\calP(X_{\geq 0},Y)$. Thus there is a decomposition $\EXT_\calP(X,Y)\simeq \lim_{n\rightarrow\infty}\EXT_\calP(X_{\geq 0},Y_{\leq n})$. As $Y\in \LMod_\calP^\Omega$ we have $\pi_p Y \simeq \pi_0 Y[p]$, and thus this tower has layers described by fiber sequences
\[
\EXT_\calP(X_{\geq 0},\Sigma^p \pi_0 Y[p])\rightarrow\EXT_\calP(X_{\geq 0},Y_{\leq p})\rightarrow\EXT_\calP(X_{\geq 0},Y_{\leq p-1}).
\]
By \cref{lem:stablespiral}, there is an equivalence
\[
\EXT_\calP(X_{\geq 0},\Sigma^p\pi_0 Y[p])\simeq\EXT_{\h\calP}(\pi_0 X,\Sigma^p\pi_0 Y[p]),
\]
so that
\[
\pi_{-q}\EXT_\calP(X_{\geq 0},\Sigma^p \pi_0 Y[p]) = \Ext_{\h\calP }^{p+q}(\pi_0 X,\pi_0 Y[p]).
\]
The theorem now follows from the usual construction of the spectral sequence of a tower, which we will review in  \cref{ssec:constructss}.
\end{proof}

We end with a remark concerning the introduction of extra structure into \cref{thm:stablemappingss}. Suppose that $\calP$ is a nonunital $\bbA_2$-monoidal stable loop theory, and write the resulting pairing on $\LMod_\calP$ by $\otimes_!$. Then $\LMod_\calP$ is closed monoidal, in that for $X,Y\in\LMod_\calP$ there are objects $F_l(X,Y),F_r(X,Y)\in\LMod_\calP$ with
\[
\EXT_\calP(X,F_l(Y,Z))\simeq\EXT_\calP(Y\otimes_! X,Z),\qquad \EXT_\calP(X,F_r(Y,Z))\simeq\EXT_\calP(X\otimes_! Y,Z).
\]
Consider just $F_r$. Here $F_r(X,Y)(P) = \EXT_\calP(h(P)\otimes_! X,Y)$; in particular, if $\otimes_!$ admits a left unit $I\in\calP$, then $F_r(X,Y)(I)\simeq\EXT_\calP(X,Y)$. The same remarks hold for $\h\calP $, so that, at least if $\otimes_!$ admits a left unit, for $X\in\LMod_\calP^\Omega$ and $M\in\LMod_\calP^\heart$ we can view $\pi_{-q}F_r(X_{\geq 0},M)\in\LMod_{\h\calP}^\heart$ as an enrichment of $\Ext^q_{\h\calP }(\pi_0 X,M)$, and in this manner obtain an enriched form of \cref{thm:stablemappingss}.

\section{Postnikov decompositions}\label{sec:postnikov}

Fix a loop theory $\calP$. This section considers the study of $\Model_\calP^\Omega$ via Postnikov decompositions in $\Model_{\calP}$. We begin in \cref{ssec:topospostnikov} with a brief review of the general theory of Postnikov decompositions available in any $\infty$-topos, then specialize in \cref{ssec:postnikovtowers} to the case of Postnikov towers in $\Model_{\calP}$, which can be computed in the ambient $\infty$-topos $\Psh(\calP)$; see also \cite{pstragowski2023moduli} for a treatment of these topics. 

Given these generalities, the construction of an obstruction theory for mapping spaces in $\Model_\calP^\Omega$ is essentially immediate, and obtained in \cref{ssec:mappingspaceobstructions}. Finally, \cref{ssec:realizations} contains a verification that the Blanc-Dwyer-Goerss obstruction theory for realizations holds in our setting.

\subsection{Eilenberg-MacLane objects and Postnikov towers}\label{ssec:topospostnikov}

Fix a Grothendieck $\infty$-topos $\calX$. Up to size issues, which for our purposes can be safely ignored, $\calX$ admits an \textit{object classifier} $\Omega$; see \cite[Theorem 6.1.6.8]{lurie2017highertopos}. In other words, there is a universal map $\Omega^\ast\rightarrow\Omega$ in $\calX$, pulling back along which induces an equivalence
\[
\Map_\calX(X,\Omega)\simeq(\calX/X)^{\simeq}\simeq\coprod_{f\in\calX/X}B\!\Aut_{\calX/X}(f)
\]
for any $X\in\calX$. For $n\geq 1$, there is a subobject $\EM_n \subset\Omega$ classifying abelian Eilenberg-MacLane objects concentrated in degree $n$, with associated universal map $\EM_n^\ast\rightarrow\EM_n$; write $\EM_1'$ for the object classifying arbitrary Eilenberg-MacLane objects concentrated in degree $1$. There are also objects $\AB$ and $\GRP$ classifying discrete abelian groups and discrete groups in $\calX$ respectively, and following \cite[Proposition 7.2.2.12]{lurie2017highertopos}, there are equivalences $\pi_n\colon\EM_n^\ast\rightarrow\AB$ and $\pi_1\colon \EM_1'^\ast\rightarrow\GRP$ with inverses $B^n\colon \AB\rightarrow\EM_n^\ast$ and $B\colon \GRP\rightarrow\EM_1'^\ast$. If $X\in\calX$, then $\Map_\calX(X,\EM_n)_{\leq 1}\simeq\Map_\calX(X,\AB)$, allowing us to construct $\pi_n\colon \EM_n\rightarrow\AB$ splitting $B^n\colon \AB\rightarrow \EM_n$. We can summarize some of the relations between these as follows.

\begin{prop}\label{prop:automorphismsem}
There are Cartesian squares
\begin{center}\begin{tikzcd}
\EM_n\ar[r,"\pi_n"]\ar[d,"\pi_n"]&\AB\ar[d,"B^{n+1}"]\\
\AB\ar[r,"B^{n+1}"]&\EM_{n+1}
\end{tikzcd}\end{center}
for $n\geq 1$. In other words, $\pi_n\colon \EM_n\rightarrow\AB$ makes $\EM_n$ into an object of $\EM_{n+1}^\ast(\calX/\AB)$, with pointing given by $B^n$.
\end{prop}
\begin{proof}
For any $X\in\calX$, the Cartesian product of the given square with $X$ is the original square taken with respect to the slice topos $\calX/X$, so it is sufficient to verify that it is Cartesian upon taking global sections. Taking global sections and looking at path components corresponding to some $M\in\AB(\calX)$, it is sufficient to verify that
\begin{center}\begin{tikzcd}
B\!\Aut_\calX(B^n M)\ar[d,"\pi_n"]\ar[r,"\pi_n"]&B\!\Aut_{\Ab(\calX)}(M)\ar[d,"B^{n+1}"]\\
B\!\Aut_{\Ab(\calX)}(M)\ar[r,"B^{n+1}"]&B\!\Aut_\calX(B^{n+1}M)
\end{tikzcd}\end{center}
is Cartesian. The structure of $\AB\simeq\EM_n^\ast\rightarrow\EM_n$ gives a fiber sequence
\[
B^n\!\Map_\calX(1_\calX,M)\rightarrow B\!\Aut_{\Ab(\calX)}(M)\rightarrow B\!\Aut_\calX(B^n M),
\]
and because $M$ is abelian, this is split by $\pi_n$. We can thus identify
\[
\Omega\Fib(\pi_n)\simeq \Fib(B^n)\simeq B^n\!\Map_\calX(1_\calX,M),\qquad \Fib(B^{n+1})\simeq B^{n+1}\!\Map_\calX(1_\calX,M),
\]
and so find that the above square is Cartesian by comparing fibers.
\end{proof}

A map $X\rightarrow\AB$ classifies a discrete abelian group in $\calX/X$; call such an object an \textit{$X$-module}. As $\AB$ is $1$-truncated, $X$-modules are equivalent to $X_{\leq 1}$-modules. As moreover $X\rightarrow \pi_0 X$ is $1$-connected, $\Map_\calX(\pi_0 X,\AB)\rightarrow\Map_\calX(X,\AB)$ is $(-1)$-truncated, i.e.\ is an inclusion of a collection of path components; call an $X$-module \textit{simple} if it is in the image of this map. A theory of Postnikov towers arises from the observation that for all $n\geq 2$, there are Cartesian squares
\begin{center}\begin{tikzcd}
X_{\leq n}\ar[r]\ar[d]&\AB\ar[d]\\
X_{\leq n-1}\ar[r]&\EM_n
\end{tikzcd},\end{center}
and for $n=1$ there is an analogous square with $\EM_1$ replaced by $\EM_1'$ and $\AB$ by $\GRP$. If when $n=1$ such a replacement is not necessary, say that $X$ has \textit{abelian homotopy groups}. The top horizontal map of the above square defines an $X$-module $\Pi_n X$ for $n\geq 2$, or for $n\geq 1$ if $X$ has abelian homotopy groups. If $X$ has abelian homotopy groups and $\Pi_n X$ is simple for $n\geq 1$, say that $X$ is \textit{simple}.

For $\Lambda\in\calX_{\leq 1}$ and $M$ a $\Lambda$-module, one may form the Eilenberg-MacLane objects $B^{n+1}_\Lambda M$ in $\calX/\Lambda$ for all $n\geq 0$. When $n\geq 1$, these fit into Cartesian squares
\begin{center}\begin{tikzcd}
B^{n+1}_\Lambda M\ar[r]\ar[d]&\EM_n\ar[d,"\pi_n"]\\
\Lambda\ar[r,"M"]&\AB
\end{tikzcd}.\end{center}

\begin{prop}\label{prop:postnikovdecomposition}
Fix $X\in\calX$. For $n\geq 2$, there is a natural Cartesian square
\begin{center}\begin{tikzcd}
X_{\leq n}\ar[d]\ar[r]&X_{\leq 1}\ar[d]\\
X_{\leq n-1}\ar[r]&B^{n+1}_{X_{\leq 1}}\Pi_n X
\end{tikzcd}\end{center}
in $\calX$. If $X$ is simple, then for $n\geq 1$ there is a natural Cartesian square
\begin{center}\begin{tikzcd}
X_{\leq n}\ar[r]\ar[d]&\pi_0 X\ar[d]\\
X_{\leq n-1}\ar[r]&B^{n+1}_{\pi_0 X}\Pi_n X
\end{tikzcd}\end{center}
in $\calX$.
\end{prop}
\begin{proof}
Both cases are handled the same way. Consider the diagram
\begin{center}\begin{tikzcd}
X_{\leq n}\ar[r]\ar[d]&X_{\leq 1}\ar[d]\ar[r]&\AB\ar[d]\\
X_{\leq n-1}\ar[r,dashed,"k"]\ar[dr]&B^{n+1}_{X_{\leq 1}}\Pi_n X\ar[r,"j"]\ar[d]&\EM_n\ar[d]\\
&X_{\leq 1}\ar[r]&\AB
\end{tikzcd}.\end{center}
Here, the map $j\circ k$ exists making the upper half of the diagram Cartesian and the bottom half commute, so the individual map $k$ exists as the bottom right square is Cartesian. As the upper half of the diagram is Cartesian, to show that the upper left square is Cartesian it is sufficient to verify that the upper right square is Cartesian. As the bottom right square is Cartesian, it is sufficient to verify that the right half of the diagram is Cartesian, which is clear.
\end{proof}

\subsection{The Postnikov tower of a model of a theory}\label{ssec:postnikovtowers}
Fix a theory $\calP$, and set $\calX = \Psh(\calP)$. Observe that $\Model_{\calP}\subset\calX$ is closed under Postnikov towers, and that all objects of $\Model_\calP$ have abelian homotopy groups by \cref{lem:herdfacts}(2).

\begin{prop}[{\cite[Theorem 3.46]{pstragowski2023moduli}}]\label{prop:algebrasaresimple}
Fix $X\in\Model_\calP$. Then any $X$-module in $\AB(\Model_\calP/X)\subset\AB(\calX/X)$ is simple. In particular each $\Pi_n X$ is simple, and so $X$ is simple.
\end{prop}
\begin{proof}
By the equivalence between $X$-modules and $X_{\leq 1}$-modules, we may suppose that $X$ is $1$-truncated, so that everything is taking place within the bicategory $\Psh(\calP,\Gpd)$. Let $\pi\colon E\rightarrow X$ be an $X$-module. \cref{lem:herdfacts} shows that for all $P\in\calP$, the covering map $\pi_P\colon E(P)\rightarrow X(P)$ has trivial monodromy. The module $E$ is classified by the map $c\colon X\rightarrow\AB$ given for $P\in\calP$ by the functor $c_P\colon X(P)\rightarrow\Ab(\Psh(\calP/P,\Set))^\simeq$ defined on objects by $c_P(x)(f\colon Q\rightarrow P) = \pi_Q^{-1}(f^\ast x)$ and on morphisms by monodromy; as a consequence, each $c_P$ factors through $\pi_0 X(P)$. By replacing $\calP$ with its homotopy $2$-category, and rigidifying $X$ to a $2$-functor and $c$ to a strict natural transformation, it is seen that this pointwise factorization lifts to factor $c$ through $\pi_0 X$.
\end{proof}

Observe that if $\Lambda\in\calX_{\leq 0}$, then $(\calX/\Lambda)_{\leq 0}\simeq\calX_{\leq 0}/\Lambda$. Thus, given $X\in\Model_{\calP}$, we may form the abuses of notation $\pi_0 X = \tau^\ast \pi_0 X$, and $\Pi_n X = \tau^\ast \Pi_n X$ as $\pi_0 X$-modules. Moreover, if $\Lambda\in\Psh(\h\calP ,\Set)$ and $M$ is a $\Lambda$-module, then $\tau^\ast B^n_\Lambda = B^n_{\tau^\ast \Lambda}\tau^\ast M$. This leads to the following.

\begin{theorem}\label{thm:modelspostnikov}
For $X\in\Model_{\calP}$ and $n\geq 1$, there is a natural Cartesian square
\begin{center}\begin{tikzcd}
X_{\leq n}\ar[d]\ar[r]&\pi_0 X\ar[d]\\
X_{\leq n-1}\ar[r]&\tau^\ast B^{n+1}_{\pi_0 X}\Pi_n X
\end{tikzcd}\end{center}
in $\Model_\calP$.
\end{theorem}
\begin{proof}
Immediate from \cref{prop:postnikovdecomposition} and \cref{prop:algebrasaresimple}.
\end{proof}

\subsection{An obstruction theory for mapping spaces}\label{ssec:mappingspaceobstructions}

Fix a loop theory $\calP$, and let $\calX = \Psh(\calP)$. For $Y\in\calX$, we may identify $\calX/X\simeq\Psh(\calP/X)$, where $\calP/X$ is the slice category of $\calP$ over $X$. Given a map $f\colon X\rightarrow Y$ in $\calX$, we may form $\pi_0 f$ and $\Pi_n f$ for $n\geq 1$, considered as objects of $\calX/X$. If $X,Y\in\Model_{\calP}$, then each $\Pi_n f$ is a simple $X$-module by \cref{prop:algebrasaresimple}.

The following theorem is somewhat technical in its full generality (see \cref{rmk:mappingtechnical} below); see \cref{sssec:intropostnikov} for a special case that may be easier to digest, and \cref{ssec:stabletheories} for the simpler stable analogue.

\begin{theorem}\label{thm:mappingspaces}
Fix a $\pi_0$-surjection $R\rightarrow S$ in $\Model_\calP^\Omega$. Fix $A,C\in R/\Model_\calP^\Omega/S$, and write $p\colon C\rightarrow S$ for the given map. Fix $\phi\colon \pi_0 A\rightarrow \pi_0 C$ in $\pi_0 R/\Model_\calP^\heart/\pi_0 S$, and let $\Map^\phi_{R/\calP/S}(A,C)$ be the space of lifts of $\phi$ to a map in $R/\Model_\calP^\Omega/S$. Then the Postnikov tower of $p$ gives rise to a decomposition
\[
\Map^{\phi}_{R/\calP/S}(A,C)\simeq\lim_{n\rightarrow\infty}\Map^{\phi,\leq n}_{R/\calP/S}(A,C),
\]
where $\Map^{\phi,\leq 0}_{R/\calP/S}(A,C)\simeq\{\phi\}$, and for each $n\geq 1$ there is a natural Cartesian square
\begin{center}\begin{tikzcd}
\Map^{\phi,\leq n}_{R/\calP/S}(A,C)\ar[r]\ar[d]&\{\phi\}\ar[d]\\
\Map^{\phi,\leq n-1}_{R/\calP/S}(A,C)\ar[r]&\Map_{\pi_0 R/h\calP/\pi_0 C}(\pi_0 A,B^{n+1}_{\pi_0 C}\Pi_n p)
\end{tikzcd}.\end{center}
\end{theorem}
\begin{proof}
Explicitly, this decomposition is obtained from
\[
\Map_{R/\calP/S}(A,C)\simeq\lim_{n\rightarrow\infty}\Map_{R/\calP/S}(A,p_{\leq n}),
\]
where $p_{\leq n}$ is the $n$'th Postnikov truncation of $C$ viewed as an object of the slice topos $\calX/S$. The layers of this tower fit into Cartesian squares
\begin{center}\begin{tikzcd}
\Map_{R/\calP/S}(A,p_{\leq n})\ar[r]\ar[d]&\Map_{R/\calP/S}(A,\pi_0 p)\ar[d]\\
\Map_{R/\calP/S}(A,p_{\leq n-1})\ar[r]&\Map_{R/\calP/S}(A,B_{\pi_0 p}^{n+1}\Pi_n p)
\end{tikzcd},\end{center}
so we claim first that $\Map_{R/\calP/S}(A,\pi_0 p)\simeq\Hom_{\pi_0 R/h\calP/\pi_0 S}(\pi_0 A,\pi_0 C)$. By resolving $A$, we reduce to the case where $A = R\amalg P$ for some $q\colon P\rightarrow S$. In this case
\[
\Map_{R/\calP/S}(R\amalg P,\pi_0 p)\simeq\Map_{\calP/S}(P,\pi_0 p)\simeq (\pi_0 p)(q\colon P\rightarrow S) = \pi_0 F,
\]
where $F$ is the fiber of the map $C(P)\rightarrow S(P)$ over $q$. As the composite $R\rightarrow C\rightarrow S$ is a $\pi_0$-surjection, the map $C\rightarrow S$ is a $\pi_0$-surjection. As $C,S\in\Model_\calP^\Omega$, it follows that $\pi_1 C(P)\rightarrow \pi_1 S(P)$ is a surjection at all basepoints. Thus the fiber sequence $F\rightarrow C(P)\rightarrow S(P)$ remains a fiber sequence on taking $\pi_0$; as the fiber of $\pi_0 C(P)\rightarrow \pi_0 S(P)$ over $q$ is $\Map_{\pi_0 R/h\calP/\pi_0 S}(\pi_0 (R\amalg P),\pi_0 C)$, this gives $\Map_{R/\calP/S}(A,\pi_0 p)\simeq\Hom_{\pi_0 R/h\calP/\pi_0 S}(\pi_0 A,\pi_0 C)$ as claimed. Next, by restricting to path components corresponding to $\phi$ in the above square, we reduce to identifying the space $\Map_{R/\calP/\pi_0 p}(A,B^{n+1}_{\pi_0 p}\Pi_n p)$. This space may be identified as
\[
\Map_{R/\calP/\pi_0 p}(A,B^{n+1}_{\pi_0 p}\Pi_n p)\simeq \Map_{R/\calP/\pi_0 C}(A,B^{n+1}_{\pi_0 C}\Pi_n p)\simeq \Map_{\tau_! R/h\calP/\pi_0 C}(\tau_!A, B^{n+1}_{\pi_0 C}\Pi_n p),
\]
and we conclude by \cref{cor:spiraldiscrete}.
\end{proof}

\begin{rmk}\label{rmk:mappingtechnical}
The preceding theorem and its proof simplifies upon omitting $R$ and $S$, whereupon one obtains a decomposition of mapping spaces in $\Model_\calP$. The somewhat technical nature of the theorem as given is necessary to deal with the following subtlety: if $\calP$ is a theory and $X\in\Model_\calP$, then one might define theories $\calP/X$ and $X/\calP$ such that $\Model_{\calP/X}\simeq\Model_\calP/X$ and $\Model_{X/\calP}\simeq X/\Model_\calP$; however, in general both the maps $\h(\calP/X)\rightarrow\h\calP/\pi_0 X$ and $\h(X/\calP)\rightarrow\pi_0 X/\h\calP$ may fail to be equivalences.
\tqed
\end{rmk}

\subsection{The spiral spectral sequence}\label{ssec:spiralss}

Fix a loop theory $\calP$. By deriving \cref{thm:modelspostnikov}, we may obtain an alternate method of decomposing models of $\calP$. Write $\h_n\calP$ for the homotopy $n$-category of $\calP$, and $\tau_n\colon \calP\rightarrow \h_n\calP$ for the truncation; thus $\h_1\calP = \h\calP$ and $\tau_1 = \tau$. The main observation is the following.

\begin{theorem}\label{thm:derivedpostnikov}
Fix $X\in\Model_\calP$.
\begin{enumerate}
\item The map $X\rightarrow\tau_{(n+1)}^\ast\tau_{(n+1)!}X$ is an equivalence on $n$-truncations. In particular, 
\[
X\simeq\lim_{n\rightarrow\infty}\tau_n^\ast\tau_{n!}X.
\]
\item For $n\geq 1$, there is a natural Cartesian square
\begin{center}\begin{tikzcd}
\tau_{(n+1)}^\ast\tau_{(n+1)!}X\ar[r]\ar[d]&\tau^\ast\tau_! X\ar[d]\\
\tau_{n}^\ast\tau_{n!}X\ar[r]&\tau^*B^{n+1}_{\tau_! X}\tau_!X_{S^n}
\end{tikzcd},\end{center}
which agrees with that in \cref{thm:modelspostnikov} when $X\in\Model_\calP^\Omega$. In particular, if $X\in\Model_\calP^\Omega$ then $\tau_{(n+1)}^\ast\tau_{(n+1)!}X\simeq X_{\leq n}$.
\end{enumerate}
\end{theorem}
\begin{proof}
(1)~~If $X = h(P)$ with $P\in\calP$, then $\tau_{(n+1)}^\ast\tau_{(n+1)!}X = X_{\leq n}$, and thus $X\rightarrow\tau_{(n+1)}^\ast\tau_{(n+1)!}X$ is an equivalence on $n$-truncations. As $n$-truncations are compatible with colimits, the claim follows in general by writing $X$ as a geometric realization of objects of the form $h(P)$. 

(2)~~Recall from \cref{thm:modelspostnikov} that for any $X\in\Model_\calP$, there is a natural Cartesian square
\begin{center}\begin{tikzcd}
X_{\leq n}\ar[r]\ar[d]&\pi_0 X\ar[d]\\
X_{\leq n-1}\ar[r]&\tau^\ast B^{n+1}_{\pi_0 X}\Pi_n X
\end{tikzcd}.\end{center}
By considering only the case where $X \in\calP$, this gives a functor from $\calP$ to squares in $\Model_\calP$. By left Kan extension, this yields a functor from $\Model_\calP$ to squares in $\Model_\calP$. This functor is determined by the property of preserving geometric realizations and fact that it agrees with the above square when $X = h(P)$ with $P\in\calP$, and so is easily seen to be of the desired form. That it is Cartesian follows from \cref{lem:cartrel}. Finally, if $X\in \Model_\calP^\Omega$ then $\tau_! X\simeq \pi_0 X$ and $\tau_! X_{S^n}\simeq \Pi_n X$ by \cref{cor:spiraldiscrete}, and it follows by induction on $n$ that the two squares agree for $X$.
\end{proof}

Call the tower
\[
X\rightarrow\cdots\rightarrow\tau_{(n+1)}^\ast\tau_{(n+1)!}X\rightarrow\tau_{n}^\ast\tau_{n!}X\rightarrow\cdots\rightarrow\tau^\ast\tau_!X
\]
the \textit{derived Postnikov tower} of $X$. The existence of this derived Postnikov tower has the following corollary.

\begin{cor}
The functor $\tau_!\colon \Model_\calP\rightarrow\Model_{\h\calP}$ is conservative.
\end{cor}
\begin{proof}
Fix $X,Y\in\Model_\calP$ and a map $f\colon X\rightarrow Y$ for which $\tau_! f$ is an equivalence; we are claiming that $f$ itself is an equivalence. As $\tau_! f$ is an equivalence, so too are the maps $\tau^\ast\tau_! X\rightarrow\tau^\ast\tau_! Y$ and $\tau^\ast B^{n+1}_{\tau_!X}\tau_!X_{S^n}\rightarrow \tau^\ast B^{n+1}_{\tau_!Y}\tau_!Y_{S^n}$. The corollary then follows by inducting up derived Postnikov towers.
\end{proof}

Our main application of \cref{thm:derivedpostnikov} is to build an analogue in our context of the \textit{spiral spectral sequence} constructed by Dwyer-Kan-Stover in the context of simplicial spaces in \cite[Section 8]{dwyerkanstover1995bigraded}, itself an enhancement of the homotopy spectral sequence of a pointed simplicial space constructed by Bousfield-Friedlander in \cite[Appendix B]{bousfieldfriedlander1978homotopy}. For $X\in\Model_\calP$, this is a spectral sequence computing the $\pi_0X$-modules $\Pi_n X$ starting with information seen in $\tau_!X$. 

The construction requires some preliminaries. Fix $\Lambda\in\Model_\calP$, and consider the category $\Lambda / \Model_\calP / \Lambda$. By \cref{lem:herdfacts}(5), the functor
\[
\pi_0\colon \Lambda / \Model_\calP / \Lambda\rightarrow \pi_0\Lambda /\Model_\calP^\heart / \pi_0\Lambda
\]
preserves products, and thus if $E$ is a group object over $\Lambda$ then $\pi_0 E$ is a group object over $\pi_0 \Lambda$. This is necessarily a $\pi_0 \Lambda$-module by \cref{lem:herdfacts}(2).

\begin{lemma}
For $\Lambda\in\Model_\calP$ and $n\geq 1$, there is an isomorphism
\[
\Pi_n \Lambda \cong \pi_0(\Lambda^{S^n})
\]
of $\pi_0\Lambda$-modules.
\end{lemma}
\begin{proof}
Fix $P\in\calP$ and $x\in \pi_0 \Lambda(P)$. Then there is a natural map
\[
(\Pi_n \Lambda)(x) = \pi_n(\Lambda(P),x) = \pi_0(\Lambda(P)^{S^n}\times_{\Lambda(P)} \{x\})\rightarrow \pi_0(\Lambda^{S^n})\times_{\pi_0 \Lambda}\{x\} = (\pi_0 \Lambda^{S^n})(x),
\]
and this is an isomorphism by \cref{lem:herdfacts}(5).
\end{proof}

Given a category $\calM$ with finite limits and $\Lambda\in \calM$, the slice category $\Lambda / \calM / \Lambda$ is pointed, and thus admits a loop functor $\Omega_\Lambda$ given by $\Omega_\Lambda E = \Lambda\times_E \Lambda$. This takes values in group objects over $\Lambda$. When $\calM = \Model_\calP$, we may use this to define for any $E\in \Lambda/\Model_\calP/\Lambda$ the $\pi_0 \Lambda$-module
\[
\pi_n^\Lambda E = \pi_0 \Omega_\Lambda^n E.
\]
When $E$ is itself a group object over $\Lambda$, we extend this notation to $n = 0$ by taking $\pi_0^\Lambda E = \pi_0 \Lambda$.

\begin{lemma}\label{lem:loopb}
Let $\calM$ be a category with finite limits, $\Lambda\in\calM$, and $E\in \Lambda / \calM / \Lambda$. Then there are Cartesian squares
\begin{center}\begin{tikzcd}
\Omega^n_\Lambda E\ar[r]\ar[d]&E^{S^n}\ar[d]\\
\Lambda\ar[r]&\Lambda^{S^n}\times_\Lambda E
\end{tikzcd}\end{center}
for $n\geq 1$.
\end{lemma}
\begin{proof}
First note that there are natural maps $\Omega^n_\Lambda E\rightarrow E^{S^n}$ allowing us to form the indicated square. Indeed, consider the diagram
\begin{center}\begin{tikzcd}
\Lambda\ar[r]\ar[d]&E\ar[d]&\Lambda\ar[l]\ar[d]\\
E\ar[r]&E\times E&E\ar[l]
\end{tikzcd}.\end{center}
Taking limits horizontally yields a map $\Omega_\Lambda E\rightarrow E^{S^1}$. For $n > 1$, the desired maps may be obtained inductively using the cofibering $S^n\vee S^1\rightarrow S^n\times S^1\rightarrow S^{n+1}$. To show that this square is Cartesian, we may use the Yoneda lemma to reduce to the case $\calM = \Gpd_\infty$. Fix $x\in \pi_0\Lambda$, and write the same for its image in $\pi_0 E$ and $\pi_0 (\Lambda^{S^q}\times_\Lambda E)$. It then suffices to verify $p^{-1}(x)\simeq q^{-1}(x)$. Write $\Omega_x$ for loops based at $x$, and let $E_x$ denote the fiber of $E\rightarrow \Lambda$ over $x$. Consider the diagram
\begin{center}\begin{tikzcd}
\Omega^n_x E_x\ar[r]\ar[d]&\Omega^n_x E\ar[d]\ar[rr]&&E^{S^q}\ar[d,"q"]\\
\{x\}\ar[r]&\Omega^n_x\Lambda\ar[r]\ar[d]&E_x\times \Omega^n_x \Lambda\ar[d]\ar[r]&E\times_\Lambda \Lambda^{S^q}\ar[d]\\
&\{x\}\ar[r]&E_x\ar[r]&E
\end{tikzcd}.\end{center}
All the bottom rectangles are Cartesian. The large square consisting of the three rightmost rectangles is Cartesian, implying that the upper right rectangle is Cartesian. The leftmost square is Cartesian, and thus the top outer rectangle is Cartesian. It follows that $p^{-1}(x)\simeq\Omega^n_x E_x\simeq q^{-1}(x)$ as claimed.
\end{proof}

\begin{lemma}\label{lem:gpisplit}
Fix $\Lambda\in\Model_\calP$ and $E\in \Lambda/\Model_\calP/ \Lambda$ for which $\pi_0 E\cong \pi_0\Lambda$. Then there is a natural splitting
\[
\Pi_n E \cong \pi_n^\Lambda E \oplus \Pi_n\Lambda
\]
of $\pi_0\Lambda$-modules for all $n\geq 1$
\end{lemma}
\begin{proof}
Consider the Cartesian square
\begin{center}\begin{tikzcd}
\Omega^n_\Lambda E\ar[r]\ar[d]& E^{S^n}\ar[d,"q"]\\
\Lambda\ar[r]&\Lambda^{S^n}\times_\Lambda E
\end{tikzcd}\end{center}
guaranteed by \cref{lem:loopb}. As $q$ admits a section, this square remains Cartesian upon taking path components. As $\pi_0 E \cong\pi_0\Lambda$, we may identify $\pi_0(\Lambda^{S^n}\times_\Lambda E)\cong \Pi_n\Lambda$. Thus there is a natural short exact sequence
\[
0\rightarrow\pi_n^\Lambda E\rightarrow \Pi_n E\rightarrow\Pi_n\Lambda\rightarrow 0
\]
of $\pi_0 \Lambda$-modules. This is naturally split by the map induced by $\Lambda\rightarrow E$.
\end{proof}

Given $X\in\Model_\calP$ and $p,q\in\bbZ$, define $\pi_0 X$-modules $\tau_{p,q} X$ as follows. If $p\geq 1$ and $q\geq 0$, then
\[
\tau_{p,q} X = \pi_q^{\tau_! X}(\tau_! X_{S^p}).
\]
If $p = 0$ and $q\geq 1$, then
\[
\tau_{0,q} X = \Pi_q \tau_! X.
\]
In all other cases, we take $\tau_{p,q} X$ to be the trivial $\pi_0 X$-module.

\begin{lemma}\label{lem:pisplit}
Fix $p,q,n\geq 1$. Then there is an natural splitting
\[
\Pi_q B^n_{\tau_! X}\tau_! X_{S^p}\cong \tau_{p,q-n} X \oplus \Pi_q \tau_! X
\]
of $\pi_0 X$-modules.
\end{lemma}
\begin{proof}
As $n\geq 1$, we have $\pi_0 B^n_{\tau_! X}\tau_! X_{S^p} = \pi_0 X$. By \cref{lem:gpisplit}, there is then a natural splitting
\[
\Pi_q B^n_{\tau_! X}\tau_! X^{S^p} \cong \pi_q^{\tau_! X}B^n_{\tau_! X}\tau_! X_{S^p}\oplus\Pi_q\tau_! X.
\]
By construction, 
\[
\pi_q^{\tau_! X}B^n_{\tau_! X}\tau_! X_{S^p}\cong \pi_{q-n}^{\tau_! X}\tau_! X_{S^p}\cong \tau_{p,q-n} X,
\]
yielding the lemma.
\end{proof}

\begin{theorem}\label{thm:spiralss}
Fix $X\in\Model_\calP$. Then there is a convergent spectral sequence
\[
E^1_{p,q} = \tau_{p,q-p}X \Rightarrow \Pi_q X,\qquad d^r_{p,q}\colon E^r_{p,q}\rightarrow E^r_{p+r,q-1}
\]
of $\pi_0 X$-modules.
\end{theorem}
\begin{proof}
Consider the derived Postnikov tower $X\simeq\lim_{n\rightarrow\infty}\tau_n^\ast\tau_{n!} X$. By \cref{thm:derivedpostnikov}, this has layers described by Cartesian squares
\begin{center}\begin{tikzcd}
\tau_{(p+1)}^\ast\tau_{(p+1)!}X\ar[r]\ar[d]&\tau^\ast\tau_!X\ar[d]\\
\tau_{p}^\ast\tau_{p!}X\ar[r]&\tau^\ast B^{p+1}_{\tau_! X} \tau_! X_{S^p}
\end{tikzcd}.\end{center}
This square gives rise to a long exact sequence
\[
\cdots\rightarrow \Pi_{q+1}\tau^\ast B^{p+1}_{\tau_! X}\tau_! X_{S^p}\rightarrow \Pi_q\tau_{(p+1)}^\ast\tau_{(p+1)!}X\rightarrow \Pi_q \tau_{p}^\ast\tau_{p!}X \oplus \Pi_q \tau_! X\rightarrow \Pi_q B^{p+1}_{\tau_! X}\tau_! X_{S^p}\rightarrow\cdots
\]
of $\pi_0 X$-modules. By \cref{lem:pisplit}, we may split off copies of $\Pi_q \tau_! X$ for $q\geq 1$ to obtain a long exact sequence
\[
\cdots\rightarrow \tau_{p,q-p}X\rightarrow \Pi_q\tau_{(p+1)}^\ast\tau_{(p+1)!}X\rightarrow \Pi_q \tau_{p}^\ast\tau_{p!}X\rightarrow \tau_{p,q-p-1}X\rightarrow\cdots
\]
of $\pi_0 X$-modules. These may be combined as usual to form an exact couple leading to the desired spectral sequence. Convergence follows from \cref{thm:derivedpostnikov}(1).
\end{proof}

Call this the \textit{spiral spectral sequence} for $X$.
The rest of this subsection relates the spiral spectral sequence to topics considered in \cref{sec:stable}. Suppose that $\calP$ is stable. For any $\Lambda\in\Model_{\h\calP}^\heart$, the category of $\Lambda$-modules is equivalent to $\Model_{\h\calP}^\heart$, this equivalence sending a $\Lambda$-module $p\colon E\rightarrow\Lambda$ to $\ker(p)$. In particular, if $X\in \Model_\calP$, then the spiral spectral sequence of $X$ can be considered a spectral sequence of objects of $\Model_{\h\calP}^\heart$.

Standard methods produce for any tower 
\[
\cdots \rightarrow X(-1)\rightarrow X(0)\rightarrow X(1)\rightarrow\cdots
\]
in $\LMod_\calP$ two spectral sequences
\begin{align*}
E^1_{p,q} &= \pi_q \Cof(X(-p-1)\rightarrow X(-p))\Rightarrow \pi_q \colim_{n\rightarrow\infty}X(n) \\
E^1_{p,q} &= \pi_q \Fib(X(-p)\rightarrow X(-p+1))\Rightarrow \pi_q \lim_{n\rightarrow \infty} X(-n)
\end{align*}
of models of $\h\calP$. Call the first the colimit-type spectral sequence associated to the tower and the second the limit-type spectral sequence associated to the tower.

\begin{lemma}\label{lem:3spiral}
Fix $X\in\Model_\calP$. Then the following spectral sequences are isomorphic.
\begin{enumerate}
\item The spiral spectral sequence for $X$;
\item The limit-type spectral sequence associated to the tower
\[
X_{\geq 1}\rightarrow \cdots\rightarrow (\tau_2^\ast\tau_{2!}X)_{\geq 1}\rightarrow (\tau_1^\ast\tau_{1!} X)_{\geq 1};
\]
\item The colimit-type spectral sequence associated to the tower
\[
\cdots\rightarrow \Sigma^2 X_{\Sigma^2}\rightarrow\Sigma X_{\Sigma}\rightarrow X_{\geq 1}.
\]
\end{enumerate}
\end{lemma}
\begin{proof}
The equivalence of (1) and (2) follows immediately from the construction. The equivalence of (2) and (3) follows from the cofiber sequences
\[
\Sigma^n X_{\Sigma^n}\rightarrow X_{\geq 1}\rightarrow (\tau_n^\ast\tau_{n!} X)_{\geq 1},
\]
present for $n\geq 1$.
\end{proof}

In \cref{thm:stablelocalization}, we identified the localization $L\colon \LMod_\calP\rightarrow\LMod_\calP^\Omega$ as
\[
LX = \colim_{n\rightarrow\infty}\Sigma^{-n}X_{\Sigma^{-n}}.
\]
This yields a spectral sequence
\[
E^1_{p,q} = \pi_{q-p}\tau_! X[p]\Rightarrow \pi_q LX,\qquad d^r_{p,q}\colon E^r_{p,q}\rightarrow E^r_{p+r,q-1},
\]
and we described a special case of this in \cref{ssec:lderss}. To be precise, $E^r_{p,q}$ here agrees with $E^r_{-p,q}$ in \cref{ssec:lderss}. Call this the \textit{localization spectral sequence} for $X$. The tower producing the localization spectral sequence for $X$ extends that of \cref{lem:3spiral}(2) producing the spiral spectral sequence for $X$, and it turns out that these two spectral sequences house the same information.

With $X$ fixed, write $\{E^r_{p,q}\}$ for the spiral spectral sequence for $X$ and $\{LE^r_{p,q}\}$ for the localization spectral sequence for $X$. Let $Z^r_{p,q}$ and $B^r_{p,q}$ denote the $r$-cycles and $r$-boundaries for the localization spectral sequence, so that 
\[
0 = B^0_{p,q} \subset B^1_{p,q}\subset\cdots\subset Z^1_{p,q}\subset Z^0_{p,q} = \pi_{q-p}\tau_! X[p],\qquad LE^{r+1}_{p,q} = Z^r_{p,q} / B^r_{p,q}.
\]

\begin{prop}\label{prop:synth}
The following hold.
\begin{enumerate}
\item There are isomorphisms $Z^r_{p,q} \cong Z^r_{0,q-p}[p]$ and $B^r_{p,q} \cong B^r_{0,q-p}[p]$ in $\Model_{\h\calP}^\heart$.
\item Provided $p\geq q \geq 0$ and $p+q\geq 1$, we have $E^{r+1}_{p,q} = Z^r_{p,q} / B^{\min(p,r)}_{p,q}$. In particular, $E^{r+1}_{p,q} = LE^{r+1}_{p,q}$ for $p\geq r$.
\item For $q\geq 1$, the diagram
\begin{center}\begin{tikzcd}
Z^r_{0,q}\ar[d,"="]\ar[r]&Z^r_{0,q}/B^r_{0,q}\ar[d,"="]\ar[r,"="]&{(Z^r_{p,q+p} / B^r_{p,q+p})[-p]}\ar[d,"="]\\
E^r_{0,q}\ar[r]\ar[d,"d^r_{0,q}"]&LE^r_{0,q}\ar[r]\ar[d,"d^r_{0,q}"]&{LE^r_{p,q+p}[-p]}\ar[d,"{d^r_{p,q+p}[-p]}"]\\
E^r_{r,q-1}\ar[r,"="]&LE^r_{r,q-1}\ar[r,"="]&{LE^r_{p+r,q+p-1}[-p]}
\end{tikzcd}\end{center}
commutes.
\end{enumerate}
\end{prop}
\begin{proof}
These all follow directly from the constructions.
\end{proof}

In particular, if $X\in \Model_\calP$, then instead of computing $\pi_\ast LX$ via the localization spectral sequence, one may as well compute $\pi_\ast X$ via the spiral spectral sequence. We end with a simple application of this perspective. Say that $T$ is a monad on $\Model_\calP^\Omega$ satisfying the hypotheses of \cref{thm:recmon}, so that $T\calP$ is a loop theory with $\Model_{T\calP}^\Omega\simeq\Alg_T$.

\begin{lemma}\label{lem:loccommute}
The diagram
\begin{center}\begin{tikzcd}
\Model_{T\calP}\ar[r,"U"]\ar[d,"L"]&\Model_\calP\ar[d,"L"]\\
\Model_{T\calP}^\Omega\ar[r,"U"]&\Model_\calP^\Omega
\end{tikzcd}\end{center}
commutes, where $L$ is the localization and $U$ is the forgetful map.
\end{lemma}
\begin{proof}
Each of these functors preserves geometric realizations, so it suffices to show this diagram commutes upon restriction to objects of the form $T(P)\in\Model_{T\calP}$ for $P\in\calP$. Indeed, these are just sent to the objects of the same name in $\Model_\calP^\Omega$ under either composite.
\end{proof}

Now, say we are given $X\in\Model_{T\calP}$. \cref{lem:loccommute} implies that $LUX$ is the underlying object of the $T$-algebra $LX$, even though our explicit description of the localization $LUX$ cannot be carried out in the category $\Model_{T\calP}$. This is analogous to the observation that if $R$ is a commutative ring and $r\in R$, then $R[r^{-1}] = \colim(R\xrightarrow{r}R\xrightarrow{r}\cdots)$ is a ring despite the fact that multiplication by $r$ is generally not a ring map.

In particular, write $\bbT$ for the monad on $\Model_{\h\calP}^\heart$ for which $\Alg_{\bbT}\simeq\Model_{\h T\calP}^\heart$, as in \cref{prop:algebraicapproximations}. Then $\pi_0 LX$ is a $\bbT$-algebra, and $\pi_0 LUX$ is its underlying model of $\h\calP$. The localization spectral sequence may be used to compute $\pi_\ast LUX$, and by \cref{prop:synth} we may as well compute the spiral spectral sequence for $UX$ instead. This in turn refines to the spiral spectral sequence for $X$. This has the following consequence. 

Note that the category of $\pi_0 X$-modules is monadic over $\Model_{\h\calP}^\heart$, so it makes sense to speak of a $\pi_0 X$-module structure on an object of $\Model_{\h\calP}^\heart$.

\begin{prop}
In the above situation, the spiral spectral sequence for $UX$ is canonically equipped with the structure of a spectral sequence of $\pi_0 X$-modules, converging to the $\pi_0 X$-module structure of $\Pi_\ast X$.
\qed
\end{prop}

Combined with the localization map $\Pi_\ast X \rightarrow \Pi_\ast LX$, this gives information about the $\pi_0 LX$-modules $\Pi_n LX$, and thus about the $\bbT$-algebra structure of $\pi_0 LX$.

\subsection{An obstruction theory for realizations}\label{ssec:realizations}

Fix a loop theory $\calP$. In the case where $\calP$ is pointed and finitary, Pstrągowski \cite{pstragowski2023moduli} set up an obstruction theory for realizing an object $\Lambda\in\Model_\calP^\heart$ as $\Lambda = \pi_0 X$ for some $X\in\Model_\calP^\Omega$. In this subsection, we verify that the same obstruction theory exists for a general $\calP$; the proof is essentially the same, only with minor modifications necessary to handle the unpointed setting.

We begin with a matter that could have been considered in \cref{ssec:postnikovtowers}. Fix an $\infty$-topos $\calX$; we will soon specialize to $\calX = \Psh(\calP)$. Fix an $(n-1)$-truncated object $Y$, and set $\pi_0 Y = \Lambda$. Let $M$ be a $\Lambda$-module. Every $\pi_0$-equivalence $Y\rightarrow B^{n+1}_\Lambda M$ gives rise, by pulling back along the zero section $\Lambda\rightarrow B^{n+1}_\Lambda M$, to an $n$-truncated object $X$ such that $X_{\leq n-1}\simeq Y$ and $\Pi_n X \simeq M$ as $\Lambda$-modules; let $\calM(Y+_\Lambda(M,n))\subset\calX^{\simeq}$ be the space of such $X$. If we write $\Map^{0\hyp\Eq}$ for spaces of $\pi_0$-equivalences, then we obtain a map $\Map_\calX^{0\hyp\Eq}(Y,B_\Lambda^{n+1} M)\rightarrow\calM(Y+_\Lambda(M,n))$. Let also $\calM(Y)\subset\calX^\simeq$ be the space of objects equivalent to $Y$, and let $\Aut(\Lambda,M)$ be the discrete group of pairs $(\alpha\colon\Lambda\simeq\Lambda,f\colon M\simeq\alpha^\ast M)$, so that $B\Aut(\Lambda,M)$ is equivalent to a path component of $(\calX/\AB)^{\simeq}$. Then $X\mapsto (X_{\leq n-1},\Pi_n X)$ determines a map $\calM(Y+_\Lambda(M,n))\rightarrow\calM(Y)\times_{B\Aut(\Lambda)}B\Aut(\Lambda,M)$. Following \cite[Theorem 3.60]{pstragowski2023moduli}, we obtain the following.

\begin{prop}\label{prop:uniqueness}
The above constructions describe Cartesian squares
\begin{center}\begin{tikzcd}[column sep=small]
\Map_{\calX/\Lambda}(Y,B^{n+1}_\Lambda M)\ar[d]\ar[r]&\Map_\calX^ {0\hyp\Eq}(Y,B^{n+1}_\Lambda M)\ar[d]\ar[r]&\calM(Y+_\Lambda (M,n))\ar[d]\\
\ast\ar[r]&\Aut(\Lambda)\ar[d]\ar[r]&\calM(Y)\times_{B\!\Aut(\Lambda)}B\!\Aut(\Lambda,M)\ar[d]\\
&\ast\ar[r]&\calM(Y)\times B\!\Aut(\Lambda,M)
\end{tikzcd}\end{center}
of $\infty$-groupoids.
\qed
\end{prop}

We can now proceed to the realization problem. Fix $\Lambda\in\Model_\calP^\heart$. Call $X\in\Model_{\calP}$ a \textit{potential $n$-stage} for $\Lambda$ when $X$ is $n$-truncated, $\pi_0 X\simeq \Lambda$, and $\smash{X_{S^1}\rightarrow X^{S^1}}$ is an $(n-1)$-equivalence over $X$. Let $\calM_n(\Lambda)$ be the space of potential $n$-stages for $\Lambda$. Truncation defines $\calM_n(\Lambda)\rightarrow\calM_{n-1}(\Lambda)$, and as in \cite[Proposition 4.14]{pstragowski2023moduli} the limit $\calM_\infty(\Lambda)$ is equivalent to the space of realizations of $\Lambda$, i.e.\ the space of $X\in\Model_\calP^\Omega$ such that $\pi_0 X \simeq \Lambda$. The following facts summarize some properties of $n$-stages.

\begin{samepage}
\begin{lemma}\label{lem:htpynstage}
Let $X$ be a potential $n$-stage for $\Lambda$, and choose an isomorphism $\pi_0 X \simeq \Lambda$.
\begin{enumerate}
\item The map $X_{S^k}\rightarrow X^{S^k}$ is an $(n-k)$-equivalence over $X$;
\item $\Pi_k X\simeq \Lambda\langle k \rangle$ for $k\leq n$;
\item The only nontrivial homotopy $\Lambda$-module of $\tau_! X$ is $\Pi_{n+2}\tau_! X \simeq \Lambda\langle n+1\rangle$.
\end{enumerate}
\end{lemma}
\end{samepage}
\begin{proof}
(1)~~ This follows from an inductive argument using the Cartesian squares
\begin{center}\begin{tikzcd}
X_{S^{k+1}}\ar[d]\ar[r]&X\ar[d]\\
(X_{S^k})_{S^1}\ar[r]&X_{S^k}\times_X X_{S^1}
\end{tikzcd}\end{center}
for $k\geq 1$; compare the proof of \cref{lem:loopequiv}.

(2)~~ This follows from (1).

(3)~~ This follows from \cref{thm:spiral}.
\end{proof}

\begin{lemma}\label{lem:detectnstage}
Fix an $n$-truncated object $X\in\Model_{\calP}$ such that $\pi_0 X\simeq \Lambda$. Then $X\in\calM_n(\Lambda)$ if and only if the map $X_{S_1}\rightarrow X^{S^1}$ induces an equivalence $(B_X X_{S^1})_{\leq n}\simeq (B_X X^{S^1})_{\leq n}$.
\qed
\end{lemma}

\begin{prop}\label{prop:liftnstage}
Suppose given $Y\in\calM_{n-1}(\Lambda)$ together with a Cartesian square
\begin{center}\begin{tikzcd}
X\ar[r]\ar[d]&\tau^\ast \Lambda\ar[d]\\
Y\ar[r,"f"]&\tau^\ast B^{n+1}_\Lambda \Lambda\langle n \rangle
\end{tikzcd}.\end{center}
Then $X\in\calM_n(\Lambda)$ if and only if $f$ is adjoint to an equivalence $\tau_! Y\simeq B^{n+2}_\Lambda\Lambda\langle n \rangle$.
\end{prop}
\begin{proof}
Form Cartesian squares
\begin{center}\begin{tikzcd}
X\ar[r]\ar[d]&\tau^\ast Z\ar[d]\ar[r]&\tau^\ast \Lambda\ar[d]\\
Y\ar[r]&\tau^\ast\tau_! Y\ar[r]&\tau^\ast B^{n+1}_\Lambda\Lambda\langle n \rangle
\end{tikzcd}.\end{center}
All maps here are $\pi_0$-equivalences, and we wish to show that $X\in\calM_n(\Lambda)$ if and only if $Z\simeq\Lambda$. Form the Cartesian cube
\begin{center}\begin{tikzcd}
B_X (X\times_Y Y_{S^1})\ar[rr]\ar[dd]\ar[dr]&&X\ar[dd]\ar[dr]\\
&B_Y Y_{S^1}\ar[dd]\ar[rr]&&Y\ar[dd]\\
X\ar[dr]\ar[rr]&&\tau^\ast Z\ar[dr]\\
&Y\ar[rr]&&\tau^\ast\tau_! Y
\end{tikzcd}.\end{center}
As $X_{\leq n-1}\simeq Y_{\leq n-1}$, there are equivalences $(X_{S^1})_{\leq n-1}\simeq (Y_{S^1})_{\leq n-1}\simeq (X\times_Y Y_{S^1})_{\leq n-1}$, and so $(B_X X_{S^1})_{\leq n}\simeq (B_X(X\times_Y Y_{S^1}))_{\leq n}\simeq (X\times_{\tau^\ast Z}X)_{\leq n}$. By \cref{lem:detectnstage}, it follows that if $Z\simeq \Lambda$, then $X\in\calM_n(\Lambda)$. Conversely, if $X\in\calM_n(\Lambda)$ then $(X\times_{\tau^\ast Z} X)_{\leq n}\simeq (X\times_{\pi_0 X}X)_{\leq n}$, from which it follows that $\tau^\ast Z\simeq\Lambda$.
\end{proof}

Let $\calM^h(\Lambda+_\Lambda(\Lambda\langle n \rangle,n+1))$ be as defined in the beginning of this subsection, only constructed with respect to $\Psh(\h\calP)$. This space can be identified using \cref{prop:uniqueness}.

\begin{lemma}\label{lem:extensionsarecohomology}
There is an equivalence 
\[
\calM^h(\Lambda+_\Lambda(\Lambda\langle n \rangle,n+1))\cong \Map_{\h\calP/\Lambda}(\Lambda;B^{n+2}_\Lambda \Lambda\langle n\rangle)_{\h\!\Aut(\Lambda,\Lambda \langle n \rangle)}.
\]
Under this equivalence, $B^{n+1}_\Lambda\Lambda\langle n \rangle\in\calM^h(\Lambda+_\Lambda(\Lambda\langle n \rangle,n+1))$ is sent to the zero section.
\end{lemma}
\begin{proof}
As the space of models of $\h\calP$ equivalent to $\Lambda$ is itself equivalent to $B\!\Aut(\Lambda)$, the upper rectangle of \cref{prop:uniqueness} gives a Cartesian square
\begin{center}\begin{tikzcd}
\Map_{\h\calP/\Lambda}(\Lambda,B^{n+2}_\Lambda\Lambda\langle n \rangle)\ar[r,"p"]\ar[d]&\calM^h(\Lambda+_\Lambda(\Lambda\langle n \rangle,n+1))\ar[d]\\
\ast\ar[r]&B\!\Aut(\Lambda,\Lambda\langle n \rangle)
\end{tikzcd}.\end{center}
This says exactly that $\Aut(\Lambda,\Lambda\langle n \rangle)$ acts on $\Map_{\h\calP}(\Lambda,B^{n+2}_\Lambda\Lambda\langle n \rangle)$ with homotopy orbits $\calM^h(\Lambda+_\Lambda(\Lambda\langle n \rangle))$. By construction $p$ sends a map $f\colon \Lambda\rightarrow B^{n+2}_\Lambda\Lambda\langle n \rangle$ to its pullback along the zero section $i\colon \Lambda\rightarrow B^{n+2}_\Lambda\Lambda\langle n \rangle$, so in particular $p(i) = B^{n+1}_\Lambda\Lambda\langle n \rangle$ as claimed.
\end{proof}

This is already enough for a coarse obstruction theory.

\begin{prop}\label{prop:coarseobstruction}
Fix $\Lambda\in\Model_\calP^\heart$, and let $Y$ be an $(n-1)$-stage for $\Lambda$. Then there is an obstruction $\epsilon_n(Y)$ in the orbit set $\left(\pi_0 \Map_{\h\calP/\Lambda}(\Lambda,B^{n+2}_\Lambda \Lambda\langle n \rangle)\right)/\Aut(\Lambda,\Lambda\langle n \rangle)$ which vanishes if and only if there is an $n$-stage $X$ such that $X_{\leq n-1}\simeq Y$.
\end{prop}
\begin{proof}
\cref{lem:htpynstage} gives us a map $\tau_!\colon\calM_{n-1}(\Lambda)\rightarrow\calM^h(\Lambda+_\Lambda(\Lambda\langle n \rangle,n+1))$. By \cref{lem:extensionsarecohomology} we have $\pi_0 \calM^h(\Lambda+_\Lambda(\Lambda\langle n \rangle,n+1)) \cong \left(\pi_0 \Map_{\h\calP/\Lambda}(\Lambda,B^{n+2}_\Lambda \Lambda\langle n \rangle)\right)/\Aut(\Lambda,\Lambda\langle n \rangle)$, so let $\epsilon_n(Y)$ be the path component of $\tau_!(Y)$. We then conclude by \cref{prop:liftnstage}.
\end{proof}

The more refined statement is the following.

\begin{theorem}[{\cite[Theorem 4.24]{pstragowski2023moduli}}]\label{thm:obstructions}
For each $n\geq 1$, there is a Cartesian square
\begin{center}\begin{tikzcd}
\calM_n(\Lambda)\ar[d]\ar[r]&B\!\Aut(\Lambda,\Lambda\langle n \rangle)\ar[d]\\
\calM_{n-1}(\Lambda)\ar[r]&\calM^h(\Lambda+_\Lambda(\Lambda\langle n \rangle,n+1))
\end{tikzcd}.\end{center}
\end{theorem}
\begin{proof}
If $\calM_{n-1}(\Lambda)$ is empty, then $\calM_n(\Lambda)$ is also empty and there is nothing left to show. Otherwise, by choosing $Y\in\calM_{n-1}(\Lambda)$ and declaring $F = \{Y\}\times_{\calM_{n-1}(\Lambda)}\calM_n(\Lambda)$ to be the space of $X\in\calM_n(\Lambda)$ equipped with an equivalence $X_{\leq n-1}\simeq Y$, we reduce to verifying that the bottom square in
\begin{center}\begin{tikzcd}
\Eq(\tau_! Y,B^{n+1}_\Lambda\Lambda\langle n\rangle)\ar[r]\ar[d]&\{\Lambda\langle n \rangle\}\ar[d]\\
F\ar[r]\ar[d]&B\!\Aut(\Lambda,\Lambda\langle n \rangle)\ar[d]\\
\{\tau_! Y\}\ar[r]&\calM^h(\Lambda+_\Lambda(\Lambda\langle n \rangle,n+1))
\end{tikzcd}\end{center}
is Cartesian. Here the top left space is the space of equivalences $\tau_! Y \simeq B^{n+1}_\Lambda\Lambda\langle n \rangle$. The outer square is Cartesian by definition, so it is sufficient to verify that the top square is Cartesian. Let $F'$ be the space of all $X\in\calM(Y+_\Lambda(\Lambda\langle n \rangle,n))$ equipped with an equivalence $X_{\leq n-1}\simeq Y$, so that $F$ is a collection of path components of $F'$, and $F'$ fits into a Cartesian square
\begin{center}\begin{tikzcd}
F'\ar[r]\ar[d]&\calM(Y+_\Lambda(\Lambda\langle n\rangle,n))\ar[d]\\
\{Y\}\times B\!\Aut(\Lambda,\Lambda\langle n \rangle)\ar[r]&\calM(Y)\times B \!\Aut(\Lambda,\Lambda\langle n \rangle)
\end{tikzcd}.\end{center}
Form the diagram
\begin{center}\begin{tikzcd}
\Eq(\tau_! Y,B^{n+1}_\Lambda\Lambda\langle n \rangle)\ar[r]\ar[d]&\Map_{\calP}^{0\hyp\Eq}(Y,\tau^\ast B^{n+1}_\Lambda\Lambda\langle n \rangle)\ar[d]\ar[r]&\{\Lambda\langle n \rangle\}\ar[d]\\
F\ar[r]&F'\ar[r]&B\!\Aut(\Lambda,\Lambda\langle n \rangle)
\end{tikzcd},\end{center}
where the upper left horizontal map is obtained via the adjunction
\[
\Map_{\h\calP}(\tau_! Y,B^{n+1}_\Lambda\Lambda\langle n \rangle)\simeq\Map_{\calP}(Y,\tau^\ast B^{n+1}_\Lambda\Lambda\langle n \rangle).
\]
The rightmost square is Cartesian by \cref{prop:uniqueness}, so to show the outer square is Cartesian it is sufficient to verify that the left square is Cartesian. This follows from \cref{prop:liftnstage}.
\end{proof}

\section{Localizations and completions}\label{sec:completions}

If $R$ is an $\bbE_2$-ring and $I\subset R_0$ is a finitely generated ideal, then there is a good notion of $I$-completeness for $R$-modules, and one can proceed to consider various algebraic structures built from $I$-complete $R$-modules. We would like to be able to apply our machinery in this setting; in fact, this was our original motivation for working with infinitary theories. The use of infinitary theories to capture completeness conditions has also been employed by Brantner in \cite{brantner2017lubin}; in particular, part of \cref{prop:tameho} may be seen in \cite[Proposition 4.1.2]{brantner2017lubin}.

If $\LMod_R^{\Cpl(I)}$ is the category of $I$-complete left $R$-modules, and $\calP = \LMod_R^{\Cpl(I),\free}$ is the category of $I$-completions of free $R$-modules, then to apply our machinery we would like to say $\Model_{\calP}^\Omega\simeq\LMod_R^{\Cpl(I)}$, and to identify $\Model_{\h\calP}$ as something recognizable in terms of $\LMod_{R_\ast}$. The former always holds, and the latter is possible under a minor algebraic condition on the ideal $I$. We describe this condition in \cref{ssec:completionsr} in the more general setting of $R$-linear theories for a connective $\bbE_2$-ring $R$. Before this, in \cref{ssec:localizingtheories} we consider some general aspects of the interaction between localizations and theories.

\subsection{Localizations of theories}\label{ssec:localizingtheories}
We begin with some facts about localizing monads in a $1$-categorical setting. 

\begin{lemma}\label{lem:localizemonad}
Let $\calC$ be a $1$-category, $L\colon\calC\rightarrow\calC$ be a localization, and $T\colon\calC\rightarrow\calC$ be a monad. If $LTC\rightarrow LTLC$ is an equivalence for all $C\in\calC$, then the composite $LTLT\simeq LTT\rightarrow LT$ equips $LT$ with the structure of a monad. Moreover, $L$ canonically lifts to a localization $L\colon\Alg_T\rightarrow\Alg_{LT}$ exhibiting $\Alg_{LT}$ as the category of $T$-algebras whose underlying object of $\calC$ is $L$-local.
\end{lemma}
\begin{proof}
This is a diagram chase; see for instance \cite[Proposition 11.5]{rezk2018analytic}.
\end{proof}

\begin{rmk}
We fully expect that \cref{lem:localizemonad} holds even when $\calC$ is an $\infty$-category. As we do not have a reference, we will avoid arguments that would rely on this.
\tqed
\end{rmk}

The monad structure on $LT$ obtained by \cref{lem:localizemonad} is essentially unique, in the following sense.

\begin{lemma}\label{lem:locpairing}
Let $\calC$ be a $1$-category, and let $L\colon \calC\rightarrow\calC$ a localization. Let $T\colon\calC\rightarrow\calC$ be a functor, and suppose that $T,LT\in\Fun(\calC,\calC)$ are equipped with right-unital pairings $(\eta,\mu)$ and $(\widehat{\eta},\widehat{\mu})$ in such a way that $T\rightarrow LT$ preserves this structure.
\begin{enumerate}
\item For any $C\in\calC$, the map $LTC\rightarrow LTLC$ is an equivalence;
\item The pairing on $LT$ is given by the composite $LT\circ LT\xleftarrow{\simeq}LT\circ T = L(T\circ T)\rightarrow LT$.
\end{enumerate}
\end{lemma}
\begin{proof}
(1)~~ Write $I$ for the identity on $\calC$ and $c\colon I\rightarrow L$ for the unit. By naturality of $c$, the diagram
\begin{center}\begin{tikzcd}
I\ar[r,"c"]\ar[d,"\eta"]&L\ar[d,"L\eta"]\\
T\ar[r,"c T"]&LT
\end{tikzcd}\end{center}
commutes, and the assumption that $c$ is compatible with units implies that the composite is $\widehat{\eta}$. As a consequence, there is a commutative diagram
\begin{center}\begin{tikzcd}
LT\ar[d,"LTc"']\ar[r,"LT\widehat{\eta}"]&LTLT\ar[r,"\widehat{\mu}"]&LT\ar[d,"LTc"]\\
LTL\ar[ur,"LTL\eta"']\ar[rr,"g"]&&LTL
\end{tikzcd},\end{center}
where $g$ is defined so the diagram commutes. To show that $LTc$ is an equivalence, it is sufficient to verify that $g$ is the identity, for then an inverse is given by $\widehat{\mu}\circ LTL\eta$. Indeed, consider the diagram

\begin{center}\begin{tikzcd}[column sep=huge]
LTL\ar[r,"LTLc=LTcL"]\ar[d,"LTL\eta"]&LTLL\ar[d,"LTL\eta L"]\\
LTLT\ar[r,"LTLTc"]\ar[d,"\widehat{\mu}"]&LTLTL\ar[d,"\widehat{\mu}L"]\\
LT\ar[r,"LTc"]&LTL
\end{tikzcd}.\end{center}
The top square commutes by naturality of $\eta$, and the bottom square commutes by naturality of $\widehat{\mu}$. The clockwise composite is the identity as $L\eta\circ c = \widehat{\eta}$ implies $LTL\eta L\circ LTcL = LT\widehat{\eta}L$, and the counterclockwise composite is $g$, hence $g$ is the identity.

(2)~~ Observe that the diagram
\begin{center}\begin{tikzcd}[column sep=large]
TT\ar[r,"c TT"]\ar[d,"\mu"]\ar[rr,"cTcT",bend left]&LTT\ar[r,"LTc T"]\ar[d,"L\mu"]&LTLT\ar[d,"\widehat{\mu}"]\\
T\ar[r,"cT"]&LT\ar[r,"="]&LT
\end{tikzcd}\end{center}
commutes. Indeed, the leftmost square commutes by naturality of $c$, and the outermost by compatibility of $c$ with the pairings. As $L$ is a localization, the rightmost square commutes because the outer square commutes. Consider the diagram
\begin{center}\begin{tikzcd}[column sep=large]
LTLT\ar[d,"\widehat{\mu}"]\ar[r,"LTL\eta T"]&LTLTT\ar[d,"\widehat{\mu}T"]\\
LT&LTT\ar[ul,"LTc T"']\ar[l,"L\mu"']
\end{tikzcd}.\end{center}
By our proof of (1), the clockwise composite is exactly the pairing $LTLT\simeq LTT\rightarrow LT$, so we must verify the outer square commutes. The left triangle commutes by the above, and the right triangle gives the identity on $LTLT$ by our proof of (1), so the outer square indeed commutes.
\end{proof}

Now let $\calP$ be a theory, and let $L\colon\Model_{\calP}\rightarrow\Model_{\calP}$ be a localization which preserves geometric realizations. Being a localization, $L$ carries the structure of a monad, and the category of $L$-local objects is equivalent to the category of $L$-algebras; \cref{prop:monadtheory} thus implies that restriction along $L\colon \calP\rightarrow L\calP$ induces a fully faithful functor $\Model_{L\calP}\rightarrow\Model_\calP$ realizing $\Model_{L\calP}$ as the category of $L$-local objects in $\Model_\calP$. Let $T$ be a monad on $\Model_{\calP}$ which preserves geometric realizations, so that $\Alg_T\simeq\Model_{T\calP}$.

\begin{prop}\label{thm:locmonmal}
Fix notation as in the preceding paragraph, and suppose that $LTh(P)$ is a $T$-algebra for each $P\in\calP$ naturally in $Th(P)$.
\begin{enumerate}
\item The map $LT\rightarrow LTL$ is a natural isomorphism;
\item The functor $LT$ carries the structure of a monad, informally described by $LTLT\simeq LTT\rightarrow LT$;
\item The localization $L$ lifts to a localization $L\colon\Alg_T\rightarrow\Alg_{LT}$ realizing $\Alg_{LT}$ as the category of $T$-algebras whose underlying object of $\Model_{\calP}$ is $L$-local.
\end{enumerate}
\end{prop}
\begin{proof}
(1--2)~~ As each $LTh(P)$ is a $T$-algebra, we can let $LT\calP\subset\Alg_T$ be the full subcategory of such objects. There is then a commutative diagram
\begin{center}\begin{tikzcd}
\Model_{LT\calP}\ar[r]\ar[d]&\Model_{T\calP}\ar[d]\\
\Model_{L\calP}\ar[r]&\Model_\calP
\end{tikzcd}\end{center}
of restriction functors which preserve geometric realizations and which are the forgetful functors of monadic adjunctions. The monad associated to $\Model_{LT\calP}\rightarrow\Model_{\calP}$ has underlying functor $LT$, and this gives a map $T\rightarrow LT$ of monads. We conclude by applying \cref{lem:locpairing} to the homotopy category of $\Model_\calP$.

(3)~~ Here we are claiming that the above diagram is Cartesian, and that
\begin{center}\begin{tikzcd}
\Model_{LT\calP}\ar[d]&\Model_{T\calP}\ar[d]\ar[l]\\
\Model_{L\calP}&\Model_{\calP}\ar[l,"L"']
\end{tikzcd}\end{center}
commutes. The latter is clear, as can be checked on objects of the form $T(P)$ for $P\in\calP$, and this implies the former.
\end{proof}

It is only a bit of extra work to include loop theories in the story.

\begin{prop}\label{prop:locres}
Let $\calP$ be a loop theory, and let $L\colon\Model_{\calP}\rightarrow\Model_{\calP}$ be a localization which preserves geometric realizations. Suppose that $Lh(P)\in\Model_\calP^\Omega$ for each $P\in\calP$. Then
\begin{enumerate}
\item $\Model_{L\calP}^\Omega$ is a localization of $\Model_\calP^\Omega$;
\item The diagram
\begin{center}\begin{tikzcd}
\Model_{L\calP}^\Omega\ar[r]\ar[d]&\Model_\calP^\Omega\ar[d]\\
\Model_{L\calP}\ar[r]&\Model_{\calP}
\end{tikzcd}\end{center}
is Cartesian.
\end{enumerate} 
\end{prop}
\begin{proof}
(1)~~ Because $Lh(P)\in\Model_\calP^\Omega$ for each $P\in\calP$, the full subcategory $L\calP\subset\Model_{L\calP}$ is a loop theory, and restriction gives $\Model_{L\calP}^\Omega\rightarrow\Model_\calP^\Omega$. This is fully faithful, being obtained from $\Model_{L\calP}\rightarrow\Model_{\calP}$, and is the forgetful functor of a monadic adjunction by \cref{prop:monadicforget}.

(2)~~ Here we are claiming that $\Model_{L\calP}^\Omega$ consists of those objects of $\Model_\calP^\Omega$ which are $L$-local in $\Model_{\calP}$, which is now clear.
\end{proof}

\begin{prop}\label{thm:locresmonad}
Let $\calP$ be a loop theory, and let $L\colon\Model_{\calP}\rightarrow\Model_{\calP}$ be a localization which preserves geometric realizations such that $Lh(P)\in\Model_\calP^\Omega$ for each $P\in\calP$. Let $T$ be a monad on $\Model_\calP^\Omega$, and suppose that
\begin{enumerate}
\item[(a)] $T$ satisfies the criteria of \cref{thm:recmon}, so that $\Alg_T\simeq\Model_{T\calP}^\Omega$;
\item[(b)] $LTh(P)\in\Model_\calP^\Omega$ for each $P\in\calP$;
\item[(c)] $LTh(P)$ is a $T$-algebra for each $P\in\calP$ naturally in $Th(P)$.
\end{enumerate}
Then
\begin{enumerate}
\item The functor $LT$ carries the structure of a monad, and the associated forgetful functor can be identified as restriction $\Model_{LT\calP}^\Omega\rightarrow\Model_\calP^\Omega$;
\item The square
\begin{center}\begin{tikzcd}
\Model_{LT\calP}^\Omega\ar[r]\ar[d]&\Model_{T\calP}^\Omega\ar[d]\\
\Model_{L\calP}^\Omega\ar[r]&\Model_\calP^\Omega
\end{tikzcd}\end{center}
is Cartesian;
\item The square
\begin{center}\begin{tikzcd}
\Model_{LT\calP}^\Omega\ar[d]&\Model_{T\calP}^\Omega\ar[d]\ar[l]\\
\Model_{L\calP}^\Omega&\Model_\calP^\Omega\ar[l,"L"']
\end{tikzcd}\end{center}
commutes.
\end{enumerate}
\end{prop}
\begin{proof}
The restriction $\Model_{LT\calP}^\Omega\rightarrow\Model_{T\calP}^\Omega$ is the forgetful functor of a monadic adjunction by \cref{prop:monadicforget}. It is obtained from $\Model_{LT\calP}\rightarrow\Model_{T\calP}$, so is the inclusion of a reflective subcategory by \cref{thm:locmonmal}. We claim that the associated localization is a lift of $L$; (2) follows quickly. This is the content of (3), and moreover shows that the monad associated to $\Model_{LT\calP}^\Omega\rightarrow\Model_\calP^\Omega$ has underlying functor $LT$, proving (1). Consider the cube
\begin{center}\begin{tikzcd}
\Model_{LT\calP}^\Omega\ar[dd]&&\Model_{T\calP}^\Omega\ar[ll,dashed]\ar[dr]\ar[dd]\\
&\Model_{LT\calP}\ar[ul]\ar[dd]&&\Model_{T\calP}\ar[dd]\ar[ll]\\
\Model_{L\calP}^\Omega&&\Model_\calP^\Omega\ar[ll,dashed]\ar[dr]\\
&\Model_{L\calP}\ar[ul]&&\Model_{T\calP}\ar[ll]
\end{tikzcd},\end{center}
consisting of restrictions or left adjoints. The dashed arrows are such that the top and bottom faces commute, and the back face is the square of (3). As the rest of the diagram commutes, so does the back face, as claimed.
\end{proof}

\subsection{\texorpdfstring{$R$}{R}-linear theories and completions}\label{ssec:completionsr}

We now consider the main cases of interest, namely those derived from $I$-completion. We begin by reviewing the relevant notion of completeness; this story is developed in \cite[Section 7]{lurie2018spectral}, and our perspective is strongly influenced by \cite{rezk2018analytic}.

Fix a connective $\bbE_2$-ring $R$. In \cite[Definition D.1.4.1]{lurie2018spectral}, the notion of an $R$-linear prestable category is introduced. In short, where $\LMod_R^\sfg$ is the category of left $R$-modules equivalent to $R^{\oplus n}$ for some $n<\infty$, an $R$-linear structure on a presentable stable category $\calM$ is equivalent to an additive and monoidal functor $\LMod_R^{\sfg}\rightarrow\Fun^L(\calM,\calM)$, where the latter is the category of colimit-preserving endofunctors of $\calM$. For convenience, we extend this definition to allow $\calM$ to be an arbitrary presentable additive category; we will only apply the theory of \cite{lurie2018spectral} beyond the definitions in the case where $\calM$ is stable.

Let $I\subset R_0$ be a finitely generated ideal, and let $\calM$ be an $R$-linear stable category. An object $M\in\calM$ is said to be \textit{$I$-nilpotent} if $R[x^{-1}]\otimes_R M \simeq 0$ for all $x\in I$, is said to be \textit{$I$-local} if $\Map_\calM(N,M)$ is contractible for all $I$-nilpotent $N$, and is said to be \textit{$I$-complete} if $\Map_\calM(N,M)$ is contractible for all $I$-local $N$. Let $\calM^{\Cpl(I)}\subset\calM$ denote the full subcategory of $I$-complete objects. Then $\calM^{\Cpl(I)}$ is a reflective subcategory of $\calM$, with associated localization the functor $M\mapsto M_I^\wedge$ of \textit{$I$-completion}.

We will need an explicit formula for $I$-completion, and for this we must first fix some notation. Let $\ul{h} = \{1,\ldots,h\}$, and let $P(\ul{h})$ denote the powerset of $\ul{h}$, so that an $h$-cube is given by a functor from $P(\ul{h})$. Given an $h$-cube $V\colon P(\ul{h})\rightarrow\calC$ in some category $\calC$, we will for $i\notin S\subset\ul{h}$ write $V_i\colon V(S)\rightarrow V(S\cup\{i\})$ for the map induced by $S\subset S\cup\{i\}$. Given an $h$-cube $V$ in a category $\calC$ with finite colimits, write $\tCof V$ for the total cofiber of $V$.

Suppose now that $I$ is a finitely generated ideal, and make a choice of generators $\ul{u} = (u_1,\ldots,u_h)$. If $M$ is any object of an additive category with countable products on which $u_1,\ldots,u_h$ act, then we can define the $h$-cube
\[
K(M;\ul{u})\colon P(\ul{h})\rightarrow\calM,\qquad S\mapsto M[[T_1,\ldots,T_h]] = M^{\times \omega^h},
\]
where
\[
K(M;\ul{u})_i = (T_i-u_1)\colon M[[T_1,\ldots,T_h]]\rightarrow M[[T_1,\ldots,T_h]].
\]

\begin{prop}\label{prop:completionformula}
Let $\calM$ be an $R$-linear stable category. Then for $M\in\calM$, there is an equivalence
\[
M_I^\wedge \simeq \tCof K(M;\ul{u}).
\]
\end{prop}
\begin{proof}
When $h=1$, this is a reformulation of \cite[Proposition 7.3.2.1]{lurie2018spectral}. The general case then follows from \cite[Proposition 7.3.3.2]{lurie2018spectral}.
\end{proof}

Let $\calP$ be an additive theory. Say that $\calP$ is an \textit{$R$-linear theory} if we have chosen an additive monoidal functor $\Mod_R^{\sfg}\rightarrow\Fun^\oplus(\calP,\calP)$, where the latter is the category of coproduct-preserving endofunctors of $\calP$. If $\calP$ is an additive loop theory, say that $\calP$ is an \textit{$R$-linear loop theory} if we have chosen an additive monoidal functor $\Mod_R^\sfg\rightarrow\Fun^{\oplus,\Sigma}(\calP,\calP)$, where the latter is the category of coproduct and suspension-preserving endofunctors of $\calP$. Note that if $\calP$ is an $R$-linear theory, then so is $\h\calP $. If $\calP$ is an $R$-linear theory, then $\Model_{\calP}$ is an $R$-linear category, and $\LMod_\calP$ is an $R$-linear stable category; if $\calP$ is an $R$-linear loop theory, then $\Model_\calP^\Omega$ is an $R$-linear category, stable so long as $\calP$ is a stable loop theory.

\begin{prop}\label{prop:completiongeom}
Let $\calP$ be an $R$-linear theory. Then $I$-completion
\[
\LMod_\calP\rightarrow\LMod_\calP,\qquad X\mapsto X_I^\wedge
\]
restricts to a localization of $\LMod_\calP^\cn\simeq\Model_{\calP}$ which preserves geometric realizations. This localization is given explicitly by
\[
X_I^\wedge = \tCof K(X;\ul{u})
\]
for $X\in\Model_{\calP}$.
\end{prop}
\begin{proof}
If $X\in \LMod_\calP$ is valued in connective spectra, then so too is $\tCof K(X;\ul{u})$. Thus $I$-completion restricts to a localization of $\LMod_\calP^{\cn}\simeq\Model_\calP$ as claimed. This localization preserves geometric realizations as geometric realizations commute with total cofibers and infinite products in $\Model_\calP$, the latter by \cref{prop:geometricrealizationsgroups}.
\end{proof}

Call $X\in\Model_{\calP}$ \textit{$I$-complete} if $X\simeq X_I^\wedge$. If $\calP$ is a stable loop theory, then $\Model_\calP^\Omega$ is itself a stable $R$-linear category, so there is a possible ambiguity in speaking of $I$-complete objects of $\Model_\calP^\Omega$. However, this ambiguity turns out to vanish.

\begin{lemma}\label{lem:compcart}
Let $\calP$ be an $R$-linear stable loop theory, and fix $X\in\Model_\calP^\Omega$. Then $X$ is $I$-complete as an object of $\Model_\calP^\Omega$ if and only if it is $I$-complete as an object of $\Model_{\calP}$.
\end{lemma}
\begin{proof}
If $X\in\Model_\calP^\Omega$ is $I$-complete in $\Model_{\calP}$, then $X\simeq\tCof(X;\ul{u})$ in $\Model_{\calP}$. So this total cofiber lives in $\Model_\calP^\Omega$, and thus $X$ is $I$-complete in $\Model_\calP^\Omega$. 

Suppose conversely that $X$ is $I$-complete in $\Model_\calP^\Omega$. By definition, $X$ is $I$-complete in $\Model_{\calP}$ if and only if for every $M\in\LMod_\calP$ which is $I$-local, the mapping space $\Map_{\calP}(M,X)$ is contractible. As $\Map_{\calP}(M,X)\simeq\Map_{\calP}(L M,X)$ where $L\colon \LMod_\calP\rightarrow\LMod_\calP^\Omega$ is the localization, and as $X$ is $I$-complete in $\LMod_\calP^\Omega$, it is sufficient to verify that $L$ preserves $I$-local objects. This is a consequence of \cite[Proposition 7.2.4.9]{lurie2018spectral} and the fact that $L$ is $R$-linear.
\end{proof}

\begin{prop}\label{prop:completion}
Let $\calP$ be an $R$-linear stable loop theory. Let $\calP_I^\wedge\subset\Model_\calP^\Omega$ be the full subcategory spanned by the $I$-completions of objects of $\calP$. Then
\[
\Model_\calP^{\Omega,{\Cpl(I)}}\simeq\Model_{\calP_I^\wedge}^\Omega.
\]
\end{prop}
\begin{proof}
By \cref{thm:recmon}, it is sufficient to verify that
\begin{center}\begin{tikzcd}
\Model_\calP^\Omega\ar[r,"(\bs)_I^\wedge"]\ar[d,"i"]&\Model_\calP^\Omega\\
\Model_\calP\ar[r,"(\bs)_I^\wedge"]&\Model_{\calP}\ar[u,"L"]
\end{tikzcd}\end{center}
commutes, where $i$ is the inclusion. Indeed, let $M\in\Model_\calP^\Omega$. Then as $L$ preserves colimits and $i$ preserves limits, \cref{prop:completionformula} implies
\[
L(i(M)_I^\wedge)\simeq L(\tCof K(i(M);\ul{u}))\simeq \tCof L(i(K(M;\ul{u})))\simeq \tCof K(M;\ul{u})\simeq M_I^\wedge
\]
as claimed.
\end{proof}

Some additional hypotheses are necessary to make practical use of \cref{prop:completion}. For example, $\h\calP $ is itself an $R$-linear theory, so one would like to identify $\Model_{\h(\calP_I^\wedge)}\simeq\Model_{\h\calP}^{\Cpl(I)}$; however, this is not true in general. Determining when properties such as this hold amount to understanding when, given $X\in\Model_\calP^\Omega$, the completion $X_I^\wedge$ as computed in $\Model_{\calP}$ still lives in $\Model_\calP^\Omega$. This turns out to be an essentially algebraic condition.

Call a theory $\calP$ \textit{pretame} if $\tau_!\colon \Model_{\calP}\rightarrow\Model_{\h\calP}$ preserves countable products. It follows from \cref{thm:spiral} that every loop theory is pretame. The purpose of the pretameness condition is the following.

\begin{lemma}\label{lem:pretamecomp}
Let $\calP$ be a pretame $R$-linear theory. Then
\[
\tau_!(X_I^\wedge) = (\tau_! X)_I^\wedge
\]
for all $X\in\Model_{\calP}$.
\end{lemma}
\begin{proof}
As $\calP$ is pretame, there is an equivalence $\tau_! K(X;\ul{u})\simeq K(\tau_! X;\ul{u})$ of $h$-cubes, so the claim follows from \cref{prop:completiongeom}.
\end{proof}

If $\calP$ is a discrete $R$-linear theory, then the $I$-completion of discrete objects of $\Model_{\calP}$ admits an algebraic description. Given an abelian category $\calA$, an $h$-cube $V\colon P(\ul{h})\rightarrow\calA$ may be viewed as an $h$-dimensional complex, and so we may form the total complex $C_\ast V$. This satisfies $H_0(C_\ast V) = \tCof V$, this total cofiber taken in the $1$-category $\calA$. In general, $C_\ast V$ is a model of the derived total cofiber of $V$ in the following sense.

\begin{lemma}\label{lem:tcof}
Let $\calA$ be an abelian category with enough projectives, and let $V\colon P(\ul{h})\rightarrow\calA$ be an $h$-cube. Then $(\bbL_n\!\tCof)(V) = H_n(C_\ast V)$.
\qed
\end{lemma}

In particular, if $\calA$ is an $R_0$-linear abelian category with countable products and $M\in\calA$, then we can form the chain complex $C_\ast K(M;\ul{u})$. In this case, set $K_n(M;\ul{u}) = H_n C_\ast K(M;\ul{u})$.

\begin{lemma}\label{lem:discom}
Let $\calP$ be a discrete $R$-linear theory. For $M\in\Model_{\calP}$ discrete, there are isomorphisms
\[
\pi_\ast (M_I^\wedge) \cong K_\ast(M;\ul{u}).
\]
In particular,
\begin{enumerate}
\item $M_I^\wedge$ is $h$-truncated, $h$ being the length of the sequence $\ul{u}$;
\item $M_I^\wedge$ is discrete if and only if $K_n(M;\ul{u})=0$ for $n>0$.
\end{enumerate}
\end{lemma}
\begin{proof}
This is an immediate consequence of \cref{prop:completiongeom} and \cref{lem:tcof}.
\end{proof}

We now arrive at the promised characterization.

\begin{prop}\label{prop:deftame}
Let $\calP$ be a pretame $R$-linear theory, and fix $X\in\Model_{\calP}$. Suppose that $\tau_! X = \pi_0 X$. Then the following are equivalent:
\begin{enumerate}
\item The object $\tau_! X_I^\wedge$ of $\Model_{\h\calP}$ is discrete;
\item The object $K_n(\pi_0 X;\ul{u})$ of $\Model_{\h\calP}^\heart$ vanishes for $n>0$.
\end{enumerate}
If $\calP$ is a loop theory, then $X\in\Model_\calP^\Omega$, and these are equivalent to:
\begin{enumerate}[resume]
\item The completion $X_I^\wedge$ as computed in $\Model_{\calP}$ lives in $\Model_\calP^\Omega$.
\end{enumerate}
\end{prop}
\begin{proof}
The equivalence of (1) and (2) follows from \cref{lem:pretamecomp} and \cref{lem:discom}, and the inclusion of (3) follows from \cref{cor:spiraldiscrete}.
\end{proof}

Say that $I$ is \textit{tame} on $X\in\Model_{\calP}$ if the equivalent conditions of \cref{prop:deftame} hold for $X$, and say that $I$ is \textit{tame} on $\calP$ if $I$ is tame on $h(P)$ for $P\in\calP$. In particular, $I$ is tame on $\calP$ if and only if it is tame on $\h\calP $, i.e.\ tameness is an algebraic condition. 

If $S$ is an $R$-algebra and $\calP = \LMod_S^\free$, then $I$ is tame on $\calP$ precisely when it is tame on $\LMod_{S_\ast}^\free=\h(\LMod_S^\free)$. Tameness in this setting coincides with the notion of tameness discussed in Greenlees-May \cite{greenleesmay1992derived} \cite{greenleesmay1995completions} and Rezk \cite[Section 8]{rezk2018analytic}, and holds in a number of situations. For example, if $M$ is an $S_0$-module, then $I$ is tame on $M$ when $I$ is generated by a sequence which is regular on $M$, or when $M = N^{\oplus J}$ for some Noetherian $S_0$-module $N$ and some set $J$ \cite[Corollary 7.3.6.1]{lurie2018spectral}.

The definition of tameness is chosen so that the following holds.

\begin{theorem}\label{prop:tameho}
Let $\calP$ be a stable $R$-linear theory, and suppose that $I$ is tame on $\calP$. Let $\calP_I^\wedge\subset\LMod_\calP^\Omega$ be the full subcategory spanned by the $I$-completions of objects of $\calP$. Then there are equivalences
\[
\Model_{\calP_I^\wedge}^\Omega\simeq\Model_\calP^{\Omega,\Cpl(I)},\qquad\Model_{\calP_I^\wedge}\simeq\Model_\calP^{\Cpl(I)},\qquad\Model_{\h(\calP_I^\wedge)}\simeq\Model_{\h\calP}^{\Cpl(I)}.
\]
\end{theorem}
\begin{proof}
The first equivalence is a restatement of \cref{prop:completion}. For the remaining two, observe that as
$I$ is tame on $\calP$, the category $\calP_I^\wedge$ may be identified as the full subcategory of $\Model_\calP$ spanned by $h(P)_I^\wedge$ for $P\in\calP$, where this completion is taken in $\Model_{\calP}$, and that $\h(\calP_I^\wedge)\subset\Model_{\h\calP}$ may be identified as the full subcategory spanned by $(\pi_0 h(P))_I^\wedge$ for $P\in\calP$. So these equivalences are consequences of \cref{prop:completiongeom} and \cref{prop:monadtheory}.
\end{proof}

\cref{prop:tameho} extends by combination with \cref{ssec:localizingtheories} to describe unstable theories built out of stable theories in completed settings.

\begin{appendix}

\section{Spectral sequences}\label{sec:app}

This section gives the facts that were needed in \cref{sec:stable} about towers in a stable category with $t$-structure and their associated spectral sequences. We will freely use material and notation from \cite[Section 1.2]{lurie2017higheralgebra}.

\subsection{Construction and convergence}\label{ssec:constructss}

Fix a stable category $\calC$ with $t$-structure, and let $\calA$ be the heart of $\calC$. There results a functor $\pi_0 = \tau_{\leq 0}\tau_{\geq 0}\colon \calC\rightarrow\calA$, and we set $\pi_p = \pi_0\circ\Sigma^{-p}$. Fix a tower
\[
X  = \cdots\rightarrow X(-1)\rightarrow X(0)\rightarrow X(1)\rightarrow\cdots
\]
in $\calC$. Following \cite[Section 1.2]{lurie2017higheralgebra}, there is for each $-\infty\leq p \leq q$ an object $X(p,q)$ in $\calC$, where $X(-\infty,p) = X(p)$ and for $p\leq q\leq r$ there is a chosen cofiber sequence
\begin{center}\begin{tikzcd}
X(p,q)\ar[r,"\eta"]& X(p,r)\ar[r,"\eta"]& X(q,r)
\end{tikzcd}.\end{center}
In particular, $X(p,q)$ sits in a cofiber sequence
\begin{center}\begin{tikzcd}
X(p)\ar[r,"\eta"] &X(q)\ar[r,"\eta"]&X(p,q)
\end{tikzcd}.\end{center}

Define
\[
E^r_{p,q} = \im\left(\pi_q X(p-r,p)\rightarrow \pi_q X(p-1,p+r-1)\right);
\]
we abbreviate $E^r_{p,\ast}$ as $E^r_p$ when it simplifies the notation. Using the diagrams
\begin{center}\begin{tikzcd}
X(p-r,p)\ar[r,"\eta"]\ar[d]&X(p-1,p+r-1)\ar[d]\\
\Sigma X(p-2r,p-r)\ar[r,"\Sigma\eta"]&\Sigma X(p-r-1,p-1)
\end{tikzcd},\end{center}
we obtain maps
\[
d^r_{p,q}\colon E^r_{p,q}\rightarrow E^r_{p-r,q-1}.
\]
\begin{prop}[{\cite[Proposition 1.2.2.7]{lurie2017higheralgebra}}]
With notation as above,
\begin{enumerate}
\item $d^r\circ d^r = 0$;
\item There are canonical equivalences $E^{r+1} = H(E^r,d^r)$.
\end{enumerate}
In particular, $\{E^r,d^r\}$ is a spectral sequence of objects of $\calA$.
\qed
\end{prop}

Write $E(X)$ for this spectral sequence. We would like to identify some simple criteria for convergence. Suppose that $\calC$ and $\calA$ admit countable direct sums, and thus all countable colimits. For a tower $X$, write $X(\infty) = \colim_{p\rightarrow\infty}X(p)$. The following may be proved just as in \cite[Proposition 1.2.2.14]{lurie2017higheralgebra}.

\begin{prop}\label{prop:convergess}
Fix a tower $X$, and suppose 
\begin{enumerate}
\item[(a)] The connectivity of $X(p)$ goes to $\infty$ as $p$ goes to $-\infty$;
\item[(b)] $\colim_{r\rightarrow\infty}\pi_\ast X(p,p+r)\cong \pi_\ast \colim_{r\rightarrow\infty}X(p,p+r)$ for all $p\in\bbZ\cup\{-\infty\}$;
\end{enumerate}
and moreover one of the following holds:
\begin{enumerate}
\item[(c)]The $t$-structure on $\calC$ is compatible with filtered colimits;
\item[(c$'$)]For all $q\in\bbZ$, the map $\pi_q X(p)\rightarrow \pi_q X(p+1)$ is an isomorphism for all but finitely many $p$.
\end{enumerate}
Then $E(X)$ converges to $\pi_\ast X(\infty)$. Explicitly, if $A_q = \pi_q X(\infty)$, then
\begin{enumerate}
\item For all fixed $p,q$ and all sufficiently large $r$, there are canonical inclusions $E_{p,q}^r\subset E_{p,q}^{r+1}$, and in case (c$'$) these eventually stabilize;
\item Where $F^p A_q = \im(\pi_q X(p) \rightarrow \pi_q X(\infty))$, both $F^p A_q=0$ for $p$ sufficiently small and $\colim_{p\rightarrow\infty} F^p A_q = A_q$, and in case (c$'$) this filtration is finite;
\item There are canonical isomorphisms $F^pA_q/F^{p-1}A_q \cong E^\infty_{p,q}$.
\qed
\end{enumerate}
\end{prop}

\subsection{Monoidal properties of towers}\label{ssec:monoidaltower}

Fix a stable category $\calC$ with $t$-structure, and let $\calO$ be a single-colored $\infty$-operad. Following \cite[Definition 2.2.1.6]{lurie2017higheralgebra}, say an $\calO$-monoidal structure on $\calC$ is \textit{compatible with the $t$-structure} on $\calC$ if it respects finite colimits, and for all $f\in\calO(n)$, the tensor product $\otimes_f$ sends $\calC_{\geq 0}^{\times n}$ into $\calC_{\geq 0}$. Fix such an $\calO$-monoidal structure on $\calC$.

\begin{prop}\label{prop:whiteheadmonoidal}
The functor $\calC\rightarrow\Fun(\bbZ,\calC)$ sending an object to its Whitehead tower is canonically lax $\calO$-monoidal.
\end{prop}
\begin{proof}
This functor factors as the composite of the diagonal $\calC\rightarrow\Fun(\bbZ,\calC)$ and the endofunctor $W$ of $\Fun(\bbZ,\calC)$ sending a tower $n\mapsto X(n)$ to the new tower $n\mapsto X(n)_{\geq -n}$. The former is lax $\calO$-monoidal, as $\bbZ\rightarrow\{0\}$ is monoidal, hence it is sufficient to verify that the latter is lax $\calO$-monoidal. This follows from \cite[Proposition 2.2.1.1]{lurie2017higheralgebra}, for $W$ is a colocalization of $\Fun(\bbZ,\calC)$, with image closed under the $\calO$-monoidal structure by compatibility with the $t$-structure on $\calC$.
\end{proof}

Restrict now to the case where $\calO$ is the nonunital $\bbA_2$-operad. In other words, fix a pairing $\otimes\colon\calC\times\calC\rightarrow\calC$ which is exact in each variable and sends $\calC_{\geq 0}\times\calC_{\geq 0}$ to $\calC_{\geq 0}$. Writing $\calA$ for the heart of $\calC$, this gives a pairing $\olotimes$ on $\calA$ by
\[
M\olotimes N = \pi_0(M\otimes N).
\]
For $X',X''\in\calC$, there is a canonical K\"unneth map $\pi_0 X' \olotimes\pi_0 X''\rightarrow \pi_0(X\otimes X'')$ given by
\[
\pi_0 X' \olotimes\pi_0 X'' = \pi_0(X'_{\geq 0}\otimes X''_{\geq 0})\rightarrow \pi_0(X'\otimes X'').
\]
There is not such a canonical map in nonzero degrees, the issue being the following. As $\otimes$ is exact in each variable, there are canonical isomorphisms $(\Sigma X')\otimes X''\simeq \Sigma(X'\otimes X'')$ and $X'\otimes (\Sigma X'')\simeq \Sigma(X'\otimes X'')$. However, the diagram
\begin{center}\begin{tikzcd}
(\Sigma^{q'} X') \otimes(\Sigma^{q''} X'')\ar[r]\ar[d]&\Sigma^{q'}(X'\otimes(\Sigma^{q''} X''))\ar[d]\\
\Sigma^{q''}((\Sigma^{q'} X')\otimes X'')\ar[r]&\Sigma^{q'+q''}(X'\otimes X'')
\end{tikzcd}\end{center}
can only be made to commute up to a switch map $S^{q'+q''}\simeq S^{q'}\otimes S^{q''}\simeq S^{q''}\otimes S^{q'}\simeq S^{q''+q'}=S^{q'+q''}$, and so on $\pi_0$ up to a sign of $(-1)^{q'q''}$. For the rest of this section, we choose the isomorphism given by the counterclockwise composite; in other words, we choose
\[
(\Sigma^{q'} X')\otimes (\Sigma^{q''} X'')=\Sigma^{q''}\Sigma^{q'}(X'\otimes X'').
\]
This choice falls naturally out of the convention of pretending that $(\Sigma X')\otimes X''$ and $\Sigma(X'\otimes X'')$ are the ``same'', whereas $X'\otimes (\Sigma X'')$ and $\Sigma(X'\otimes X'')$ are ``different''. Having made a choice, we obtain a natural transformation
\begin{align*}
\pi_{q'} X' \olotimes \pi_{q''} X'' &= \pi_0 \Sigma^{-q'} X'\olotimes\pi_0\Sigma^{-q''} X''\\
&\rightarrow \pi_0(\Sigma^{-q'} X'\otimes\Sigma^{-q''}X'')\simeq\pi_{q'+q''}(X'\otimes X'').
\end{align*}
With this choice, the diagram
\begin{center}\begin{tikzcd}
\pi_{q'}\Sigma X'\olotimes \pi_{q''}X''\ar[r,"="]\ar[d]&\pi_{q'-1}X'\otimes \pi_{q''}X''\ar[d]\\
\pi_{q'+q''}\Sigma X'\otimes X''\ar[r,"\simeq"]&\pi_{q'+q''-1}X'\otimes X''
\end{tikzcd}\end{center}
commutes, whereas the diagram
\begin{center}\begin{tikzcd}
\pi_{q'} X'\olotimes\pi_{q''} \Sigma X''\ar[d]\ar[r,"="]&\pi_{q'} X' \olotimes\pi_{q''-1}X''\ar[d]\\
\pi_{q'+q''}X'\otimes\Sigma X''\ar[r,"\simeq"]&\pi_{q'+q''-1}X'\otimes X''
\end{tikzcd}\end{center}
commutes up to a factor of exactly $(-1)^{q'}$. This is the origin of the signs that will appear for us.

We end this subsection by recording a concrete description of a pairing in $\Fun(\bbZ,\calC)$.

\begin{lemma}
A pairing $X'\otimes X''\rightarrow X$ in $\Fun(\bbZ,\calC)$ is equivalent to the choice of pairings $X'(p')\otimes X''(p'')\rightarrow X(p'+p'')$ for $p',p''\in\bbZ$, together with homotopies filling in the cubes
\begin{center}\begin{tikzcd}[column sep=0mm]
X'(p'-1)\otimes X''(p''-1)\ar[rr]\ar[dd]\ar[dr]&&X'(p')\otimes X''(p''-1)\ar[dd]\ar[dr]\\
&X(p'+p''-2)\ar[rr]\ar[dd]&&X(p'+p''-1)\ar[dd]\\
X'(p'-1)\otimes X''(p'')\ar[rr]\ar[dr]&&X'(p')\otimes X''(p'')\ar[dr]\\
&X(p'+p''-1)\ar[rr]&&X(p'+p'')
\end{tikzcd}.\end{center}
\end{lemma}
\begin{proof}
This is immediate from the construction of the tensor product on $\Fun(\bbZ,\calC)$ and the form of mapping spaces in $\Fun(\bbZ\times\bbZ,\calC)$.
\end{proof}

\subsection{Pairings of spectral sequences}\label{ssec:pairings}
Fix conventions as in the previous subsections. Our goal in this subsection is to verify that every pairing $X'\otimes X''\rightarrow X$ of towers gives rise to a pairing $E(X')\olotimes E(X'')\rightarrow E(X)$ of spectral sequences. Before giving the main construction, we point out the following. By a cofibering $X''(-1)\rightarrow X''(0)\rightarrow C''(0)$, we refer really to the left square in a suitable coherently commutative diagram
\begin{center}\begin{tikzcd}
X''(-1)\ar[r]\ar[d]&X''(0)\ar[d]\ar[r]&0\ar[d]\\
0\ar[r]&C''(0)\ar[r,"\delta"]&\Sigma X''(-1)
\end{tikzcd},\end{center}
from which we obtain the right square, in particular the boundary map $\delta$. As $\otimes$ is exact in both variables, for any $X'\in\calC$, we may tensor the original cofiber sequence with $X'$ to obtain a cofiber sequence $X'\otimes X''(-1)\rightarrow X'\otimes X''(0)\rightarrow X'\otimes C''(0)$. Again, this refers really to the left square in
\begin{center}\begin{tikzcd}
X'\otimes X''(-1)\ar[r]\ar[d]&X'\otimes X''(0)\ar[d]\ar[r]&0\ar[d]\\
0\ar[r]&X'\otimes C''(0)\ar[r,"\delta'"]&\Sigma(X'\otimes X''(-1))
\end{tikzcd},\end{center}
where we have implicitly identified $X'\otimes 0\simeq 0$, and from this we obtain the right square, in particular the boundary map $\delta'$. This diagram is canonically equivalent to the diagram obtained by tensoring the first with $X'$, and so $\delta'$ is homotopic to the composite
\begin{center}\begin{tikzcd}
X'\otimes C''(0)\ar[r,"X'\otimes\delta"]&X'\otimes \Sigma X''(-1)\ar[r,"\simeq"]&\Sigma(X'\otimes X''(-1))
\end{tikzcd}.\end{center}

We now proceed to the main construction. Fix the data of cofiberings
\begin{gather*}
X(-2)\rightarrow X(-1)\rightarrow C(-1)\qquad X(-1)\rightarrow X(0)\rightarrow C(0)\\
X'(-1)\rightarrow X'(0)\rightarrow C'(0)\qquad X''(-1)\rightarrow X''(0)\rightarrow C''(0),
\end{gather*}
as well as the data of a filled in cube
\begin{center}\begin{tikzcd}
X'(-1)\otimes X''(-1)\ar[rr]\ar[dd]\ar[dr]&&X'(0)\otimes X''(-1)\ar[dd]\ar[dr]\\
&X(-2)\ar[rr]\ar[dd]&&X(-1)\ar[dd]\\
X'(-1)\otimes X''(0)\ar[rr]\ar[dr]&&X'(0)\otimes X''(0)\ar[dr]\\
&X(-1)\ar[rr]&&X(0)
\end{tikzcd}.\end{center}
Our initial cofiberings, together with the fact that $\otimes$ is exact in each variable, give rise to a canonical isomorphism from the total cofiber of the back face of this cube to $C'(0)\otimes C''(0)$. As a consequence of this, we can form the commutative diagrams
\begin{center}\begin{tikzcd}
X'(-1)\otimes X''(0)\cup_{X'(-1)\otimes X''(-1)}X'(0)\otimes X''(-1)\ar[r]\ar[d]&X(-1)\ar[d]\\
X'(0)\otimes X''(0)\ar[r]\ar[d]&X(0)\ar[d]\\
C'(0)\otimes C''(0)\ar[r,dashed]&C(0)
\end{tikzcd}\end{center}
\begin{center}\begin{tikzcd}
X'(-1)\otimes X''(-1)\ar[r]\ar[d]&X(-2)\ar[d]\\
X'(-1)\otimes X''(0)\cup_{X'(-1)\otimes X''(-1)}X'(0)\otimes X''(-1)\ar[r]\ar[d,"f"]&X(-1)\ar[d]\\
X'(-1)\otimes C''(0)\oplus C'(0)\otimes X''(-1)\ar[r,dashed]&C(-1)
\end{tikzcd},\end{center}
where the columns have the structure of cofiber sequences and the bottom squares are induced from this. From the construction of the maps involved, we obtain the following.

\begin{lemma}\label{lem:leib}
In the diagram
\begin{center}\begin{tikzcd}
C'(0)\otimes C''(0)\ar[r]\ar[d]&C(0)\ar[d]\\
\Sigma(X'(-1)\otimes X''(0)\cup_{X'(-1)\otimes X''(-1)}X'(0)\otimes X''(-1))\ar[r]\ar[d,"\Sigma f"]&\Sigma X(-1)\ar[d]\\
\Sigma(X'(-1)\otimes C''(0)\oplus C'(0)\otimes X''(-1))\ar[r]&\Sigma C(-1)
\end{tikzcd}\end{center}
obtained from the above data, the left vertical composite is given by the sum of the maps
\begin{gather*}
C'(0)\otimes C''(0)\rightarrow(\Sigma X'(-1))\otimes C''(0)\simeq \Sigma(X'(-1)\otimes C''(0))\\
C'(0)\otimes C''(0)\rightarrow C'(0)\otimes (\Sigma X''(-1))\simeq \Sigma(C'(0)\otimes X''(-1)).
\end{gather*}
\qed
\end{lemma}

We are now in a position to prove the following.

\begin{theorem}\label{thm:pairingss}
A pairing $X'\otimes X''\rightarrow X$ of towers gives rise to a pairing $E(X')\olotimes E(X'')\rightarrow E(X'')$ of spectral sequences, i.e.\ pairings
\[
\smile\colon E^r_{p',q'}(X')\olotimes E^r_{p'',q''}(X'')\rightarrow E^r_{p'+p'',q'+q''}(X)
\]
satisfying the Leibniz rule
\[
d^r (x'\smile x'') = d^r(x')\smile x'' + (-1)^{q'} x'\smile d^r(x''),
\]
where moreover the pairing on $E^r$ is induced from naturally defined maps
\[
X'(p'-r,p')\otimes X''(p''-r,p'')\rightarrow X(p'+p''-r,p'+p''),
\]
and the pairing on $E^{r+1}$ is induced by that on $E^r$.
\end{theorem}
\begin{proof}
From the pairing $X'\otimes X''\rightarrow X$, we obtain solid cubes 
\begin{center}\begin{tikzcd}[column sep=0mm]
X'(p'-r)\otimes X''(p''-r)\ar[rr]\ar[dr]\ar[dd]&&X'(p')\otimes X''(p''-r)\ar[dd]\ar[dr]\\
&X(p'+p''-2r)\ar[rr]\ar[dd]&&X(p'+p''-r)\ar[dd]\\
X'(p'-r)\otimes X''(p'')\ar[rr]\ar[dr]&&X'(p')\otimes X''(p'')\ar[dr]\\
&X(p'+p''-r)\ar[rr]&&X(p'+p'')
\end{tikzcd},\end{center}
giving rise to pairings
\[
X'(p'-r,p')\otimes X''(p''-r,p'')\rightarrow X(p'+p''-r,p'+p'')
\]
fitting into commutative diagrams
\begin{center}\begin{tikzcd}[column sep=large]
X'(p'-r,p')\otimes X''(p''-r,p'')\ar[d]\ar[r,"\mu"]& X(p'+p''-r,p'+p'')\ar[d]\\
\Sigma\left(\begin{array}{c}X'(p'-2r,p'-r)\otimes X''(p''-r,p'')\\ \oplus\\ X'(p'-r,p')\otimes X''(p''-2r,p''-r)\end{array}\right)\ar[r,"\Sigma(\mu+\mu)"]& \Sigma X(p'+p''-2r,p'+p''-r)
\end{tikzcd}.\end{center}
As $\pi_\ast X(p-1,p) = E^1_p(X)$, the pairings obtained for $r=1$ give $E^1_{p'}(X')\olotimes E^1_{p''}(X'')\rightarrow E^1_{p'+p''}(X)$. The above square, together with \cref{lem:leib} identifying the left vertical composite and our conventions regarding K\"unneth maps, implies this satisfies the Leibniz rule, and hence passes to a pairing on $E^r$ for all $r\geq 1$. By construction there are canonically commutative diagrams
\begin{center}\begin{tikzcd}
X'(p'-r,p')\otimes X''(p''-r,p'')\ar[r]\ar[d]&X(p'+p''-r,p'+p'')\ar[d]\\
X'(p'-1,p')\otimes X''(p''-1,p'')\ar[r]&X(p'+p''-1,p'+p'')
\end{tikzcd},\end{center}
and these tell us that the pairing on $E^r$ is induced from the pairing in $\calC$.
\end{proof}

We end with a remark concerning convergence of this product. Fix a pairing of towers $X'\otimes X''\rightarrow X$. Suppose that $\calC$ and $\calA$ admit countable sums, and that these distribute across $\otimes$ and $\olotimes$. In particular, we obtain a pairing $X'(\infty)\otimes X''(\infty)\rightarrow X(\infty)$. Under the convergence conditions of \cref{prop:convergess}, the pairing $E(X')\olotimes E(X'')\rightarrow E(X)$ of \cref{thm:pairingss} passes to $E^\infty_{p'}(X')\olotimes E^\infty_{p''}(X'')\rightarrow E^\infty_{p'+p''}(X)$. As there are canonically commutative diagrams
\begin{center}\begin{tikzcd}
X'(p')\otimes X''(p'')\ar[r]\ar[d]&X(p'+p'')\ar[d]\\
X'(\infty)\otimes X''(\infty)\ar[r]&X(\infty)
\end{tikzcd},\end{center}
this is the associated graded of the pairing $\pi_\ast X'(\infty)\olotimes \pi_\ast X''(\infty)\rightarrow\pi_\ast X(\infty)$.

\end{appendix}

\begingroup
\raggedright
\bibliographystyle{alpha}
\bibliography{refs} % Put filename (without the .bib) here.

\begin{thebibliography}{GIKR22}

\bibitem[ARV11]{adamekrosickyvitale2011algebraic}
J.~Ad{\'a}mek, J.~Rosick{\'y}, and E.~M. Vitale.
\newblock {\em Algebraic theories. {A} categorical introduction to general
  algebra. {With} a foreword by {F}. {W}. {Lawvere}}, volume 184 of {\em Camb.
  Tracts Math.}
\newblock Cambridge: Cambridge University Press, 2011.

\bibitem[Bad02]{badzioch2002algebraic}
Bernard Badzioch.
\newblock Algebraic theories in homotopy theory.
\newblock {\em Ann. of Math. (2)}, 155(3):895--913, 2002.

\bibitem[Bal23]{balderrama2021algebraic}
William Balderrama.
\newblock Algebraic theories of power operations.
\newblock {\em J. Topol.}, 16(4):1543--1640, 2023.

\bibitem[BDG04]{blancdwyergoerss2004realization}
D.~Blanc, W.~G. Dwyer, and P.~G. Goerss.
\newblock The realization space of a {$\Pi$}-algebra: a moduli problem in
  algebraic topology.
\newblock {\em Topology}, 43(4):857--892, 2004.

\bibitem[Ber06]{bergner2006rigidification}
Julia~E. Bergner.
\newblock Rigidification of algebras over multi-sorted theories.
\newblock {\em Algebr. Geom. Topol.}, 6:1925--1955, 2006.

\bibitem[BF78]{bousfieldfriedlander1978homotopy}
A.~K. Bousfield and E.~M. Friedlander.
\newblock Homotopy theory of {$\Gamma $}-spaces, spectra, and bisimplicial
  sets.
\newblock In {\em Geometric applications of homotopy theory ({P}roc. {C}onf.,
  {E}vanston, {I}ll., 1977), {II}}, volume 658 of {\em Lecture Notes in Math.},
  pages 80--130. Springer, Berlin, 1978.

\bibitem[BHS23]{burklundhahnsenger2019boundaries}
Robert Burklund, Jeremy Hahn, and Andrew Senger.
\newblock On the boundaries of highly connected, almost closed manifolds.
\newblock {\em Acta Math.}, 231(2):205--344, 2023.

\bibitem[Bra17]{brantner2017lubin}
Lukas Brantner.
\newblock {\em {The Lubin-Tate Theory of Spectral Lie Algebras}}.
\newblock Ph.{D}. thesis, Harvard University, 2017.
\newblock urn-3:HUL.InstRepos:41140243.

\bibitem[Cis19]{cisinski2019higher}
Denis-Charles Cisinski.
\newblock {\em Higher categories and homotopical algebra}, volume 180 of {\em
  Cambridge Studies in Advanced Mathematics}.
\newblock Cambridge University Press, Cambridge, 2019.

\bibitem[DKS95]{dwyerkanstover1995bigraded}
W.~G. Dwyer, D.~M. Kan, and C.~R. Stover.
\newblock The bigraded homotopy groups {$\pi_{i,j}X$} of a pointed simplicial
  space {$X$}.
\newblock {\em J. Pure Appl. Algebra}, 103(2):167--188, 1995.

\bibitem[DP61]{doldpuppe1961homologie}
Albrecht Dold and Dieter Puppe.
\newblock Homologie nicht-additiver {F}unktoren. {A}nwendungen.
\newblock {\em Ann. Inst. Fourier Grenoble}, 11:201--312, 1961.

\bibitem[GH05]{goersshopkinsxxxxmoduli}
Paul Goerss and Michael Hopkins.
\newblock Moduli problems for structured ring spectra.
\newblock \url{http://www.math.northwestern.edu/~pgoerss/spectra/obstruct.pdf},
  2005.

\bibitem[GIKR22]{gheorgheisaksenkrausericka2018cmotivic}
Bogdan Gheorghe, Daniel~C. Isaksen, Achim Krause, and Nicolas Ricka.
\newblock {$\mathbb{C}$}-motivic modular forms.
\newblock {\em J. Eur. Math. Soc. (JEMS)}, 24(10):3597--3628, 2022.

\bibitem[Gla16]{glasman2016day}
Saul Glasman.
\newblock Day convolution for {$\infty$}-categories.
\newblock {\em Math. Res. Lett.}, 23(5):1369--1385, 2016.

\bibitem[GM92]{greenleesmay1992derived}
J.~P.~C. Greenlees and J.~P. May.
\newblock Derived functors of {$I$}-adic completion and local homology.
\newblock {\em J. Algebra}, 149(2):438--453, 1992.

\bibitem[GM95]{greenleesmay1995completions}
J.~P.~C. Greenlees and J.~P. May.
\newblock Completions in algebra and topology.
\newblock In {\em Handbook of algebraic topology}, pages 255--276.
  North-Holland, Amsterdam, 1995.

\bibitem[HL17]{hopkinslurie2017brauer}
Michael Hopkins and Jacob Lurie.
\newblock On {B}rauer groups of {L}ubin-{T}ate spectra {I}.
\newblock \url{https://www.math.ias.edu/~lurie/papers/Brauer.pdf}, 2017.

\bibitem[HS99]{hoveystrickland1999morava}
Mark Hovey and Neil~P. Strickland.
\newblock Morava {$K$}-theories and localisation.
\newblock {\em Mem. Amer. Math. Soc.}, 139(666):viii+100, 1999.

\bibitem[Lam55]{lambek1955groups}
J.~Lambek.
\newblock Groups and herds.
\newblock {\em Bull. Amer. Math. Soc.}, 61:78, 1955.

\bibitem[Lam92]{lambek1992ubiquity}
J.~Lambek.
\newblock On the ubiquity of {M}al{'}cev operations.
\newblock In {\em Proceedings of the {I}nternational {C}onference on {A}lgebra,
  {P}art 3 ({N}ovosibirsk, 1989)}, volume 131 of {\em Contemp. Math.}, pages
  135--146. Amer. Math. Soc., Providence, RI, 1992.

\bibitem[Law04]{lawvere1963functorial}
F.~William Lawvere.
\newblock Functorial semantics of algebraic theories and some algebraic
  problems in the context of functorial semantics of algebraic theories.
\newblock {\em Repr. Theory Appl. Categ.}, pages 1--121, 2004.

\bibitem[Law19]{lawsonxxxxen}
Tyler Lawson.
\newblock {$E_n$}-spectra and {D}yer-{L}ashof operations.
\newblock In {\em {H}andbook of {H}omotopy {T}heory}. CRC Press, 2019.

\bibitem[LM06]{lewsismandell2006equivariant}
L.~Gaunce Lewis, Jr. and Michael~A. Mandell.
\newblock Equivariant universal coefficient and {K}\"{u}nneth spectral
  sequences.
\newblock {\em Proc. London Math. Soc. (3)}, 92(2):505--544, 2006.

\bibitem[Lur11]{lurie2011quasi}
Jacob Lurie.
\newblock Quasi-coherent sheaves and {T}annaka duality theorems.
\newblock \url{https://www.math.ias.edu/~lurie/papers/DAG-VIII.pdf}, 2011.

\bibitem[Lur17a]{lurie2017higheralgebra}
Jacob Lurie.
\newblock Higher algebra.
\newblock \url{http://www.math.ias.edu/~lurie/papers/HA.pdf}, 2017.

\bibitem[Lur17b]{lurie2017highertopos}
Jacob Lurie.
\newblock Higher topos theory.
\newblock \url{http://www.math.ias.edu/~lurie/papers/HTT.pdf}, 2017.

\bibitem[Lur18]{lurie2018spectral}
Jacob Lurie.
\newblock Spectral algebraic geometry.
\newblock \url{https://www.math.ias.edu/~lurie/papers/SAG-rootfile.pdf}, 2018.

\bibitem[MG19]{mazelgee2019grothendieck}
Aaron Mazel-Gee.
\newblock On the {Grothendieck} construction for {{\(\infty\)}}-categories.
\newblock {\em J. Pure Appl. Algebra}, 223(11):4602--4651, 2019.

\bibitem[Pst23a]{pstragowski2023moduli}
Piotr Pstrągowski.
\newblock Moduli of spaces with prescribed homotopy groups.
\newblock {\em Journal of Pure and Applied Algebra}, 227(10):107409, 2023.

\bibitem[Pst23b]{pstragowski2023synthetic}
Piotr Pstrągowski.
\newblock Synthetic spectra and the cellular motivic category.
\newblock {\em Invent. Math.}, 232(2):553--681, 2023.

\bibitem[Qui67]{quillen1967homotopical}
Daniel~G. Quillen.
\newblock {\em Homotopical algebra}.
\newblock Lecture Notes in Mathematics, No. 43. Springer-Verlag, Berlin-New
  York, 1967.

\bibitem[Rez06]{rezk2006lectures}
Charles Rezk.
\newblock Lectures on power operations.
\newblock \url{https://rezk.web.illinois.edu/power-operation-lectures.dvi},
  2006.

\bibitem[Rez09]{rezk2009congruence}
Charles Rezk.
\newblock The congruence criterion for power operations in {M}orava
  {$E$}-theory.
\newblock {\em Homology Homotopy Appl.}, 11(2):327--379, 2009.

\bibitem[Rez18]{rezk2018analytic}
Charles Rezk.
\newblock Analytic completion.
\newblock \url{https://rezk.web.illinois.edu/analytic-paper.pdf}, 2018.

\bibitem[Smi76]{smith1976malcev}
Jonathan D.~H. Smith.
\newblock {\em Mal'cev varieties}.
\newblock Lecture Notes in Mathematics, Vol. 554. Springer-Verlag, Berlin-New
  York, 1976.

\bibitem[SS03]{schwedeshipler2003stable}
Stefan Schwede and Brooke Shipley.
\newblock Stable model categories are categories of modules.
\newblock {\em Topology}, 42(1):103--153, 2003.

\bibitem[{Til}16]{tilson2016power}
Sean {Tilson}.
\newblock {Power operations in the K\"unneth spectral sequence and commutative
  {$H\mathbb{F}_p$}-algebras}.
\newblock {\em arXiv e-prints}, page arXiv:1602.06736, February 2016.

\bibitem[TV69]{tierneyvogel1969simplicial}
Myles Tierney and Wolfgang Vogel.
\newblock Simplicial resolutions and derived functors.
\newblock {\em Math. Z.}, 111:1--14, 1969.

\bibitem[Wra70]{wraith1969algebraic}
G.~C. Wraith.
\newblock {\em Algebraic theories}.
\newblock Lectures Autumn 1969. Lecture Notes Series, No. 22. Matematisk
  Institut, Aarhus Universitet, Aarhus, 1970.

\end{thebibliography}
\endgroup

\end{document}